\documentclass[reqno]{amsart}
\usepackage[utf8]{inputenc}
\usepackage{amsmath}
\usepackage{amsthm}
\usepackage{amssymb}
\usepackage{mathrsfs}
\usepackage{todonotes}
\usepackage[english]{babel}

\DeclareMathOperator{\supp}{supp}
\DeclareMathOperator{\opt}{opt}
\DeclareMathOperator{\Hess}{Hess}

\theoremstyle{plain}
\newtheorem{thm}{Theorem}[section]
\newtheorem{lem}[thm]{Lemma}
\newtheorem{prop}[thm]{Proposition}
\newtheorem{cor}[thm]{Corollary}
\newtheorem{notation}[thm]{Notation}

\theoremstyle{definition}
\newtheorem{Def}[thm]{Definition}

\theoremstyle{remark}
\newtheorem{exmp}[thm]{Example}
\newtheorem{rem}[thm]{Remark}

\usepackage[bottom=3cm, top=3cm, left=3.5cm, right=3.5cm]{geometry}

\usepackage{xcolor}
\usepackage[pdftex]{hyperref}

\usepackage{bbold}
\usepackage{upgreek}
\usepackage{mathtools}

\usepackage[shortlabels]{enumitem}

\usepackage{autonum} 

\title[LQ optimal transport and interpolation inequalities]{Linear quadratic optimal transport and interpolation inequalities}

\author{Luca Rizzi}
\address{Luca Rizzi: SISSA, via Bonomea, 265 - 34136 Trieste - Italy}
\email{\href{mailto:lrizzi@sissa.it}{lrizzi@sissa.it}}

\author{Alec J. A. Schiavoni Piazza}
\address{Alec J. A. Schiavoni Piazza: SISSA, via Bonomea, 265 - 34136 Trieste - Italy}
\email{\href{mailto:aschiavo@sissa.it}{aschiavo@sissa.it}}

\thanks{This project has received funding from the European Research Council (ERC) under the European Union's Horizon 2020 research and innovation programme (grant agreement GEOSUB, No. 945655). The authors also acknowledge the INdAM support.}
\date{\today}

\newcommand{\N}{\mathbb N}

\newcommand{\Adm}{{\rm Adm}}

\newcommand{\R}{\mathbb{R}}

\begin{document}
\begin{abstract}
This paper investigates the optimal transport problem within the framework of Linear Quadratic optimal control systems. We establish the well-posedness of the Monge problem and analyze the regularity of the resulting optimal transport map, extending the results obtained in \cite{hinpomriff2011} for non-negative costs. Furthermore, we study the displacement interpolation of measures and derive general interpolation inequalities for entropy functionals. Our analysis is motivated by the role of these systems as natural model spaces for comparison theory in sub-Riemannian geometry.
\end{abstract}

\maketitle
\tableofcontents

\section{Introduction}
		The theory of optimal transport, introduced by Monge in the $18^{th}$ century \cite{monge1781}, in recent decades found new links with differential geometry and dynamical systems. One of the most relevant ones is the Lott-Sturm-Villani synthetic formulation of Ricci curvature lower bounds for metric measure spaces \cite{lott2006riccicurvmmspaces,SturmGeomMmsI,SturmGeomMmsII}, realized via entropy interpolation inequalities along Wasserstein geodesics. Recently, Barilari-Rizzi-Mondino \cite{Barilari_2016,SRintineq,BakriEmeryandmodelspaces,UnifiedsyntheticRicci} generalized these to the sub-Riemannian setting, building on previous work for Heisenberg and corank $1$ Carnot groups \cite{balogh2018geometricinequalitiesheisenberggroups,balogh2019jacobiandeterminantinequalitycorank}. In this context, Linear Quadratic (LQ) optimal control problems \cite{agrachev2013control, Jurdjevic1996GeometricCT} emerged as the natural ``constant curvature'' models.
		
		The study of optimal transport in this framework was initiated in \cite{hinpomriff2011}, specifically for $-Q \geq 0$, corresponding to ``negative curvature" and costs bounded from below. The purpose of this work is to provide a comprehensive study of the problem: we establish the well-posedness of the Monge problem for general LQ costs, analyze displacement interpolations, and derive entropy inequalities, aiming to provide a reference for future research.
		
	\subsection{LQ problems}
	
	Let  $n,m\in \mathbb{N}$, $m\leq n$, and let $A,Q\in \mathbb{R}^{n \times n}$ and $B\in\mathbb{R}^{n \times m}$ be matrices, with $Q=Q^*$. The LQ optimal control problem $(A,B,Q,T)$ consists in minimizing the functional 
	\begin{equation}\label{eq:introLQcost}
		\mathcal{J}(u)=\int_{0}^{T}\left(|u|^2-x^*Qx\right)dt
	\end{equation}
	among all trajectories $x:[0,T]\to\mathbb{R}^n$ with fixed endpoints satisfying
	\begin{equation}
		\dot{x}=Ax+Bu,
	\end{equation}
	for some $u\in L^2([0,T],\mathbb{R}^m)$. Optimal trajectories are projections on $\mathbb{R}^n$ of integral curves of the quadratic Hamiltonian $\mathcal{H}:T^*\R^n (\simeq \mathbb{R}^{2n})\to \mathbb{R}$, given by 
	\begin{equation}\label{eq_ham_LQ_intro}
		\mathcal{H}(p,x)=\frac{1}{2}\left(p^*BB^*p+2p^*Ax+x^*Qx\right).
	\end{equation}
    A time $t>0$ is called a conjugate time if there exists an integral curve of \eqref{eq_ham_LQ_intro} whose projection $x(\cdot)$ onto $\mathbb{R}^n$ satisfies $x(0)=x(t)=0$. LQ problems exhibit a dichotomy: they either possess no conjugate times or an infinite, discrete set of them that does not accumulate at $t=0$  \cite{Agrachev_2014}. Let $t^*>0$ denote the first conjugate time. We observe that if $-Q \geq 0$ (the case treated in \cite{hinpomriff2011}), then $t^* = +\infty$ and the cost functional \eqref{eq:introLQcost} is non-negative.
    
    If $T<t^*$, for any pair of endpoints $x,y\in \R^n$, there is a unique optimal trajectory with these endpoints minimizing \eqref{eq:introLQcost}. Such minimum, called LQ cost function, is an homogeneous polynomial of degree $2$,
    of the form
	\begin{equation}
		c(x,y)=\frac{1}{2}\left(x^*Cx+2x^*Dy+y^*Ey\right),\qquad\forall x,y\in\mathbb{R}^n,
	\end{equation} 
	with $C,D,E\in\mathbb{R}^{n\times n}$ and $C,E$ symmetric. The matrices $C, D$ and $E$ are explicitly computable in terms of the Hamiltonian flow of the system (and of the backwards one), and depend smoothly on $T\in(0,t^*)$. Note that, unless $-Q\geq 0$, the LQ cost function is not bounded from below.

	\subsection{Optimal transport}
	Given a cost function $c:\mathcal{X}\times\mathcal{Y}\to\mathbb{R}$ between Polish spaces and probability measures $\mu\in P(\mathcal{X})$, $\nu\in P(\mathcal{Y})$, the Monge optimal transport problem \cite{monge1781} corresponds to minimizing the quantity
	\begin{equation}
		\int_{\mathcal{X}}c(x,T(x))d\mu(x),
	\end{equation}
	among all transport maps, i.e.\ measurable maps $T:\mathcal{X}\to\mathcal{Y}$, that satisfy $T_{\sharp}\mu=\nu$. A relaxed version of it, the Kantorovich formulation \cite{Kantorovich2006}, is the problem of minimizing 
	\begin{equation}
		\int_{\mathcal{X}\times\mathcal{Y}}c(x,y)d\alpha(x,y),
	\end{equation}
	among all transport plans, i.e.\ probability measures $\alpha\in P(\mathcal{X}\times\mathcal{Y})$, such that $(\pi_{\mathcal{X}})_\sharp\alpha=\mu$ and $(\pi_{\mathcal{Y}})_\sharp\alpha=\nu$. A fundational result in the field was the one of Brenier \cite{Brenier1991}, for the case of the Euclidean squared-distance cost, proving existence
	and uniquness of solutions to the Monge problem, for $\mu \ll \mathscr{L}^n$. Since then, many authors contributed to the study of existence and regularity of optimal plans and maps (see \cite{Villanioldandnew,Usersguide} and references therein). Under suitable assumptions on the measures, existence and uniqueness of optimal transport maps is known for:
	\begin{itemize}
		\item compact Riemannian manifolds without boundary and squared-distance cost \cite{McCann2001},
		\item infinite-dimensional Hilbert spaces and squared-distance cost \cite{ambrosioGradientFlowsMetric2008},
		\item Heisenberg groups \cite{AMBROSIO2004261} and a class of sub-Riemannian manifolds \cite{figalli2009masstransSRmflds}, with cost equal to the Carnot-Carathéodory squared-distance,
		\item cost functions induced by Tonelli Lagrangians on compact manifolds \cite{bernard2007optimalmasstransportationmather}, and the generalization to non-compact manifolds \cite{FathiFigalli},
		\item cost functions induced by sub-Riemannian-type Lagrangians bounded from below \cite{agrachev2007opttransnonholconst},
		\item metric measure spaces that satisfy suitable MCP-like and non-branching assumptions, with cost equal to a convex non-decreasing function of the distance \cite{cavalletti2014existenceuniquenessoptimaltransport,Cavalletti_2017},
        \item cost functions of LQ optimal control problems with $-Q\geq 0$ \cite{hinpomriff2011}.
	\end{itemize} 

We first extend the results in \cite{hinpomriff2011} to any LQ problem, and to a broader class of measures, possibly with non-compact support, resembling the sharp version of Brenier theorem in \cite{GigliInverseBrenier2009}. In addition, we characterize optimal maps in terms of the LQ exponential map applied to the gradient of a semi-convex function. The next result corresponds to Theorem \ref{Brenier_McCann}.
	
	\begin{thm}[Well-posedness]\label{Brenier_McCann_intro}  Let $(A,B,Q,T)$ be a LQ optimal control problem satisfying the Kalman condition, with $T<t^*$. Let $\mu,\nu\in P_2(\mathbb{R}^n)$, and assume that $\mu$ does not give mass to any $c-c$-hypersurface. Then, there exists a unique optimal transport plan $\alpha\in P_2(\mathbb{R}^n\times\mathbb{R}^n)$ between $\mu$ and $\nu$ for the LQ optimal transport problem, induced by a map $T:\mathbb{R}^n\to\mathbb{R}^n$.
        In addition, any Kantorovich potential $\psi$ associated with $\alpha$ is $\mu$-almost everywhere differentiable and the optimal transport map $T$ satisfies
        \begin{equation}
            T(x)=\exp_{x,T}(\nabla\psi(x)),\qquad\qquad \mu\text{-a.e. }x\in\mathbb{R}^n,
        \end{equation}
        where $\exp_{x,\tau}:T_x\mathbb{R}^n(\cong\mathbb{R}^n)\to\mathbb{R}^n$, for $\tau\in[0,T]$, denotes the LQ exponential map.

        If $\mu\ll\mathscr{L}^n$, then $\psi$ is $\mu$-almost everywhere twice differentiable. In particular, $T$ is $\mu$-almost everywhere differentiable.
	\end{thm}
    
    \begin{rem}
 Due to the quadratic nature of the LQ cost function, it is natural to choose $\mu,\nu \in P_2(\R^n)$, the Euclidean $2$-Wasserstein space.
    \end{rem}
\begin{rem}    
	Since $c-c$-hypersurfaces are $(n-1)$-rectifiable, Theorem \ref{Brenier_McCann_intro} applies to all $\mu\in P_2(\mathbb{R}^n)$ that do not give mass to $(n-1)$-rectifiable sets. 
    \end{rem}
	
    Following \cite{hinpomriff2011}, and relying on the regularity theory for the Euclidean Monge-Ampère equation \cite{CaffarelliRegConvMaps}, we obtain the following result, which corresponds to Theorem \ref{continuity_transport_map}.

	\begin{thm}[Continuity of the transport map]
		 Consider two compactly supported measures $\mu, \nu\in P_2(\mathbb{R}^n)$, with $\mu, \nu\ll\mathscr{L}^n$. Assume, in addition, that $\mu=f\mathscr{L}^n$, $\nu=g\mathscr{L}^n$ with $f$ and $g$ bounded from above and below on $\supp(\mu)$ and $\supp(\nu)$, respectively, and $\supp(\mu)$ has connected interior, whilst $\supp(\nu)$ is convex. Then, the unique optimal transport map $T:\mathbb{R}^n\to\mathbb{R}^n$ between $\mu$ and $\nu$ is $\beta$-Hölder continuous, for some $\beta\in(0,1)$.
	\end{thm}

    We refer to Remark \ref{rmk:highereg} for higher-order regularity results.

    \subsection{Displacement interpolations}
    
	Another classical topic of interest is the dynamical version of optimal transport. To illustrate the problem, consider a Riemannian manifold $(M,d_g)$ with volume measure $m_g$, and let $\mu,\nu\in P_2(M)$ be compactly supported measures, with $\mu\ll m_g$. Then, by McCann's generalization of Brenier's Theorem \cite{McCann2001}, there exists a unique optimal transport map $T:M\to M$ such that $T_{\sharp}\mu=\nu$, solving the Monge problem:
	\begin{equation}
		\int_{M}d_g^2(x,T(x))d\mu(x)=\inf_{S_{\sharp}\mu=\nu}\int_{M}d_g^2(x,S(x))d\mu.
	\end{equation}
	Moreover, for $\mu$-almost every $x\in M$, there exists a unique constant-speed geodesic $[0,1]\ni \tau\mapsto T_{\tau}(x)$ such that $T_0(x)=x$ and $T_1(x)=T(x)$. The map $T_{\tau}:M\mapsto M$ induces the ``displacement interpolation" $\mu_{\tau}=T_{\tau\sharp}\mu$, a curve of measures between $\mu$ and $\nu$. More precisely, $(\mu_{\tau})_{0\leq \tau\leq 1}$ is the unique geodesic between $\mu$ and $\nu$ in the Wasserstein space $(P_2(M),W_2)$. Moreover, $\mu_\tau\ll m_g$, for every $\tau\in [0,1)$. These results were established in a number of settings, including:
	\begin{itemize}
    \item Heisenberg groups \cite{AMBROSIO2004261,FiJu2008} and ideal sub-Riemannian manifolds \cite{figalli2009masstransSRmflds},
	\item coercive Lagrangians on Polish spaces \cite[Chapter 7 and references therein]{Villanioldandnew}.
	\end{itemize}
	Concerning LQ problems, their cost function may not be bounded from below, which is a standard hypothesis in the literature. Furthermore, the Lagrangian is not Tonelli-type nor  the corresponding cost functional \eqref{eq:introLQcost} coercive. We fill this gap, establishing a theory of displacement interpolations for LQ optimal transport. The next result corresponds to Theorem \ref{dyn_opt_coup_characterization}.
	\begin{thm}[Displacement interpolations]\label{thm:displacementinterpolation_intro}
		Let $(A,B,Q,T)$ be a LQ optimal control problem satisfying the Kalman condition, with $T<t^*$. Let $\mu, \nu\in P_2(\mathbb{R}^n)$. There is a bijection between:
        \begin{itemize}
            \item optimal plans $\alpha\in P_2(\mathbb{R}^n\times\mathbb{R}^n)$,
            \item optimal dynamical plans $\Pi\in P(\mathcal{C}([0,T],\mathbb{R}^n))$,
            \item displacement interpolations $(\mu_\tau)\in \mathcal{C}([0,T],P_2(\mathbb{R}^n))$.
        \end{itemize}
        In particular, for every optimal dynamical plan $\Pi\in P(\mathcal{C}([0,T],\mathbb{R}^n))$, the corresponding displacement interpolation is given by
		\begin{equation}
			e_{\tau\sharp}\Pi=\mu_\tau,\qquad \forall \tau\in[0,T].
		\end{equation}
		Moreover, for $0\leq t<s\leq T$, we define 
		\begin{equation}
			\mathbb{A}^{t,s}(\eta_{\tau}):=\sup\limits_{\{N\in\mathbb{N}\}}\sup\limits_{\{t=\tau_0<\dots<\tau_N=s\}}\sum_{i=1}^{N}C^{\tau_{i-1},\tau_i}\left(\eta_{\tau_{i-1}},\eta_{\tau_i}\right),\qquad(\eta_{\tau})\in\mathcal{C}([t,s],P_2(\mathbb{R}^n)),
		\end{equation} 
          where $C^{t,s}(\mu,\nu)$ is the Kantorovich transport cost, associated with the pointwise intermediate cost $c^{t,s}:\R^n\times \R^n \to \R$.
		Then, the family of functionals $\{\mathbb{A}^{t,s}\}_{0\leq t<s\leq T}$ on $P_2(\mathbb{R}^n)$ is a Lagrangian action. The corresponding family of action-minimizing costs coincides with the Kantorovich transport costs $\{C^{t,s}\}_{0\leq t<s\leq T}$. Furthermore, $(\mu_\tau)$ is a displacement interpolation between $\mu$ and $\nu$ if and only if it is an action minimizer.
	\end{thm}
    
	If $\mu\ll\mathscr{L}^n$, the displacement interpolation inherits the absolute continuity at intermediate times. The next result corresponds to Theorem \ref{absolute_continuity-dyn_opt_coupl}.
	
	\begin{thm}[Absolute continuity]
  Let $(A,B,Q,T)$ be a LQ optimal control problem satisfying the Kalman condition, with $T<t^*$. Let $\mu, \nu\in P_2(\mathbb{R}^n)$, with $\mu\ll\mathscr{L}^n$, and let $(\mu_\tau)=(e_{\tau\sharp}\Pi)\in\mathcal{C}([0,T],P_2(\mathbb{R}^n))$ be the unique displacement interpolation joining them. Then, $\mu_\tau\ll\mathscr{L}^n$, for every $\tau\in[0,T)$. 

        Additionally, let $\psi$ be a Kantorovich potential from $\mu$ to $\nu$ for the cost $c^{0,T}$. Then, the optimal transport map $T_\tau:\mathbb{R}^n\to\mathbb{R}^n$, from $\mu$ to $\mu_\tau$ for the cost $c^{0,\tau}$ is given by
        \begin{equation}
			T_\tau(x)=\exp_{x,\tau}(\nabla\psi(x)),\qquad\mu\text{-almost every }x\in\mathbb{R}^n, \quad\forall \tau\in[0,T].
		\end{equation}
        In particular, $T_\tau$ is $\mu$-almost everywhere differentiable.
	\end{thm}
    
    \subsection{Interpolation inequalities}
    
	Interpolation inequalities connect optimal transport to the study of functional inequalities and geometric analysis. In the Euclidean space $\mathbb{R}^n$, typical examples are the Borell-Brascamp-Lieb inequality, and its geometric counterpart, the Brunn-Minkowski inequality. A geodesic version of these inequalities has been proven in the Riemannian setting \cite{ARiemintineq,McCann1997ACP}, provided that the geometry is taken into account by means of appropriate distortion coefficients. To illustrate these results, let $(M,d_g)$ be a Riemannian manifold of dimension $n$ and let $Z_\tau(x,y)$ be the set of $\tau$-intermediate points of geodesics between $x$ and $y$ in $M$. Then, we can define the distortion coefficients as
	\begin{equation}
		\beta_\tau(x,y):=\limsup\limits_{r\to 0}\frac{m_g(Z_\tau(x,B_r(y)))}{m_g(B_r(y))},
	\end{equation}
	where $B_r(y)$ is the Riemannian ball of center $y$ and radius $r$. If $\mu,\nu\in P_2(M)$ are compactly supported measures, with $\mu,\nu\ll m_g$, and $(\mu_\tau)=(\rho_\tau m_g)$ is the unique displacement interpolation joining $\mu=\mu_0$ and $\nu=\mu_1$, then we can estimate the density $1/\rho_\tau$ in terms of the initial and final densities. More precisely, if $\mu_\tau=T_{\tau\sharp}\mu$ for $0\leq \tau \leq 1$, then 
	\begin{equation}\label{interpolation_intro}
		\frac{1}{\rho_\tau(T_\tau(x))^{1/n}}\geq\frac{\beta_{1-\tau}(T(x),x)^{1/n}}{\rho_0(x)^{1/n}}+\frac{\beta_{\tau}(x,T(x))^{1/n}}{\rho_1(T(x))^{1/n}},
	\end{equation}
	for $\mu$-almost every $x\in M$.
	
	If $\mathrm{Ric}_g(M)\geq k\cdot g$, for some $k\in\mathbb{R}$, then the distortion coefficients $\beta_t$ are controlled from below by those of model spaces of constant Ricci curvature equal to $k$ and dimension $n$. Explicitly,
		\begin{equation}\label{reference_coefficients}
		\beta_{\tau}(x,y)\geq \beta^{k,n}_{\tau}(x,y):=\begin{cases}
			\tau\left(\frac{\sin(\tau\theta)}{\sin(\theta)}\right)^{n-1}&\text{ for }k>0,\\
			\tau^n&\text{ for }k=0,\\
			\tau\left(\frac{\sinh(\tau\theta)}{\sinh(\theta)}\right)^{n-1}&\text{ for }k<0,
		\end{cases}
	\end{equation}
	where
	\begin{equation}
		\theta=\sqrt{\frac{|k|}{n-1}}d_g(x,y).
	\end{equation}

    Inequality \eqref{interpolation_intro} is lifted along displacement interpolations in Wasserstein spaces \cite{ARiemintineq}, yielding the so-called curvature-dimension conditions. This approach allows for the extension of Ricci curvature bounds to metric measure spaces \cite{SturmGeomMmsI,SturmGeomMmsII,lott2006riccicurvmmspaces}. Such inequalities, and their equivalence to suitably defined lower Ricci curvature bounds, were investigated in \cite{rotem} for costs induced by Tonelli Lagrangians with free end-time; in the Riemannian case, this corresponds to analyzing the $L^1$ optimal transport cost, rather than the standard $L^2$ formulation.

Recently, a series of works by Barilari, Rizzi, and Mondino \cite{Barilari_2016,BakriEmeryandmodelspaces,UnifiedsyntheticRicci} established analogous results for a class of sub-Riemannian structures. In these studies, LQ optimal control problems emerge as the canonical model spaces for sub-Riemannian comparison theory. This provides the primary motivation for our study of interpolation inequalities directly within the LQ framework. The following result corresponds to Theorem \ref{density_interpolation_inequality}.

	\begin{thm}[Interpolation inequalities]\label{5}
		Let $(A,B,Q,T)$ be a LQ problem satisfying the Kalman condition, with $T<t^*$. Consider $\mu,\nu\in P_2(\mathbb{R}^n)$, with $\mu, \nu \ll\mathscr{L}^n$, and let $(\mu_\tau)=(e_{\tau\sharp}\Pi)$ be the unique displacement interpolation between them. Let $T_\tau$ be the optimal transport map from $\mu$ to $\mu_\tau$, for the cost $c^{0,\tau}$. Then, for $\mu$-almost every $x\in\mathbb{R}^n$ and every $\tau\in[0,T]$, we have
			\begin{equation}\label{density_interpolation_inequality_formula_intro}
				\frac{1}{\rho_\tau(T_\tau(x))^{1/n}}\geq\frac{\beta_{T-\tau}^{1/n}}{\rho_0(x)^{1/n}} +\frac{\beta_{\tau}^{1/n}}{\rho_T(T(x))^{1/n}},
			\end{equation}
			where $\beta_s=\beta_s^{(A,B,Q,T)}$ is the distortion coefficient of the LQ problem (see Definition \ref{def:distcoeff}) and $\mu_\tau=\rho_\tau\mathscr{L}^n$. If $\nu$ is not absolutely continuous, \eqref{density_interpolation_inequality_formula_intro} holds for $\tau\in[0,T)$ and omitting the second term in the right hand side.
	\end{thm}
    
    We finally prove interpolation inequalities for entropy functionals of displacement convexity class $\mathcal{DC}_n$ (see \cite{McCann1997ACP} or \cite[Chapter 17]{Villanioldandnew}). The next result corresponds to Theorem \ref{dist_displ_interp}.
	 \begin{thm}[Entropic inequalities]
 	Let $(A,B,Q,T)$ be a LQ optimal control problem satisfying the Kalman condition, with $T<t^*$.
	 	Consider $U\in\mathcal{DC}_n$ and $\mu,\nu\in P_2(\mathbb{R}^n)$, with $\mu,\nu\ll \mathscr{L}^n$. Let $(\mu_\tau)$ be the unique displacement interpolation between $\mu$ and $\nu$. Assume that $(\mu_\tau)\subset\mathrm{Dom}(\mathcal{U})$ for every $\tau\in[0,T]$. Then, the entropy functional $\mathcal{U}$ satisfies
			\begin{multline}
				\mathcal{U}(\mu_\tau)\leq\left(\frac{T}{T-\tau}\right)^{n-1}\beta_{T-\tau}\int_{\mathbb{R}^n}U\bigg(\frac{\rho_0(x)}{\beta_{T-\tau}}\left(\frac{T-\tau}{T}\right)^n\bigg)d\mathscr{L}^n(x)\\
				+\left(\frac{T}{\tau}\right)^{n-1}\beta_{\tau}\int_{\mathbb{R}^n}U\bigg(\frac{\rho_T(y)}{\beta_{\tau}}\left(\frac{\tau}{T}\right)^n\bigg)d\mathscr{L}^n(y),\qquad\forall \tau\in(0,T).
			\end{multline}
            where $\beta_s=\beta_s^{(A,B,Q,T)}$ is the distortion coefficient of the LQ problem and $\mu_\tau=\rho_\tau\mathscr{L}^n$.
            \end{thm}
            \begin{rem}
           We stress that, for every $U\in\mathcal{DC}_n$, the entropy functional $\mathcal{U}$ is well-defined on the whole $P_2^{ac}(\mathbb{R}^n)$ for $n\geq 3$. Otherwise, the condition $(\mu_\tau)\subset\mathrm{Dom}(\mathcal{U})$ might be hard to check apriori. For $n\in\{1,2\}$, such condition is still satisfied if the endpoints of the optimal curve (and thus the whole curve) are compactly supported. See the discussion in Section \ref{entr_ineq}.
       \end{rem}

To conclude, we show that the Riemannian model coefficients \eqref{reference_coefficients} can be obtained as distortion coefficients of suitable LQ problems. This means that one can mimick the effect of Riemannian curvature-dimension bounds using entirely LQ dynamics (rather than curved geometry itself as on Riemannian space forms, or weighted measures as on 1-dimensional model metric measure spaces \cite[Example 14.10]{Villanioldandnew}). The next statement corresponds to Example \ref{reproducing_CD_kn_Riem} in the paper.
 \begin{prop}
	 	For any $k\in\mathbb{R}$, any $n\in\mathbb{N}$ and any points $x,y$ in a Riemannian space form with constant Ricci curvature equal to $k$ and dimension $n$, with $d_g(x,y)<\pi\sqrt{(n-1)/k}$ for $k>0$, there exists a LQ problem $(A,B,Q,1)$ on $\mathbb{R}^n$, whose distortion coefficients satisfy
	 	\begin{equation}
	 		\beta_\tau^{(A,B,Q,1)}=\beta_\tau^{k,n}(x,y),\qquad\forall \tau\in[0,1],
	 	\end{equation}
	 	where $\beta_\tau^{k,n}(x,y)$ are model Riemannian distortion coefficients given in \eqref{reference_coefficients}. Explicitly,
	 	\begin{equation}
	 		A=\mathbb{0}_n,\qquad B=\mathbb{1}_n,\qquad Q=\begin{pmatrix}
	 			K\mathbb{1}_{n-1} & \mathbb{0}_{n-1,1}\\
	 			\mathbb{0}_{1,n-1} & \mathbb{0}_1
	 		\end{pmatrix},
	 	\end{equation}
	 	where $\mathbb{1}_i$ is a square identity matrix of dimension $i$, $\mathbb{0}_j$ is a square null matrix of dimension $j$, $\mathbb{0}_{i,j}$ is rectangular null matrix of dimension $i\times j$ and $K=d_g^2(x,y)\cdot k/(n-1)$.
	 \end{prop}

\section{Linear quadratic control problems}\label{LQ_control_problems}
\subsection{Problem statement}

    Let $n,m\in \N$, with $m\leq n$. Given $A\in \mathbb{R}^{n \times n}$ and $B\in\mathbb{R}^{n \times m}$. On $\mathbb{R}^n$, we consider the linear control system
		\begin{equation}\label{lq}
		\dot{x}=Ax+Bu.
		\end{equation}
	  The term $Ax$ represent a linear drift term, while the columns of $B$ are the so-called controllable directions.  Without loss of generality, we assume that $B$ has rank $m$.
      
      Denote by $\mathcal{C}(I,\mathbb{R}^n)$ be set of continuous curves from a bounded interval $I\subset \R$ to $\R^n$.
      \begin{Def}[Admissible trajectory] A curve $\gamma\in\mathcal{C}(I,\R^n)$ is an admissible trajectory for \eqref{lq} if it is absolutely continuous and there exists $u \in L^2\left(I,\R^m\right)$ (called control) such that
			\begin{equation}
				\dot{\gamma}(\tau)=A\gamma(\tau)+Bu(\tau),\qquad a.e.\ \tau \in I.
			\end{equation}
      \end{Def}
      
      \begin{rem}\label{bijection_trajectories_controls}
          Since $B$ has rank $m$, the control associated to an admissible trajectory is unique. Conversely, for every $x_0\in\mathbb{R}^n$ and every $u\in L^2\left(\left[t,s\right],\R^m\right)$, there is a unique admissible trajectory $\gamma_u(\cdot,x_0):\left[t,s\right]\to\mathbb{R}^n$, given by 
          \begin{equation}\label{Cauchy}
			\gamma_u(\tau,x_0)=e^{(\tau -t) A}\bigg(x_0+\int_{t}^{\tau}e^{(t-\xi) A}Bu(\xi)d\xi\bigg),\qquad \tau\in\left[t,s\right].
		\end{equation}
      \end{rem}

We say that the linear system \eqref{lq} is controllable if for any pair of points one can find an admissible trajectory joining them. The following characterization is well-known \cite{coron2007control,agrachev2013control}.
\begin{thm}[Kalman condition]\label{thm:kalman}
        The linear control system \eqref{lq} is controllable if and only if 
		\begin{equation}\label{Kalman}
			\mathrm{span}\left\{A^jb_i\mid j\in\left\{0,\dots,n-1\right\},i\in\left\{1,\dots,k\right\}\right\}=\mathbb{R}^n,
		\end{equation}
		where the $b_i$-s are the columns of the matrix $B$.
\end{thm}
\begin{Def}[LQ action and cost function]\label{cost}
            Consider the system \eqref{lq}, a symmetric matrix $Q\in\mathbb{R}^{n \times n}$ and an interval $I\subset\mathbb{R}^n$. The LQ action $\mathcal{A}^{I}:\mathcal{C}\left(I,\mathbb{R}^n\right)\to\mathbb{R}\cup\{+\infty\}$ is defined as 
            \begin{equation}
		\mathcal{A}^{I}(\gamma) :=
		\begin{cases}
			\displaystyle\frac{1}{2}\int_I\left(|u(\tau)|^2-\gamma^*(\tau)Q\gamma(\tau)\right)d\tau & \text{if $\gamma$ is admissible,} \\
			+\infty & \text{otherwise.}
		\end{cases}
		\end{equation}
        In addition, the LQ cost function $c^{I}:\mathbb{R}^n\times\mathbb{R}^n\to\mathbb{R}\cup\{-\infty\}$ is defined as
        \begin{equation}\label{LQ_infimization_problem}
			c^{I}(x,y):=\inf\{\mathcal{A}(\gamma)\mid \gamma\in \mathcal{C}\left(I,\mathbb{R}^n\right),\, \text{$\gamma$ joining $x$ with $y$}\},\qquad x,y\in\mathbb{R}^n.	
		\end{equation}
        Given $A,B,Q$ and an interval $I=[0,T]$ for $T>0$, the problem of computing \eqref{LQ_infimization_problem} is referred to as Linear Quadratic (LQ) optimal control problem, and denoted by $(A,B,Q,T)$.
	\end{Def}
        
        In the literature one often considers a LQ functional $\mathcal{J}$ acting on controls rather than the action functional $\mathcal{A}$ defined on continuous trajectories (see, e.g., \cite[Chapter 16]{agrachev2013control}). Both functionals agree on admissible trajectories. The functional $\mathcal{A}$ can be viewed as the lower semi-continuous relaxation of $\mathcal{J}$ with respect to uniform convergence (see Proposition \ref{properties_LQ_action}).
        
        \begin{notation}\label{LQ_intermediate_action_and_cost} Let $0\leq t < s \leq T$. In the following we denote the LQ action on the interval $[t,s]$ by $\mathcal{A}^{t,s}$, and the corresponding cost function by $c^{t,s}$. In the case $[t,s]=[0,T]$, when there is no risk of confusion, we omit the superscripts, namely $\mathcal{A}=\mathcal{A}^{0,T}$ and $c=c^{0,T}$.
        \end{notation}

    

\subsection{LQ Hamiltonian flow and conjugate times}
In this section, we recall some classical properties of optimal trajectories of a LQ problem, which can be found in \cite[Chapter 16]{agrachev2013control}. A necessary condition for optimality is given by the following result.
\begin{thm}\label{ham_LQ}
 Consider the system \eqref{lq} with  $A\in \mathbb{R}^{n \times n}$ and $B\in\mathbb{R}^{n \times m}$. Let $Q\in\mathbb{R}^{n\times n}$ be a symmetric matrix. Define, on $T^*\mathbb{R}^n\simeq \R^{2n}$, the corresponding LQ Hamiltonian 
	\begin{equation}\label{explicit_LQ_hamiltonian}
		\mathcal{H}(p,x) := \frac{1}{2}(p^*,x^*)\begin{pmatrix}
			BB^* & A\\
			A^* & Q
		\end{pmatrix}\begin{pmatrix}
			p\\
			x
		\end{pmatrix} = \frac{1}{2}(p^*,x^*)H\begin{pmatrix}
			p\\
			x
		\end{pmatrix}
	\end{equation}
	and its associated Hamiltonian vector field
	\begin{equation}\label{ident_ham_v_field}
	\overrightarrow{\mathcal{H}}_{(p,x)}:=-\Omega H\begin{pmatrix}
		p\\
		x
	\end{pmatrix}.
	\end{equation}
	Here, $\Omega$ is the standard simplectic form on $T^*\mathbb{R}^n$ in canonical Darboux coordinates, namely
\begin{equation}	
\Omega=\begin{pmatrix}
		\mathbb{0}_n & \mathbb{1}_n\\
		-\mathbb{1}_n & \mathbb{0}_n
	\end{pmatrix}.
\end{equation}
Then, any optimal trajectory $\gamma \in \mathcal{C}(I,\R^n)$ for the LQ action $\mathcal{A}$ is the projection of an integral curve $\tau\mapsto (p(\tau),x(\tau))\in T^*\R^n$ such that $x(\tau)=\gamma(\tau)$ and $u(\tau)=Bp(\tau)$ for all $\tau \in I$.
\end{thm}
Note that, in Theorem \ref{ham_LQ}, we canonically identified $T^*\mathbb{R}^n$ with $\mathbb{R}^{2n}$, so that the Hamiltonian vector field $\overrightarrow{\mathcal{H}}$ is itself identified with the linear vector field $-\Omega H$, as in equation \eqref{ident_ham_v_field}. 


\begin{Def}[Conjugate time]
  A time $t>0$ is called conjugate (to zero) if there exists a non-trivial integral curve $(p(\cdot),x(\cdot)):\left[0,t\right]\to T^*\mathbb{R}^n$ of the Hamiltonian flow \eqref{ident_ham_v_field}, i.e.\ solution of
		\begin{align}
			\dot{p}& =-A^*p-Qx,\\
			\dot{x}& =BB^*p+Ax,
		\end{align}
		that satisfies $x(0)=x(t)=0$. 
			
\end{Def}

\begin{prop}\label{corollary_LQ_optimal_trajectories}
			 Let $(A,B,Q,T)$ be a LQ optimal control problem satisfying the Kalman condition, with $T<t^*$. Let $t^*\in\left(0,+\infty\right]$ be the first conjugate time. Then, we have 
			\begin{itemize}
				\item if $T<t^*$, then for all $x_0,y_0 \in \mathbb{R}^n$ there exists a unique optimal trajectory $\gamma:[0,T]\to \mathbb{R}^n$, between $x_0$ and $y_0$,
				\item if $T=t^*$, existence of optimal trajectories does not hold for all pair of points $(x_0,y_0)\in\mathbb{R}^{2n}$, and there is no uniqueness, 
				\item if $T>t^*$, then for all $x_0,y_0 \in \mathbb{R}^n$ then there are no optimal trajectories $\gamma:[0,T]\to \mathbb{R}^n$ between $x_0$ and $y_0$.  
			\end{itemize}
\end{prop}

        Note that if $Q\leq 0$, then the first conjugate time $t^*$ is $+\infty$ (see \cite{hinpomriff2011}).

    The Hamiltonian vector field associated to a LQ control system \eqref{lq} is linear (upon identifying $T^*\mathbb{R}^n$ with $\mathbb{R}^{2n}$), and thus complete. For the study of its dynamics, it is convenient to study its flow as a smooth curve of block matrices.
\begin{Def}[Matrix blocks of the flow]\label{blocks_R_i}
    Consider the LQ system \eqref{lq}. For $i=1,\ldots,4$, we define $R_i:\mathbb{R}\to\mathbb{R}^{n\times n}$ as the smooth curves of matrices satisfying 
    \begin{equation}\label{formula_R_i}
			e^{-\tau\Omega H}=\begin{pmatrix}
				R_1(\tau) & R_2(\tau)\\
				R_3(\tau) & R_4(\tau)
			\end{pmatrix},\qquad \tau\in\mathbb{R},
		\end{equation}
    where $\tau\mapsto e^{-\tau\Omega H}$ is the flow of the Hamiltonian vector field given by Theorem \ref{ham_LQ}.
\end{Def}
  \begin{lem}\label{time_reflection_symmetry}
        In the setting of Theorem \ref{ham_LQ}, for all $\tau \in \R$ it holds
        \begin{equation}
		R_1^*(\tau)=R_4(-\tau),\qquad R_2^*(\tau)=-R_2(-\tau),\qquad R_3^*(\tau)=-R_3(-\tau).
	\end{equation}
    \end{lem} 
    \begin{proof} Note that
    \begin{equation} 
		(\Omega H)^*=H^*\Omega^*=-H\Omega=-\Omega(-\Omega H)\Omega.
	\end{equation}
    Therefore, for every $\tau\in\mathbb{R}$, the conjugate transposed of the flow is 
    \begin{equation}\label{conj_transposed_of_flow}
        \left(e^{-\tau\Omega H}\right)^*
		=e^{\tau H\Omega}
		=e^{-\Omega(\tau\Omega H)\Omega}
		=-\Omega e^{\tau\Omega H}\Omega.
    \end{equation}
    The thesis follow by comparing block by block in  \eqref{conj_transposed_of_flow}.
    \end{proof}

    As a consequence of Proposition \ref{corollary_LQ_optimal_trajectories}, for every $x,y\in\mathbb{R}^n$, we have $c(x,y)>-\infty$, and the infimum is achieved at a unique trajectory, called optimal trajectory (or action minimizing trajectory).  The corresponding control is referred to as an optimal control. In the following result, we establish the properties of the LQ exponential map. As a byproduct of the proof, we obtain a characterization of all conjugate times $t$ in terms of the singularity of the matrix $R_3(t)$.

    \begin{Def}[Exponential map]
        In the setting of Theorem \ref{ham_LQ}, and for any $x\in\mathbb{R}^n$, we define the LQ exponential map $\exp_{x,T}:T^*_x\mathbb{R}^n\to\mathbb{R}^n$ as
		\begin{equation}
		\exp_{x,T}(p):=x(T), \qquad p \in T_x^*\R^n,
		\end{equation}
        where $x(\cdot)$ is the projection onto $\mathbb{R}^n$ of the integral curve of the LQ Hamiltonian system with initial datum $(p(0),x(0))=(p,x)$.
    \end{Def}
    
\begin{thm}\label{LQ_exponential}
		  Let $(A,B,Q,T)$ be a LQ optimal control problem satisfying the Kalman condition, with $T<t^*$. Then, for every $x\in\mathbb{R}^n$, the exponential map $\exp_{x,T}$ is a smooth diffeomorphism between $T^*_x\mathbb{R}^n$ and $\mathbb{R}^n$.
        
		  In particular, the restriction to the time interval $[0,T]$ of any integral curve of the LQ Hamiltonian system is optimal.
	\end{thm}
	\begin{proof}
By Definition \ref{blocks_R_i}, the integral curve of the LQ Hamiltonian system with initial datum $(p,x) \in T^*\R^n$ has the form
\begin{equation}
    		\begin{pmatrix}
			p(\tau)\\
			x(\tau)
		\end{pmatrix}=e^{-\tau\Omega H}\begin{pmatrix}
			p\\
			x
		\end{pmatrix}=\begin{pmatrix}
				R_1(\tau) & R_2(\tau)\\
				R_3(\tau) & R_4(\tau)
			\end{pmatrix}\begin{pmatrix}
			p\\
			x
		\end{pmatrix}, \qquad \tau\in\mathbb{R}^n.
\end{equation}
		We claim that $R_3(t)$ is singular if and only if $t>0$ is a conjugate time. Indeed, if $t$ is a conjugate time and we take a nontrivial trajectory $(p(\cdot),x(\cdot))$ that goes from $(0,0)$ to $(0,0)$ in time $t$, then
		\begin{equation}
			0=x(t)=R_3(t)p(0),\qquad p(0)\neq 0,
		\end{equation}
		that is $R_3(t)$ is singular. Conversely, if $R_3(t)$ is singular, we can take $p_0\in \ker R_3(t)\setminus\{0\}$. The integral curve starting from $(0,0)$ with initial datum $(p_0,0)$ is nontrivial and arrives at $(0,0)$ in time $t$, proving that $t$ is conjugate.
        
		As a consequence, since $T<t^*$, we have
		\begin{equation}
			\det(R_3(t))\neq 0,\qquad \forall t\in (0,T],
		\end{equation}
		so that, trivially,
		\begin{equation}\label{LQ_exponential_formula}
			\exp_{x,T}(p)=x(T)=R_3(T)p+R_4(T)x
		\end{equation}
		is an affine smooth diffeomorphism, whose inverse is given by 
		\begin{equation}\label{LQ_inverse_exponential_formula}
			y\mapsto \exp_{x,T}^{-1}(y)=R_3^{-1}(T)(y-R_4(T)x).
		\end{equation}

To prove the last statement, for any $x,y \in \R^n$ there exists a unique $p_y\in T^*_x\mathbb{R}^n$, such that $\exp_{x,T}(p_y)=y$. As a consequence, there exists a unique optimal trajectory $(p(\cdot),x(\cdot))$, whose projection joins $x$ and $y$ in time $T$. Such trajectory $x(\cdot)$ is the only candidate to be optimal for the LQ cost between $x$ and $y$ and, therefore, it is indeed optimal.
	\end{proof}
    \begin{rem}\label{formula_opt_traj}
	From previous result, we get a formula for the unique minimizing trajectory between any two points $x,y\in\mathbb{R}^n$ in time $T$. Indeed, we can write it as
	\begin{equation}
		\tau\mapsto\exp_{x,\tau}(\exp_{x,T}^{-1}(y)),\qquad \tau\in[0,T],
	\end{equation}
	which, thanks to formulas \eqref{LQ_exponential_formula} and \eqref{LQ_inverse_exponential_formula}, can be written explicitly in terms of the matrices $R_i$, with $i=1,\ldots,4$, as
	\begin{equation}
		\tau\mapsto R_3(\tau)R_3^{-1}(T)y+\left(R_4(\tau)-R_3(\tau)R_3^{-1}(T)R_4(T)\right)x,\qquad \tau\in[0,T].
	\end{equation}
		\end{rem}

\subsection{Backwards problem}

    The linear system \eqref{lq} is not invariant under time reversal, in general. More precisely, given an admissible trajectory $\gamma_u(\cdot, x_0)$ on the interval $[t,s]$, the reversed trajectory $\tilde{\gamma}_u(\tau,x_0):=\gamma_u(s+t-\tau,x_0)$ is not a solution of \eqref{lq}. Instead, it is a solution of
			\begin{equation}\label{reversed}
				\dot{x}=-Ax+Bu,
			\end{equation}
    with control $\tilde{u}(t)=-u(s+t-\tau)$. We call \eqref{reversed} the backwards system associated with \eqref{lq}. In particular, given an LQ optimal control problem $(A,B,Q,T)$, its backwards counterpart is the LQ problem $(-A,B,Q,T)$.

We analyze the relation between the forward and the backwards Hamiltonian systems. 
\begin{lem}\label{backwards_flow}
Let $(A,B,Q,T)$ a LQ optimal control problem, and let $(-A,B,Q,T)$ be its backwards counterpart . The corresponding LQ backwards Hamiltonian $\tilde{H}:T^*\mathbb{R}^n(\simeq \R^{2n})\to\mathbb{R}$  is
 \begin{equation}
        \tilde{\mathcal{H}}(p,x) \coloneqq (p^*,x^*)\begin{pmatrix}
				BB^* & -A\\
				-A^* & Q
			\end{pmatrix}\begin{pmatrix}
				p\\
				x
			\end{pmatrix}=(p^*,x^*)\tilde{H}\begin{pmatrix}
				p\\
				x
			\end{pmatrix},
    \end{equation}
		 with Hamiltonian vector field $-\Omega\tilde{H}$. Consider the blocks $R_i:\mathbb{R}\to\mathbb{R}^{n\times n}, i=1,\dots,4,$ as in Definition \ref{blocks_R_i}. Then, the flow of the backwards Hamiltonian can be written as 
		\begin{equation}\label{back_flow}
			e^{-\tau\Omega\tilde{H}}=\begin{pmatrix}
				R_1(-\tau) & -R_2(-\tau)\\
				-R_3(-\tau) & R_4(-\tau)
			\end{pmatrix}, \quad \tau\in\mathbb{R}^n.
		\end{equation} 
	\end{lem}
	\begin{proof}
		From explicit computation, we have that 
		\begin{equation}\label{relations_matrix_tilde_H_H}
		\tilde{H}=\begin{pmatrix}
			BB^* & -A\\
			-A^* & Q
		\end{pmatrix}= \begin{pmatrix}
			\mathbb{1}_n & \mathbb{0}_n\\
			\mathbb{0}_n & -\mathbb{1}_n
		\end{pmatrix}\begin{pmatrix}
			BB^* & A\\
			A^* & Q
		\end{pmatrix}\begin{pmatrix}
			\mathbb{1}_n & \mathbb{0}_n\\
			\mathbb{0}_n & -\mathbb{1}_n
		\end{pmatrix}=MHM,
		\end{equation}
		with 
		\begin{equation}\label{relations_matrix_M_Omega}
		M= \begin{pmatrix}
			\mathbb{1}_n & \mathbb{0}_n\\
			\mathbb{0}_n & -\mathbb{1}_n
		\end{pmatrix},\qquad M^2=\mathbb{1}_{2n},\qquad \Omega M=-M\Omega.
		\end{equation}
        Taking into account \eqref{relations_matrix_tilde_H_H} and \eqref{relations_matrix_M_Omega}, the thesis follows by direct exponentiation of $-\tau\Omega\tilde{H}=\tau M\Omega HM$, for every $\tau\in\mathbb{R}^n$.
	\end{proof}
    \begin{rem}\label{rmk_conj_times}
        In Theorem \ref{LQ_exponential}, we proved that a time $t>0$ is conjugate for the LQ problem if and only if the matrix $R_3(t)$ from Definition \ref{blocks_R_i} is singular.
       From Lemmas \ref{backwards_flow} and \ref{time_reflection_symmetry}, in the same way, we see that a time $t>0$ is a conjugate time for the backwards problem if and only if $-R_3(-t)=R_3^*(t)$ is not invertible. This implies that  any conjugate time of a LQ problem is a conjugate time for the corresponding backwards problem and vice versa.
       
	Lastly, we take an integral curve of the forward Hamiltonian $(p(\cdot),x(\cdot))$ with optimal control $u=B^*p$ and we consider the backwards trajectory $\tilde{x}(\cdot)=x(T-\cdot)$. Define $\tilde{u}=B^*\tilde{p}$ and $(\tilde{p}(\cdot),\tilde{x}(\cdot))$ the corresponding control and lift. We already observed that 
	\begin{equation}
		\tilde{u}(\tau)=-u(T-\tau),
	\end{equation}
	and, from the Hamiltonian equations, it is straightforward to see that also
	\begin{equation}
		\tilde{p}(\tau)=-p(T-\tau).
	\end{equation} 
    \end{rem}

	\subsection{The LQ cost function}
	We compute the family of LQ cost functions of Definition \ref{cost}. 
	\begin{thm}\label{comp_cost}
		Let $(A,B,Q,T)$ be a LQ optimal control problem  satisfying the Kalman condition, with $T<t^*$. Let 
		\begin{equation}
			e^{-\tau\Omega H}=\begin{pmatrix}
				R_1(\tau) & R_2(\tau)\\
				R_3(\tau) & R_4(\tau)
			\end{pmatrix},\qquad \tau\in\mathbb{R},
		\end{equation}
		be the associated Hamiltonian flow. Then, the LQ cost function is a homogeneous polynomial of degree $2$ of the form
		\begin{equation}\label{formula_cost_LQ}
			c(x,y)=\frac{x^*R_3(T)^{-1}R_4(T)x}{2}-x^*R_3(T)^{-1}y-\frac{y^*R_3^{-1}(-T)R_4(-T)y}{2},\qquad x,y\in\mathbb{R}^n.
		\end{equation}
	\end{thm}
A similar expression was obtained in \cite{hinpomriff2011} for the case $-Q\geq0$. Our proof combines the forward and backwards problems to yield a neat and manageable formula. 
With the same argument, we deduce a neat formula also for the intermediate LQ costs. The proof is omitted.
	\begin{cor} In the same setting of Theorem \ref{comp_cost}, for $0\leq t<s\leq T$, the LQ intermediate cost functions are homogeneous polynomials of degree $2$ of the form
		\begin{equation}
			c^{t,s}(x,y)=\frac{x^*R_3^{-1}(s-t)R_4(s-t)x}{2}-x^*R_3^{-1}(s-t)y-\frac{y^*R_3^{-1}(t-s)R_4(t-s)y}{2},\qquad x,y\in\mathbb{R}^n.
		\end{equation}
	\end{cor}
    \begin{rem}[The cost is not bounded from below and not symmetric]
	The LQ cost function is a quadratic form but not positive definite in general. Indeed, the LQ cost function may be:
    \begin{itemize}
        \item not globally bounded from below. E.g., for $n=1$ and $(A,B,Q,T)=(0,1,1,\pi/2)$, the Hamiltonian function is $\mathcal{H}(p,x)=(p^2+x^2)/2$ and the cost is $c(x,y)=-xy$,
        \item non-invariant upon exchanging $x$ with $y$, in presence of non-zero drift. E.g., for $n=1$ and $(A,B,Q,T)=(1,1,1,1)$, the Hamiltionian function is $\mathcal{H}(p,x)=(p+x)^2/2$ and the cost is $c(x,y)=x^2-xy$.
    \end{itemize}
    \end{rem}

In order to prove Theorem \ref{comp_cost} we need two lemmas.
	\begin{lem}\label{lem1}
		In the same setting of Theorem \ref{comp_cost}, there exist matrices $C,D,E\in\mathbb{R}^{n\times n}$, with $C=C^*$ and $E=E^*$, such that the LQ cost function can be written as
		\begin{equation}
			c(x,y)=\frac{x^*Cx}{2}+x^*Dy+\frac{y^*Ey}{2},\qquad x,y\in\mathbb{R}^n.
		\end{equation} 
	\end{lem}
	\begin{proof}
		Let $(p(\cdot),x(\cdot)):[0,T]\to \R^{2n}$ be the unique integral trajectory of $\overrightarrow{\mathcal{H}}$, whose projection is the optimal trajectory between $x$ and $y$. 
         In terms of the Hamiltonian flow, we have
        \begin{align}
        p(\tau) & = R_1(\tau) p + R_2(\tau) x,\\
        x(\tau) & = R_3(\tau) p + R_4(\tau)x.
        \end{align}
        Since $x(T)=y$, it holds
		\begin{equation}\label{inv_exp}
			p =R_3^{-1}(T)(y-R_4(T)x).
		\end{equation}
		Then, recalling that the corresponding control satisfies $u(\tau)=B^*p(\tau)$, we have 
		\begin{align}
			c(x,y) &=\frac{1}{2}\int_{0}^{T}\left(p^*(\tau)BB^*p(\tau)-x^*(\tau)Qx(\tau)\right)d\tau\\
			&=\frac{p^*Jp}{2}+x^*Kp+\frac{x^*Lx}{2},
		\end{align}
		where
		\begin{align}
			J&:=\int_{0}^{T}\left(R_1^*(\tau)BB^*R_1(\tau)-R_3^*(\tau)QR_3(\tau)\right)d\tau,\\
			K&:=\int_{0}^{T}\left(R_2^*(\tau)BB^*R_1(\tau)-R_4^*(\tau)QR_3(\tau)\right)d\tau,\\
			L&:=\int_{0}^{T}\left(R_2^*(\tau)BB^*R_2(\tau)-R_4^*(\tau)QR_4(\tau)\right)d\tau.
		\end{align}
		Further, using relation \eqref{inv_exp}, we get 
		\begin{equation}
			c(x,y)=\frac{x^*Cx}{2}+x^*Dy+\frac{y^*Ey}{2},\qquad x,y\in\mathbb{R}^n,
		\end{equation} 
		where 
		\begin{align}
			C&=R_4^*(T)R_3^{-*}(T)JR_3^{-1}(T)R_4(T)-2KR_3^{-1}(T)R_4(T)+L,\\
			D&=-R_4^*(T)R_3^{-*}(T)JR_3^{-1}(T)+KR_3^{-1}(T),\\
			E&=R_3^{-*}(T)JR_3^{-1}(T),
		\end{align}
		and $R^{-*}_3(T)$ stands for $(R_3^*(T))^{-1}$.
        
		We notice that $J=J^*$, which implies that $E=E^*$. Observe that the same computation for the backwards problem yields another formula for the same cost function (with the entries swapped). Namely, the backwards cost has the form
		\begin{equation}
			\tilde{c}(y,x)=\frac{y^*\tilde{C}y}{2}+y^*\tilde{D}x+\frac{x^*\tilde{E}x}{2},
		\end{equation}
		with $\tilde{E}=\tilde{E}^*$. Since $c(x,y)=\tilde{c}(y,x)$ for every $x,y\in\mathbb{R}^{n}$, we get $C=C^*$.
	\end{proof}
	Now that we have established the structure and regularity of the cost function, we make the matrices $C,D,E$ more explicit. The next Lemma is a version of \cite[Proposition 4.3]{Agrachev_2018}, adapted to the LQ framework. The proof is omitted.
	\begin{lem}\label{lem2}
Let $x,y\in\R^n$. The unique $p\in \R^n$ such that $y=\exp_{x,T}(p)$ can be obtained as
        \begin{equation}
            p=-\nabla_x c(x,y),
        \end{equation}
        where $\nabla_x$ denotes the Euclidean gradient w.r.t.\ the subscript variable.
 	\end{lem}
\begin{proof}[Proof of Theorem \ref{comp_cost}]
    Recall that by Lemma \ref{lem1} we have
    \begin{equation}
		c(x,y)=\frac{x^*Cx}{2}+x^*Dy+\frac{y^*Ey}{2}.
	\end{equation} 
    We compute $C,D,E$ in terms of the LQ Hamiltonian flow. By Lemma \ref{lem2}, and \eqref{inv_exp}  we have
\begin{equation}
    R_3^{-1}(T)(y-R_4(T)x) = -\nabla_x c(x,y) = -(Cx+Dy).
\end{equation}
 By comparing term by term, it follows that 
	\begin{align}
		C=R_3^{-1}(T)R_4(T), \qquad \text{and} \qquad
		D=-R_3^{-1}(T).
	\end{align}
     To obtain an expression for $E$, recall that $c(x,y) = \tilde{c}(y,x)$ where $\tilde{c}$ is the cost of the backwards LQ problem.        Thus, arguing as above but for the backwards problem, we obtain
    \begin{equation}
            -\nabla_y \tilde{c}(y,x) = -\nabla_y c(x,y)= -(Dx + Ey).
    \end{equation}
        Using the relation between forward and backwards Hamiltonian flow of Lemma \ref{backwards_flow}, we get
    \begin{equation}
        E = -R_3^{-1}(-T)R_4(-T),
    \end{equation}
    concluding the proof.
\end{proof}

\section{Optimal transport with linear quadratic cost}

\subsection{Monge and Kantorovich problems}
Given a lower semi-continuous function $c:\R^n \times \R^n \to\mathbb{R}$ (called cost function) and Borel probability measures $\mu,\nu\in P(\R^n)$, the Monge problem of optimal transport from $\mu$ to $\nu$ is the problem of minimizing
\begin{equation}\label{Monge}
	C_{\mathcal{M}}(T)=\int_{\R^n}c(x,T(x))d\mu(x),
\end{equation}
among all measurable maps $T:\mathbb{R}^n\to \mathbb{R}^n$, such that $T_{\sharp}\mu=\nu$, i.e.\ $\mu(T^{-1}(B))=\nu(B)$ for all Borel sets $B\subset \R^n$. Any such $T$ is called a transport map, while any map $T$ that realizes the minimum in \eqref{Monge} is called an optimal transport map.
We define the total transport cost as  
\begin{equation}\label{trans_cost}
	C_{\mathcal{M}}(\mu,\nu):=\inf \left\{C_{\mathcal{M}}(T)\mid T \text{ measurable, } T_{\sharp}\mu=\nu\right\}.
\end{equation}
The set of transport maps may be empty and, even when transport maps do exist, the constraint $T_{\sharp}\mu=\nu$ is non-linear, which makes difficult to solve directly the Monge problem from an optimization viewpoint. For these reasons, one considers instead  the relaxed version of  \eqref{Monge}, called Kantorovich problem, namely the minimization of
	\begin{equation}\label{Kantorovich}
		C_{\mathcal{K}}(\alpha)=\int_{\R^n \times \R^n}c(x,y)d\alpha(x,y),
	\end{equation} 
	among all transport plans $\alpha$, namely $\alpha \in P(\R^n \times \R^n)$ such that $\pi_{1\sharp}\alpha=\mu,\pi_{2\sharp}\alpha=\nu$.
	Any trasport map $T$ induces an associated transport plan $\alpha_T=(id,T)_{\sharp}\mu$, which is concentrated on the graph of $T$. Conversely, if a transport plan is concentrated on a graph, then it is induced by a measurable map.
	Similarly to \eqref{trans_cost}, we define the Kantorovich total transport cost as 
	\begin{equation}
		C_{\mathcal{K}}(\mu,\nu):=\inf \left\{C_{\mathcal{K}}(\alpha)\mid \alpha\in P(\mathcal{X}\times \mathcal{Y}),\pi_{1\sharp}\alpha=\mu,\pi_{2\sharp}\alpha=\nu\right\}.
	\end{equation}
	Since $C_{\mathcal{K}}(\alpha_T)=C_{\mathcal{M}}(T)$ for all transport maps $T$, we have
	\begin{equation}
		C_{\mathcal{K}}(\mu,\nu)\leq C_{\mathcal{M}}(\mu,\nu).
	\end{equation}

\subsection{Cyclical monotonicity and potentials}
           The supports of optimal transport plans satisfy a geometric condition, that ensures the impossibility of lowering the cost, by relocating the mass.
        \begin{Def}[$c$-cyclical monotonicity]
			Let $c:\R^n\times \R^n \to \R$. A set $S\subset \R^n \times \R^n$ is $c$-cyclically monotone if, for every $N\in\mathbb{N}$, for any points $\left\{(x_1,y_1),\dots,(x_N,y_N)\right\} \subset S$ and any permutation $\sigma\in S_N$, we have 
			\begin{equation}
				\sum_{i=1}^{N}c(x_i,y_i)\leq\sum_{i=1}^{N}c(x_{\sigma(i)},y_i).
			\end{equation}
		\end{Def}

An important property of $c$-cyclically monotone sets is that they are realized as contact sets of a suitable pair of functions, called Kantorovich potentials. 
	We start by the following definition.
	\begin{Def}[$c$-convexity and $c$-concavity]\label{c-convexity}
		Let $c:\mathbb{R}^n\times\mathbb{R}^n\to\mathbb{R}$.
		A function $\psi:\mathbb{R}^n\to\mathbb{R}\cup\left\{+\infty\right\}$, not identically $+\infty$, is said to be $c$-convex if there is a non-empty set $\mathcal{A}\subset\mathbb{R}^n\times\mathbb{R}$, such that
		\begin{equation}\label{def_c_convex}
			\psi(x)=\sup\left\{\lambda-c(x,y)\;|\;(y,\lambda)\in\mathcal{A}\right\},\qquad\forall x\in\mathbb{R}^n.
		\end{equation}
		The $c$-transform of $\psi$, denoted by $\psi^{c+}$, is the function $\psi^{c+}:\mathbb{R}^n\to\mathbb{R}\cup\left\{-\infty\right\}$ defined by 
		\begin{equation}
			\psi^{c+}(y):=\inf\{\psi(x)+c(x,y)\mid x\in\mathbb{R}^n\},\qquad\forall y\in\mathbb{R}^n.
		\end{equation}
		 
	Similarly, a function $\phi:\mathbb{R}^n\to\mathbb{R}\cup\{-\infty\}$, not identically $-\infty$, is said to be $c$-concave if there exists a non-empty set $\mathcal{B}\subset\mathbb{R}^n\times\mathbb{R}$, such that
	\begin{equation}
		\phi(y)=\inf\left\{\lambda+c(x,y)\;|\;(x,\lambda)\in\mathcal{B}\right\},\qquad\forall y\in\mathbb{R}^n.
	\end{equation}
	In such case, the inverse $c$-transform $\phi^{c-}$ of $\phi$ is defined as 
	\begin{equation}
		\phi^{c-}(x):=\sup\{\phi(y)-c(x,y)\mid y\in\mathbb{R}^n\},\qquad\forall x\in\mathbb{R}^n.
	\end{equation}
	A pair $\left(\psi,\psi^{c+}\right)$ or a pair $(\phi^{c-},\phi)$ is called a pair of $c$-potentials.
	\end{Def}
	\begin{rem}
		The $c$-transform of a $c$-convex function is $c$-concave, provided that it is not identically $-\infty$. Conversely, the inverse $c$-transform of a $c$-concave function is $c$-convex, if it is not identically $+\infty$.
	\end{rem}
A $c$-convex (resp.\ $c$-concave) function can be reconstructed from its $c$-transform (resp.\ inverse $c$-transform), see, for instance, \cite[Proposition 3.5]{RiffSrOptTrans}.

 	\begin{prop}\label{prop_c_concavity}
		Given a $c$-convex function $\psi$, it holds
		\begin{equation}\label{recover_c_convex}
			\psi(x)=\sup\{\psi^{c+}(y)-c(x,y)\mid  y\in\mathbb{R}^n\},\qquad\forall x\in\mathbb{R}^n.
		\end{equation}
		Similarly, given a $c$-concave function $\phi$, it holds
		\begin{equation}\label{recover_c_concave}
			\phi(y)=\inf\{\phi^{c-}(x)+c(x,y)\mid  x\in\mathbb{R}^n\},\qquad\forall y\in\mathbb{R}^n.
		\end{equation}
	\end{prop}
	If two functions are one the $c$-transform of the other, we call them $c$-conjugate and we always indicate with a superscript ``$c$'' the $c$-concave one of the pair, dropping the notation $\psi^{c+}$, $\phi^{c-}$ when there is no risk of confusion.
	\begin{Def}
		Let $\psi:\mathbb{R}^n\to\mathbb{R}\cup\left\{+\infty\right\}$ be a $c$-convex function. For every $x\in\mathbb{R}^n$, the $c$-subdifferential of $\psi$ at $x$ is
		\begin{equation}
			\partial_c\psi(x):=\{y\in\mathbb{R}^n\;|\;\psi^c(y)=c(x,y)+\psi(x)\}.
		\end{equation}
		We define the contact set of the pair $\left(\psi,\psi^c\right)$ as
		\begin{equation}
			\partial_c\psi:=\{(x,y)\in\mathbb{R}^n\times\mathbb{R}^n \mid y\in\partial_c\psi(x)\}.
		\end{equation}
	\end{Def}
	\begin{rem}\label{charach_contact_set}
		A pair $(x,y)\in\mathbb{R}^n\times\mathbb{R}^n$ belongs to the contact set $\partial_c\psi$ if and only if
		\begin{equation}
			\psi(x)+c(x,y)\leq\psi(z)+c(z,y),\qquad\forall z\in\mathbb{R}^n,
		\end{equation}
		or, equivalently,
		\begin{equation}
			\psi^c(y)-c(x,y)\geq\psi^c(w)-c(x,w),\qquad\forall w\in\mathbb{R}^n.
		\end{equation}
		In particular, both $\psi(x)$ and $\psi^c(y)$ are finite for $(x,y)\in \partial_c\psi$.
	\end{rem}

        \subsection{The case of quadratic costs}

                We will consider the above problems for the case of LQ optimal control problems. Since the LQ cost is a quadratic homogeneous polynomial, the correct framework for studying the LQ optimal transport problem is the Wasserstein space of order $2$. We recall it in the following definition (see \cite[Chapter 6]{Villanioldandnew} or \cite[Section 3]{Usersguide}). 
    \begin{Def}[Wasserstein distance]
		For $\mu,\nu\in P(\mathbb{R}^n)$, define
		\begin{equation}
			W_2(\mu,\nu):=\left(\inf\limits_{\{\alpha\in P(\mathbb{R}^{2n}),\pi_{1\sharp}\alpha=\mu,\pi_{2\sharp}\alpha=\nu\}}\int_{\mathbb{R}^{n}\times \R^n}|x-y|^2d\alpha(x,y)\right)^{1/2}.
		\end{equation}
		On the set of probability measures with finite second momentum, that is
		\begin{equation}
			P_2(\mathbb{R}^n):=\left\{\mu\in P(\mathbb{R}^n)\;\bigg|\;\int_{\mathbb{R}^n}|x-x_0|^2d\mu(x)<+\infty\right\},
		\end{equation}
        for some (and then any) $x_0\in\R^n$. The function $W_2$ is a distance, giving it the structure of a complete separable metric space, called the Wasserstein space of order $2$.
	\end{Def}
        Note that, if $\mu,\nu\in P_2(\mathbb{R}^n)$, then any transport plan $\alpha$ between them is in $ P_2(\mathbb{R}^{n}\times \R^n)$.
        

        The following result is a special case of \cite[Theorem 4.1]{Villanioldandnew}.
		\begin{thm}\label{existence_optimal_plans}
		Let $c:\mathbb{R}^n\times\mathbb{R}^n\to\mathbb{R}$ be a lower semi-continuous function satisfying 
		\begin{equation}\label{hyp_quadratic}
			c(x,y)\geq -\Lambda(|x|^2+|y|^2),\qquad\forall(x,y)\in\mathbb{R}^n\times\mathbb{R}^n,
		\end{equation}
		for some $\Lambda>0$. Let $\mu,\nu\in P_2(\mathbb{R}^n)$ and define the set of  admissible transport plans as 
		\begin{equation}
			\Adm(\mu,\nu):=\{\alpha\in P(\mathbb{R}^n\times\mathbb{R}^n)\mid  \pi_{1\sharp}\alpha=\mu,\pi_{2\sharp}\alpha=\nu\}.
		\end{equation}  
		Then there exists a solution for the Kantorovich problem, namely $\tilde{\alpha}\in\Adm(\mu,\nu)$ such that
		\begin{equation}
			C_{\mathcal{K}}(\tilde{\alpha})=\min\left\{C_{\mathcal{K}}(\alpha)\mid \alpha\in\Adm(\mu,\nu)\right\}=C_{\mathcal{M}}(\mu,\nu).
		\end{equation}
		\end{thm}
        As an application, the Kantorovich problem for the LQ cost admits solutions in $P_2(\mathbb{R}^n)$. 
		\begin{cor}\label{finite_total_cost}
			Let $(A,B,Q,T)$ be a LQ optimal control problem satisfying the Kalman condition, with $T<t^*$. Then, for all $\mu,\nu\in P_2(\mathbb{R}^n)$, the Kantorovich total transport cost $C_{\mathcal{K}}(\mu,\nu)$, associated with the LQ cost, is finite and the Kantorovich problem admits solutions.
		\end{cor}
		\begin{proof}
            By Theorem \ref{comp_cost}, the cost $c=c^{0,T}$ is a quadratic homogeneous polynomial. In particular, there exists $\Lambda>0$ such that 
			\begin{equation}\label{quadratic_bound_LQ_cost}
				-\Lambda (|x|^2+|y|^2) \leq c(x,y) \leq \Lambda (|x|^2+|y|^2),\qquad \forall x,y\in\R^n.
			\end{equation}
            The upper bound implies that the total transport cost $C_{\mathcal{K}}(\mu,\nu)$ is finite. The lower bound, using Theorem \ref{existence_optimal_plans}, implies the existence of an optimal transport plan.
		\end{proof}

        The next result clarifies the relation between $c$-cyclical monotonicity, $c$-convex functions and optimal plans. Thanks to Corollary \ref{finite_total_cost}, it can be applied to the case of LQ costs.  The proof of Theorem \ref{supp_opt_plans} is classical (see e.g. \cite[Theorem 5.10]{Villanioldandnew}) and we omit it.

    \begin{thm}\label{supp_opt_plans} Let $c:\mathbb{R}^n\times\mathbb{R}^n\to\mathbb{R}$ be a continuous function satisfying 
		\begin{equation}\label{hp_quadraticity_cost}
			c(x,y)\geq -\Lambda(|x|^2+|y|^2),\qquad\forall x,y \in\mathbb{R}^n,
		\end{equation}
		and consider $\mu,\nu\in P_2(\mathbb{R}^n)$, such that $C_{\mathcal{K}}(\mu,\nu)<+\infty$. Then, for every $\alpha\in\Adm(\mu,\nu)$, the following are equivalent
		\begin{enumerate}
			\item $\alpha$ is optimal,
			\item $\alpha$ is concentrated on a $c$-cyclically monotone set,
                \item there exists a $c$-convex function $\psi:\mathbb{R}^n\to\mathbb{R}\cup\{+\infty\}$ such that $\alpha(\partial_c\psi)=1$.
		\end{enumerate}
       Since $c$ is continuous, any optimal plan $\alpha$ is actually supported on a $c$-cyclically monotone set.
       \end{thm}
    
We end this section with regularity properties of Kantorovich potentials in the case of LQ costs. We provide a direct and simple proof, for a general argument see \cite[Theorem 10.26]{Villanioldandnew}. Definition of semi-convexity and semi-concavity can be found in Appendix \ref{app:semiconcave}.

	\begin{prop}\label{Kant_potentials_are_semi-convex}
        Let $(A,B,Q,T)$ be a LQ optimal control problem satisfying the Kalman condition, with $T<t^*$. Then, any $c$-convex function  $\psi:\mathbb{R}^n\to\mathbb{R}\cup\left\{+\infty\right\}$ is lower semi-continuous and semi-convex on $\mathbb{R}^n$. In addition
        \begin{equation}
            \mathrm{Dom}(\psi):=\{x\in\mathbb{R}^n\mid \psi(x)<+\infty\} \subset \R^n
        \end{equation}
        is convex. Similarly, any $c$-concave function $\phi: \mathbb{R}^n\to \mathbb{R}\cup \{-\infty\}$ is upper semi-continuous and semi-concave on $\mathbb{R}^n$, and $\mathrm{Dom}(\phi)$ is a convex set.
        \end{prop}
	\begin{proof} 
We prove the statement for $c$-convex functions, the one for $c$-concave ones is similar. By Definition \ref{c-convexity}, there exists $\mathcal{A} \subset \R^n \times \R$ such that
 \begin{equation}
      \psi(x) = \sup_{(y,\lambda) \in \mathcal{A}} \{\lambda - c(x,y)\}, \qquad \forall x \in \R^n.
       \end{equation}
By Theorem \ref{comp_cost}, the cost $c=c^{0,T}$ is a degree $2$  polynomial, and so $\psi$ is the pointwise supremum of a family of smooth functions whose Hessian is uniformly bounded from above. In particular, $\psi$ is lower semi-continuous and, by one of the equivalent characterizations of semi-convexity (see Proposition \ref{prop_charac_semi_concave_smooth}), $\psi$ is semi-convex on $\R^n$. Furthermore, since any semi-convex function is the sum of a convex function and a $C^2$ function (again by Proposition \ref{prop_charac_semi_concave_smooth}), if $x_0,x_1 \in \mathrm{Dom}(\psi)$, the whole segment between $x_0$ and $x_1$ belongs to $\mathrm{Dom}(\psi)$.
\end{proof}

    \subsection{Brenier-McCann theorem and continuity of the transport map}
    In this section we discuss the well-posedness of Monge problem for the LQ cost $c=c^{0,T}$, and the continuity of the optimal transport map. The existence and uniqueness part is not new for measures that are Lebesgue-absolutely continuous. For LQ problems with $-Q\geq 0$, a proof can be found in \cite{hinpomriff2011}, and also is a consequence of \cite[Theorem 4.1]{agrachev2007opttransnonholconst}. The extension to general $Q$ is straightforward. We further extend such results to a wider class of measures, namely those that do not charge $c-c$-hypersurfaces (cf. Definition \ref{def_c_c_hyper}), inspired by the techniques applied in the case of squared-distance cost (see \cite{GangboMcCann1996,GigliInverseBrenier2009}).
    \begin{thm}[Well-posedness]\label{Brenier_McCann}
        Let $(A,B,Q,T)$ be a LQ optimal control problem satisfying the Kalman condition, with $T<t^*$. Let $\mu,\nu\in P_2(\mathbb{R}^n)$, and assume that $\mu$ does not give mass to any $c-c$-hypersurface. Then, there exists a unique optimal transport plan $\alpha\in P_2(\mathbb{R}^n\times\mathbb{R}^n)$ between $\mu$ and $\nu$ for the LQ optimal transport problem, induced by a map $T:\mathbb{R}^n\to\mathbb{R}^n$.
        In addition, any Kantorovich potential $\psi$ associated with $\alpha$ is $\mu$-almost everywhere differentiable and the optimal transport map $T$ satisfies
        \begin{equation}
            T(x)=\exp_{x,T}(\nabla\psi(x)),\qquad\qquad \mu\text{-a.e. }x\in\mathbb{R}^n,
        \end{equation}
        where $\exp_{x,\tau}:T_x\mathbb{R}^n(\cong\mathbb{R}^n)\to\mathbb{R}^n$, for $\tau\in[0,T]$, denotes the LQ exponential map.

        If $\mu\ll\mathscr{L}^n$, then $\psi$ is $\mu$-almost everywhere twice differentiable. In particular, $T$ is $\mu$-almost everywhere differentiable.
	\end{thm}
	\begin{proof}
        Consider any optimal transport plan $\alpha$, between $\mu$ and $\nu$, and let $\psi$ be any Kantorovich potential associated with $\alpha$ via Theorem \ref{supp_opt_plans}. By Remark \ref{charach_contact_set}, $\mathrm{Dom}(\psi) \subset \pi_1(\partial_c\psi)$. Furthermore, $\alpha(\partial_c\psi)=1$ and $\mu=\pi_{1\sharp}\alpha$. It follows that $\mu(\mathrm{Dom})=1$.
              
        Therefore, since $\mu$ does not charge $c-c$-hypersurfaces, $\mathrm{Dom}(\psi)$ is a convex set with non-empty interior and $\mu$ is concentrated on $\mathrm{int}(\mathrm{Dom}(\psi))$, where $\mathrm{int}(E)$ denotes the topological interior of $E\subset \mathbb{R}^n$. 
        By Zaj\'i\v cek Theorem (see Theorem \ref{Zajicek_thm}), $\psi$ is differentiable $\mu$-almost everywhere in $\mathrm{int}(\mathrm{Dom}(\psi))$.

		Consider $x\in \mathrm{int}(\mathrm{Dom}(\psi))$, where $\psi$ is differentiable and $\partial_c\psi(x)$ is non-empty (which happens $\mu$-almost everywhere since $\alpha(\partial_c\psi)$=1). We want to prove that 
		\begin{equation}
			\partial_c\psi(x)=\{y\in\mathbb{R}^n\;|\;\psi^c(y)=c(x,y)+\psi(x)\},
		\end{equation}
		is a singleton.
        
        Thus, consider $y\in\partial_c\psi(x)$. This means that 
		\begin{equation}
			\psi(x)+c(x,y)\leq\psi(z)+c(z,y),\qquad\forall z\in\mathbb{R}^n.
		\end{equation}
		The function $z\mapsto\psi(z)+c(z,y)$ is differentiable at $z=x$ and it attains a minimum there.
		Therefore, thanks to Lemma \ref{lem2}, we have 
		\begin{equation}
			0=\nabla\psi(x)+\nabla_x c(x,y)=\nabla\psi(x)-\exp_{x,T}^{-1}(y),
		\end{equation} 
		so that 
		\begin{equation}
			y=\exp_{x,T}(\nabla\psi(x))
		\end{equation}
		is uniquely determined by $x$. As a consequence, the disintegration of $\alpha$ with respect to the first projection $\pi_1$ is a Dirac delta, for $\mu$-almost every point. Therefore, $\alpha$ is unique, and induced by a map that can be represented as 	
		\begin{equation}\label{representation_t_map_internal}
			T(x)=\exp_{x,T}(\nabla\psi(x)),
		\end{equation}
	for $\mu$-almost every $x\in\mathbb{R}^n$.
        
        Finally, by Alexandroff's Theorem (see Theorem \ref{Alexandroff's_theorem}), $\psi$ is twice differentiable $\mathscr{L}^n$-almost everywhere in $\mathrm{int}(\mathrm{Dom}(\psi))$. In particular, if $\mu\ll\mathscr{L}^n$, then $\psi$ is $\mu$-almost everywhere twice differentiable; therefore, thanks to \eqref{representation_t_map_internal}, also $T$ is differentiable $\mu$-almost everywhere.
	\end{proof}

    The continuity of the transport map can be reduced to the regularity theory of the Euclidean the Monge-Amp\`ere equation, following the same argument used in \cite[Theorem 2.2]{hinpomriff2011}, for $-Q\geq 0$.  Indeed, under suitable regularity of the measures $\mu,\nu$, the first Kantorovich potential $\psi$ is unique, up to addition of a constant, and satisfies 
    \begin{equation}\label{Monge-Ampere}
			\det\big(\Hess(\psi)(x)+R_3^{-1}(T)R_4(T)\big)=|\det(R_3(T))|^{-1}\frac{f(x)}{g(T(x))},
    \end{equation}
    where $f,g$ are the densities of the measures $\mu,\nu$ with respect to $\mathscr{L}^n$, and $T$ is the optimal transport map for the LQ cost, from $\mu$ to $\nu$. Up to change of variables, equation \eqref{Monge-Ampere} can be reduced to the usual Euclidean Monge-Amp\`ere equation, to prove regularity properties of $\psi$ (and thus of the optimal transport map $T$).
    The next theorem is a part of \cite[Theorem 12.50]{Villanioldandnew} and is based on the original work of Caffarelli, see \cite{CaffarelliRegConvMaps}.
		\begin{thm}[Caffarelli]\label{Caffarelli_regularity}
			Let $|x-y|^2/2$ be the squared-distance cost in $\mathbb{R}^n\times\mathbb{R}^n$, and let $\Omega_1,\Omega_2\subset\mathbb{R}^n$ be connected bounded open sets. Let $f,g$ be probability densities on $\Omega_1$, $\Omega_2$, respectively, bounded both from above and below. Let $\psi:\Omega_1\to\mathbb{R}$ be the unique (up to an additive constant) Kantorovich potential associated with the probability measures $\mu=f\mathscr{L}^n$, $\nu=g\mathscr{L}^n$ and the cost $\tilde{c}$. Then, if $\Omega_2$ is convex, $\psi\in\mathcal{C}^{1,\beta}(\Omega_1)$, for some $\beta\in(0,1)$.
		\end{thm}
		\begin{thm}[Continuity of the transport map]\label{continuity_transport_map}
			Let $\mu, \nu\in P_2(\mathbb{R}^n)$ be compactly supported measures, such that $\mu,\nu\ll\mathscr{L}^n$. Assume, in addition, that $\mu=f\mathscr{L}^n$, $\nu=g\mathscr{L}^n$ with $f$ and $g$ bounded from above and below on $\supp(\mu)$ and $\supp(\nu)$, respectively, and $\mathrm{int}(\supp(\mu))$ is connected, whilst $\supp(\nu)$ is convex. Then, the optimal transport map between $\mu$ and $\nu$ admits a representative $T:\mathbb{R}^n\to\mathbb{R}^n$ that is $\beta$-Hölder continuous, for some $\beta\in(0,1)$.
		\end{thm}
		\begin{proof}
			Let $(\psi,\psi^c)$ be a pair Kantorovich potentials from  $\mu$ to $ \nu$ for the cost $c=c^{0,T}$. In particular, using notation of Lemma \ref{lem1}, for every $x\in\mathbb{R}^n$, we have
			\begin{align}
				\psi(x)&=\sup\big\{\psi^c(y)-c(x,y)\mid y\in\mathbb{R}^n\big\}\\
					&=-\frac{x^*Cx}{2}+\sup\left\{\psi^c(y)-x^*Dy-\frac{y^*Ey}{2}\bigg|\;y\in\mathbb{R}^n\right\},
			\end{align}
			with $C=R_3^{-1}(T)R_4(T)$ (recall Theorem \ref{comp_cost}).
			Hence, $\psi(x)=-x^*Cx/2+\phi(x)$, where 
			\begin{equation}
				\phi(x):=\sup\left\{\psi^c(y)-x^*Dy-\frac{y^*Ey}{2}\;\bigg|\;y\in\mathbb{R}^n\right\},\qquad\forall x\in\mathbb{R}^n,
			\end{equation}
			is convex and lower semi-continuous, being the supremum of a family of affine functions.
			Moreover, by Theorem \ref{Brenier_McCann}, for $\mu$-almost every $x\in\mathbb{R}^n$ we have that
			\begin{align}
				T(x)&=\exp_{x,T}(\nabla\psi(x))\\
						&=R_3(T)\nabla\psi(x)+R_4(T)x\\
						&=R_3(T)\left(\nabla\psi(x)+Cx\right)\\
						&=R_3(T)\nabla\phi(x),
			\end{align}
			where we used formula \eqref{LQ_exponential_formula} from Theorem \ref{LQ_exponential}.
			We consider the invertible linear map $F:\mathbb{R}^n\to\mathbb{R}^n$, defined by $F(z):=R_3(T)z$, and we set $\eta=F^{-1}_{\sharp}\nu$. Since $T_{\sharp}\mu=\nu$, it follows that $\nabla\phi_{\sharp}\mu=\eta$. In addition, $\phi$ is convex, hence $\nabla\phi$ is the optimal transport map from $\mu$ to $\eta$, for the Euclidean squared-distance cost $|x-y|^2/2$ (see \cite[Theorem 1.3]{Brenier1991}). $F$ is an invertible linear map, $\supp(\eta)$ is convex and, if $\eta=h\mathscr{L}^n$, then $h$ is bounded above and below on $\supp(\eta)$. Therefore, by Theorem \ref{Caffarelli_regularity}, we obtain that $\phi\in\mathcal{C}^{1,\beta}(\mathrm{int}(\supp(\eta)))$, for some $\beta\in(0,1)$, which, by construction, implies that $\psi\in\mathcal{C}^{1,\beta}(\mathrm{int}(\supp(\mu)))$ and $T$ is $\beta$-Hölder continuous. 
		\end{proof}
		\begin{rem}[Higher regularity of the transport map]\label{rmk:highereg}
			 With the same strategy, Caffarelli's theory can be used to derive higher regularity of the transport map from higher regularity of the measures $\mu$ and $\nu$ (see \cite[Theorem 12.50]{Villanioldandnew} and \cite{CaffarelliRegConvMaps}). More precisely, set $\Omega_1=\mathrm{int}(\supp(\mu))$ and $\Omega_2=\mathrm{int}(\supp(\nu))$ and assume that, for some $k\in\mathbb{N}$ and $\alpha\in(0,1)$, it holds
             \begin{itemize}
                 \item $\Omega_1,\Omega_2$ are uniformly convex, relatively compact and with boundary of class $\mathcal{C}^{k+2}$,
                 \item $f\in\mathcal{C}^{k,\alpha}(\overline{\Omega}_1)$ and $g\in\mathcal{C}^{k,\alpha}(\overline{\Omega}_2)$.
             \end{itemize}
             Then, the unique optimal transport map from $\mu$ to $\nu$, for the LQ cost $c=c^{0,T}$, admits a representative of class $\mathcal{C}^{k+1,\alpha}(\overline{\Omega}_1)$. 
		\end{rem}
        
    \section{Displacement interpolations}\label{representation_theorem}
    \subsection{Properties of the LQ action}
  
In this section, we consider a LQ problem $(A,B,Q,T)$, satisfying the Kalman condition and such that $T<t^*$. Recalling Notation \ref{LQ_intermediate_action_and_cost}, we study the family of functionals $\{\mathcal{A}^{t,s}\}_{0\leq t<s\leq T}$ .
    First, we give the following definitions (see \cite[Definitions 7.11 and 7.13]{Villanioldandnew})
	\begin{Def}[Lagrangian action]\label{Lagrangian_action}
		Let $(\mathcal{X},d)$ be a Polish space. A Lagrangian action $\{\mathcal{A}^{t,s}\}_{0\leq t<s\leq T}$ on $(\mathcal{X},d)$ is a family of lower semi-continuous functionals on $\mathcal{C}([t,s],\mathcal{X})$, satisfying
		\begin{enumerate}
		\item\label{additivity} $\mathcal{A}^{\tau_1,\tau_2}(\gamma|_{[\tau_1,\tau_2]}) + \mathcal{A}^{\tau_2,\tau_3}(\gamma|_{[\tau_2,\tau_3]}) = \mathcal{A}^{\tau_1,\tau_3}(\gamma)$,\qquad $\forall\gamma\in\mathcal{C}([\tau_1,\tau_3],\mathcal{X}),$\quad$\forall \tau_1 < \tau_2 < \tau_3,$
		\item\label{reconstruction_property} $\mathcal{A}^{0,T}(\gamma)=\sup\limits_{\{N\in\mathbb{N}\}}\sup\limits_{\{0=\tau_0<\dots<\tau_N=T\}}\sum_{i=1}^{N}c^{\tau_{i-1},\tau_i}\left(\gamma(\tau_{i-1}),\gamma(\tau_i)\right)$,\quad$\forall\gamma\in\mathcal{C}([0,T],\mathcal{X}),$
		\end{enumerate}
		where, for all $0\leq t<s\leq T$,
		\begin{equation}
			c^{t,s}(x,y):=\inf\{\mathcal{A}^{t,s}(\gamma)\mid \gamma\in \mathcal{C}\left(\left[t,s\right],\mathcal{X}\right), \,  \gamma(t)=x,\, \gamma(s)=y\}	
		\end{equation}
		are the associated cost functions.
	\end{Def}
		Properties \eqref{additivity} and \eqref{reconstruction_property} of Definition \ref{Lagrangian_action} are referred to as \emph{additivity} and \emph{reconstruction property}, respectively. For uniformity with  Section \ref{LQ_control_problems}, we call ``optimal trajectories'' the action-minimizing curves.
	\begin{Def}[Coercive action]\label{Coercive_action}
		Let $\{\mathcal{A}^{t,s}\}_{0\leq t<s\leq T}$ be a Lagrangian action on a Polish space $(\mathcal{X},d)$. For any pair of times $0\leq t<s\leq T$ and any non-empty compact sets $K_t, K_s\subset\mathcal{X}$, define $\Gamma^{t,s}_{K_t\to K_s}\subset\mathcal{C}([t,s],\mathcal{X})$ as the set of action minimizing curves that start in $K_t$, at time $t$, and end in $K_s$, at time $s$. The Lagrangian action is called coercive if 
		\begin{enumerate}
			\item it is bounded from below, in the sense that
			\begin{equation}
				\inf\limits_{t<s}\inf\limits_{\gamma}\mathcal{A}^{t,s}(\gamma)>-\infty,
			\end{equation}
			\item if $K_t$ and $K_s$ are two non-empty compact sets such that $c^{t,s}(x,y)<+\infty$ for every $x\in K_t$ and $y\in K_s$, then $\Gamma^{t,s}_{K_t\to K_s}$ is compact and non-empty.
		\end{enumerate}
	\end{Def}
	The LQ action is not coercive, since it is not bounded from below. Nonetheless, it still satisfies the second property of Definition \ref{Coercive_action}. Before proving that, we need an auxiliary lemma concerning compactness properties of curves with bounded action.
\begin{lem}\label{coercivity_action_on_controls}
Let $(A,B,Q,T)$ be a LQ optimal control problem satisfying the Kalman condition, with $T<t^*$. Consider $0\leq t<s\leq T$ and let $K_t,K_s\subset \R^n$ be non-empty compact sets. Let $C\subset \mathcal{C}([t,s],\R^n)$ be a family of admissible curves with $\gamma(t)\in K_t$, $\gamma(s)\in K_s$ for all $\gamma \in C$. Then
\begin{equation}\label{eq:iffimplication}
 \sup_{\gamma\in C}   \mathcal{A}^{t,s}(\gamma) < +\infty \quad \iff \quad \sup_{\gamma_u \in C} \|u\| < +\infty.
\end{equation}
In addition, if $\{\gamma_{u_m}\}$ is a sequence of admissible trajectories with action uniformly bounded from above, then there exists an admissible trajectory $\gamma_u$ such that, up to extraction, $u_m \rightharpoonup u$ in the weak topology and $\gamma_m\to \gamma$ in the uniform topology.
\end{lem}
\begin{proof} We consider the case $[t,s]=[0,T]$, the general case is analogous (note that $s-t\leq T < t^*$).

We begin with a general observation from the Cauchy formula \eqref{Cauchy}. There is a constant $C_0>0$ such that for any admissible trajectory $\gamma_u$ it holds
\begin{equation}\label{estimatefromCauchy}
|\gamma_u(\tau)| \leq C_0 \left( |\gamma_u(0)| + \|u\|\right),\qquad \forall \tau\in [0,T].
\end{equation}

($\Leftarrow$) This direction is trivial. Using \eqref{estimatefromCauchy} we find $C_1>0$ such that for all $\gamma_u\in C$ it holds
\begin{equation}
   | \mathcal{A}(\gamma_u) |\leq C_1(1+\|u\|^2).
\end{equation}

($\Rightarrow$) This direction is also trivial if $-Q\geq 0$, as in this case $\|u\|^2\leq 2\mathcal{A}(\gamma_u)$ for all admissible trajectories. The general case is standard when $K_0, K_T$ are points, see \cite[Proof of Proposition 16.1]{agrachev2013control}. For general compact sets one needs an extra construction of a continuous selection of controls whose associated admissible trajectories have endpoints in $K_0, K_T$. We sketch the proof for completeness.

For $x,y\in\R^n$ denote by $\mathcal{U}(x,y)$ the set of all controls such that the corresponding admissible trajectory satisfies $\gamma_u(0)=x$ and $\gamma_u(T)=y$. Note that the subspace $\mathcal{U}(x,y)$ is affine, while $\mathcal{U}(0,0)$ is linear, moreover
\begin{equation}
    \mathcal{U}(x,y)=u_0 +\mathcal{U}(0,0),\qquad \text{for any }u_0\in\mathcal{U}(x,y).
\end{equation}

Under our assumption $T<t^*$, as proved in \cite[Section 16.4]{agrachev2013control}, there exists $\alpha>0$ such that
\begin{equation}\label{alphau00}
\mathcal{A}(\gamma_v) \geq \alpha \|v\|^2,\qquad \forall v \in \mathcal{U}(0,0).
\end{equation}

For $u\in \mathcal{U}(x,y)$ choose $u_{0}\in\mathcal{U}(x,y)$ so that $u=u_0+v$ for some $v\in\mathcal{U}(0,0)$. Using \eqref{estimatefromCauchy} we find constants $C_2,C_3$, not depending on the choice of $x \in K_0$, $y\in K_T$, and the representative $u_0\in \mathcal{U}(x,y)$, such that for all admissible trajectory $\gamma_u$ it holds
\begin{align}
\mathcal{A}(\gamma_u) & \geq -C_2(1+\|u_0\|^2) - C_3(1+\|u_0\|)\|v\|+ \mathcal{A}(v)  \\
& \geq  -C_2(1+\|u_0\|^2) - C_3(1+\|u_0\|)\|u\| +\alpha \|u\|^2, \label{eq:estimate1}
\end{align}
where in the second line we used \eqref{alphau00}, and possibly increased the constants.

We claim that we can choose $u_0\in\mathcal{U}(x,y)$ in such a way that its $L^2$ norm remains uniformly bounded as $(x,y)\in K_0\times K_T$. To prove the claim, define $F: \R^n\times L^2 \to \R^n\times \R^n$  by
\begin{equation}
F(x,u):=(x,\gamma_u(T,x)).
\end{equation}
Under the Kalman condition, $F$ is a surjective submersion (surjectivity is indeed Theorem \ref{thm:kalman}, while submersivity can be easily proven using Cauchy formula). Therefore, by the implicit function theorem, for any $(\bar{x},\bar{y})\in \R^n\times \R^n$ there are neighborhoods $V,W$ of $\bar{x}$ and $\bar{y}$, respectively, and a smooth map $G: V\times W \to L^2$ such that $G(x,y)$ is the control of an admissible trajectory from $x$ to $y$, for all $x\in V$, $y\in W$.

Using this choice in \eqref{eq:estimate1} we find constants $C_2',C_3'>0$ such that, for any admissible trajectory $\gamma_u$ with $\gamma_u(0) \in K_0$ and $\gamma_u(T)\in K_T$ it holds
\begin{equation}
    \mathcal{A}(\gamma_u) \geq - C_2' - C_3'\|u\| + \alpha\|u\|^2.
\end{equation}
The direction $(\Rightarrow)$  of \eqref{eq:iffimplication} easily follows.

To prove the final part of the statement, if $\{\gamma_m\}$ is a sequence of trajectories with action uniformly bounded from above (indeed these must be all admissible trajectories), then also the sequence of corresponding controls $\{u_m\}\subset L^2$ is bounded. Therefore up to extraction $u_n \rightharpoonup u$ for some $u\in L^2$. Again up to extraction we can assume that $\gamma_m(0) \in K_0$ is convergent. By Cauchy formula \eqref{Cauchy} it follows that $\gamma_m \to \gamma$ uniformly, where $\gamma$ is the admissible trajectory with control $u$ and with $\gamma(0) = \lim_m \gamma_m(0)$.
\end{proof}
    
	\begin{prop}\label{properties_LQ_action}
		Let $\{\mathcal{A}^{t,s}\}_{0\leq t<s \leq T}$ be a family of LQ actions, with $T<t^*$. Then, it is a Lagrangian action on the Polish space $(\mathbb{R}^n,d_{|\cdot|})$.
        
		Moreover, for $0\leq t<s \leq T$ and for every non-empty compact sets $K_t,K_s\subset\mathbb{R}^n$. the set $\Gamma^{t,s}_{K_t\to K_s}$ is compact and non-empty.
	\end{prop}
        \begin{rem}\label{rem_subadditivity_costs}
        As a consequence of previous proposition (see \cite[Proposition 7.16]{Villanioldandnew}), given a family of LQ action $\{\mathcal{A}^{t,s}\}_{0\leq t <s \leq T}$, the corresponding family of costs $\{c^{t,s}\}_{0\leq t <s \leq T}$ satisfies:
        \begin{equation}
				c^{\tau_1,\tau_3}(x,y)\leq c^{\tau_1,\tau_2}(x,z)+c^{\tau_2,\tau_3}(z,y),\qquad x,y,z\in\mathbb{R}^n,\qquad \tau_1<\tau_2<\tau_3
			\end{equation}
            with equality if and only if the point $z$ is the $\tau_2$-intermediate point of the unique optimal trajectory (with respect to the action $\mathcal{A}^{\tau_1,\tau_3}$) from $x$ to $y$.
        \end{rem}
	\begin{proof}
		We start by showing that $\mathcal{A}^{t,s}$ is lower semi-continuous on $\mathcal{C}([t,s],\mathbb{R}^n)$, with respect to the topology of uniform convergence.
		Consider a sequence of continuous curves $\{\gamma_m\}$, uniformly converging to some curve $\gamma_{\infty}$. If $\lim_{m}\mathcal{A}^{t,s}(\gamma_m)=+\infty$, there is nothing to prove. Therefore, up to passing to a subsequence, we may assume that all $\gamma_m$ are admissible, with control $u_m$, and
 \{$\mathcal{A}^{t,s}(\gamma_m)$\} is bounded. By Lemma \ref{coercivity_action_on_controls}, up to extraction, there is an admissible curve $\gamma\in \mathcal{C}([t,s],\mathbb{R}^n)$ with control $u$ such that $u_m\rightharpoonup  u$ weakly and $\gamma_m \to \gamma$ uniformly. This means that $\gamma = \gamma_\infty$ and is thus admissible with control $u_\infty=u$.
As a consequence, we have
		\begin{align}
			\mathcal{A}^{t,s}(\gamma_{\infty})&=\frac{1}{2}\|u_\infty\|^2-\frac{1}{2}\int_{t}^{s}\gamma_{\infty}^*(\tau)Q\gamma_{\infty}(\tau)d\tau\\
			&\leq\liminf_{m}\bigg(\frac{1}{2}\|u_m\|^2-\frac{1}{2}\int_{t}^{s}\gamma_m^*(\tau)Q\gamma_m(\tau)d\tau\bigg)=\liminf_{m}\mathcal{A}^{t,s}(\gamma_m),
		\end{align}
        by weak lower semi-continuity of the $L^2$ norm. This proves lower semi-continuity of the actions.

		Now, let $\tau_1<\tau_2<\tau_3$ and let $\gamma\in\mathcal{C}([\tau_1,\tau_3],\mathbb{R}^n)$. If both $\gamma|_{[\tau_1,\tau_2]}$ and $\gamma|_{[\tau_2,\tau_3]}$ are admissible, also $\gamma$ itself is, with control given by the concatenation of the controls of the restrictions. Thus,
		\begin{equation}\label{eq:additivityinproof}
			\mathcal{A}^{\tau_1,\tau_2}(\gamma|_{[\tau_1,\tau_2]}) + \mathcal{A}^{\tau_2,\tau_3}(\gamma|_{[\tau_2,\tau_3]}) = \mathcal{A}^{\tau_1,\tau_3}(\gamma)
		\end{equation}
		trivially holds. Otherwise, if either one of the two restrictions is not admissible, neither is $\gamma$, so that \eqref{eq:additivityinproof} still holds.
        
		We prove the reconstruction property. Let $\gamma\in\mathcal{C}([0,T],\mathbb{R}^n)$ and let $\{0=\tau_0<\ldots<\tau_N=T\}$ be a partition of $[0,T]$, with $N\in\mathbb{N}$. By definition of the cost functions and additivity, we have
		\begin{equation}
			\mathcal{A}^{0,T}(\gamma)=\sum_{i=1}^{N}\mathcal{A}^{\tau_{i-1},\tau_i}(\gamma|_{[\tau_{i-1},\tau_i]})\geq\sum_{i=1}^{N}c^{\tau_{i-1},\tau_i}\left(\gamma(\tau_{i-1}),\gamma(\tau_i)\right),
		\end{equation}
		so that, passing to the supremum, we have 
		\begin{equation}
			\mathcal{A}^{0,T}(\gamma)\geq\sup\limits_{\{N\in\mathbb{N}\}}\sup\limits_{\{0=\tau_0<\dots<\tau_N=T\}}\sum_{i=1}^{N}c^{\tau_{i-1},\tau_i}\left(\gamma(\tau_{i-1}),\gamma(\tau_i)\right).
		\end{equation}
		For the converse inequality, for every $m\in\mathbb{N}$, we consider the partition $\mathcal{P}_m=\{\tau_i^{(m)}=\frac{iT}{2^m}\mid i=0,\dots,2^m\}$ and we define the curve $\gamma_m$ as
		\begin{itemize}
			\item $\gamma_m(\tau_i^{(m)})=\gamma(\tau_i^{(m)}),\qquad i=0,\dots, 2^m$,
			\item $\gamma_m|_{[\tau_{i-1}^{(m)},\tau_i^{(m)}]}$\quad is the unique optimal trajectory for $\mathcal{A}^{\tau_{i-1}^{(m)},\tau_i^{(m)}}$ between its endpoints, for $ i=0,\dots, 2^m$.
		\end{itemize}
		The curve $\gamma_m$ is admissible, with control $u_m$ obtained from the concatenation of optimal controls on each sub-interval. Note that
        \begin{equation}
			\mathcal{A}^{0,T}(\gamma_m)=\sum_{i=1}^{2^m}c^{\tau_{i-1}^{(m)},\tau_i^{(m)}}(\gamma(\tau_{i-1}^{(m)}),\gamma(\tau_i^{(m)}))\leq \mathcal{A}^{0,T}(\gamma).
		\end{equation}
        We distinguish between two cases. If $\gamma$ is admissible, then $\{\mathcal{A}^{0,T}(\gamma_m)\}$ is bounded. By Lemma \ref{coercivity_action_on_controls}, up to extraction, there is an admissible curve $\gamma_\infty \in \mathcal{C}([0,T],\R^n)$ such that $\gamma_m \to \gamma_\infty$ uniformly.

		By construction of the sequence $\{\gamma_m\}$ and passing to the limit, we have that 
		\begin{equation}
			\gamma_{\infty}(t_i^{(j)})=\gamma(t_i^{(j)}),\qquad\forall j\in\mathbb{N},\quad i=0,\dots,2^j,
		\end{equation}
		so that by continuity, $\gamma=\gamma_{\infty}$ and 
		\begin{equation}
			\mathcal{A}^{0,T}(\gamma_{\infty})\leq\liminf\limits_{m}\mathcal{A}^{0,T}(\gamma_m)\leq\sup\limits_{\{N\in\mathbb{N}\}}\sup\limits_{\{0=\tau_0<\dots<\tau_N=T\}}\sum_{i=1}^{N}c^{\tau_{i-1},\tau_i}\left(\gamma(\tau_{i-1}),\gamma(\tau_i)\right).
		\end{equation}
        
		The second case is when $\gamma$ is not admissible, that is $\mathcal{A}^{0,T}(\gamma)=+\infty$. Again, by Lemma \ref{coercivity_action_on_controls} we must have that
        \begin{equation}
            \lim_m \mathcal{A}^{0,T}(\gamma_m) =  +\infty,
        \end{equation}
		completing the proof of the reconstruction property.

        We prove the compactness statement. Again, it is sufficient to prove it for $[t,s]=[0,T]$, as the general case can be proven analogously. Consider two non-empty compact sets $K_0,K_T\subset\mathbb{R}^n$. Since between any two points there exist a (unique) minimizing trajectory, the set $\Gamma^{0,T}_{K_0\to K_T}$ is non-empty. Let $\{\gamma_m\}\subset\Gamma^{0,T}_{K_0\to K_T}$. Since these trajectories are optimal it holds
        \begin{equation}
            \mathcal{A}^{0,T}(\gamma_m) = c^{0,T}(\gamma_m(0),\gamma_m(T))\leq \sup_{(x,y)\in K_0\times K_T} c(x,y) < +\infty, \qquad \forall m.
        \end{equation}
        By Lemma \ref{coercivity_action_on_controls}, up to extraction there is an admissible trajectory $\gamma \in \mathcal{C}([0,T],\R^n)$ such that $\gamma_m \to \gamma$ uniformly. Indeed $\gamma(0)\in K_0$ and $\gamma(T)\in K_T$. Furthermore, since the LQ cost is continuous and the action is lower semi-continuos we have
        \begin{equation}
         c^{0,T}(\gamma(0),\gamma(T)) = \lim_m c^{0,T}(\gamma_m(0),\gamma_m(T)) = \liminf_{m}\mathcal{A}^{0,T}(\gamma_m) \geq \mathcal{A}^{0,T}(\gamma).
        \end{equation}
        This means that $\gamma$ is optimal and so $\gamma\in \Gamma^{0,T}_{K_0\to K_T}$, proving that the latter set is compact.
	\end{proof}
\subsection{Displacement interpolations}   

Now, we address the problem of interpolating continuously between the initial measure $\mu$ and the final measure $\nu$, in an optimal way with respect to the LQ cost. More precisely, we lift the LQ actions $\{\mathcal{A}^{t,s}\}_{0\leq t<s\leq T}$ to a suitable action on $(P_2(\mathbb{R}^n),W_2)$ that selects the displacement interpolations for the LQ problem. In addition, we prove a suitable version of the classical representation theorem for such optimal curves.

        \begin{Def}[Selection maps]\label{def_selection_maps}
        Let $(A,B,Q,T)$ be a LQ optimal control problem satisfying the Kalman condition, with $T<t^*$. 
        Then, for every $0\leq t<s\leq T$, we define the map $S_{t\to s}:\mathbb{R}^n\times\mathbb{R}^n\to\mathcal{C}([t,s],\mathbb{R}^n)$ as
        \begin{equation}
            (x,y)\mapsto \bigg([t,s]\ni\tau\mapsto S_{t\to s}(x,y)(\tau):=\exp_{x,\tau-t}(\exp_{s-t}^{-1}(y))\bigg),\qquad\forall x,y \in \R^n.
        \end{equation}
        Similarly, for every $0\leq t<s\leq T$, we define the map 
        $\opt_{t,s}:\mathcal{C}([t,s],\mathbb{R}^n)\to\mathcal{C}([t,s],\mathbb{R}^n)$ as
        \begin{equation}
            \gamma\mapsto \bigg([t,s]\ni\tau\mapsto \opt_{t,s}(\gamma)(\tau):=\exp_{\gamma(t),\tau-t}(\exp_{s-t}^{-1}(\gamma(s)))\bigg),\qquad\forall\gamma\in\mathcal{C}([t,s],\mathbb{R}^n).
        \end{equation}
        \end{Def}
        The application of the maps $S_{t\to s}$ and $\opt_{t,s}$ can be regarded as an optimization procedure. $S_{t\to s}$ associates to any pair of points the unique trajectory that joins them in time $s-t$ and is optimal for the action $\mathcal{A}_{t,s}$. Similarly, $\opt_{t,s}$ associates to any continuous trajectory the corresponding one that is optimal on the same time interval, between the same endpoints. Therefore, it is natural to consider dynamical objects that posses some form of invariance under their action.
        
        Thanks to Theorem \ref{LQ_exponential}, the maps $S_{t\to s}$ and $\opt_{t,s}$ are continuous, with the natural topologies of source and target, so that they can be used to push forward Borel measures.
	\begin{Def}[Optimal dynamical plan]
		Let $(A,B,Q,T)$ be a LQ optimal control problem satisfying the Kalman condition, with $T<t^*$. Let $\mu,\nu\in P(\mathbb{R}^n)$. A Borel measure $\Pi\in P(\mathcal{C}([0,T],\mathbb{R}^n))$ is called a optimal dynamical plan between $\mu$ and $\nu$ if it is concentrated on optimal trajectories and
		\begin{align}
			(e_0,e_T)_{\sharp}\Pi=\alpha,
		\end{align}
		where $e_t$ is the evaluation map at time $t$ and $\alpha$ is an optimal transport plan between $\mu$ and $\nu$ for the LQ cost $c=c^{0,T}$.
	\end{Def}

	We denote by $(\mu_\tau)$ a curve of measures $[0,T]\ni \tau\mapsto\mu_\tau$, and by $\mu_\tau$ its evaluation at time $\tau$.

    \begin{Def}[Displacement interpolation]\label{def_displ_interp}
        Given $(\mu_\tau)\in\mathcal{C}([0,T],P_2(\mathbb{R}^n))$, we say that $(\mu_\tau)$ is a displacement interpolation with respect to the costs $\{c^{t,s}\}_{0\leq t,s\leq T}$, if 
		\begin{equation}\label{uguaglianzacosti}
			 C^{\tau_1,\tau_2}(\mu_{\tau_1},\mu_{\tau_2})+C^{\tau_2,\tau_3}(\mu_{\tau_2},\mu_{\tau_3})=C^{\tau_1,\tau_3}(\mu_{\tau_1},\mu_{\tau_3}), \qquad\forall \tau_1 < \tau_2 < \tau_3,
		\end{equation}
		where $C^{t,s}(\mu,\nu)$ is the Kantorovich transport cost between $\mu$ and $\nu$, with respect to the pointwise intermediate cost $c^{t,s}$.
    \end{Def}
    \begin{rem}
For general $\mu \in P(\R^n)$, the cost $C^{t,s}(\mu)$ can attain the values $\pm \infty$. Therefore, to give a meaning to \eqref{uguaglianzacosti}, we \emph{require} in Definition \ref{def_displ_interp} that $\mu_\tau \in P_2(\R^n)$ for all $\tau$ (and not just $\tau=0,1$), in which case $C^{t,s}$ is finite for $t,s \in [0,T]$. This is not necessary when the cost functional is bounded from below.
    \end{rem}
    \begin{prop}
        		Let $(A,B,Q,T)$ be a LQ optimal control problem satisfying the Kalman condition, with $T<t^*$. Let $\mu,\nu\in P_2(\mathbb{R}^n)$ and let $\Pi\in P(\mathcal{C}([0,T],\mathbb{R}^n))$ be an optimal dynamical plan between $\mu$ and $\nu$. Then, the curve of measures $(\mu_\tau):=(e_{\tau\sharp}\Pi)$ is a displacement interpolation between $\mu$ and $\nu$.
    \end{prop}
    \begin{proof}
        Define $\alpha=(e_0,e_T)_\sharp\Pi\in\Adm(\mu,\nu)$. Note that $\alpha$ is optimal by definition of optimal dynamical plan and is in $P_2(\mathbb{R}^{n}\times \R^n)$. Recalling Definition \ref{def_selection_maps}, one has that
			\begin{equation}
				S_{0\to T}\circ(e_0,e_T)=\opt_{0,T},
			\end{equation}
            so that $\Pi=(S_{0\to T})_\sharp\alpha$, where we used that $\Pi$ is concentrated on the set of optimal trajectories $\Gamma$, where $\opt_{0,T}=\mathrm{id}$.
        By Remark \ref{formula_opt_traj}, we have 
			\begin{equation}
				S_{0\to T}(x,y)=\left(\tau\mapsto\exp_{x,\tau}(\exp_{x,T}^{-1}(y))=M(\tau)y+N(\tau)x\right),\qquad\forall x,y\in\mathbb{R}^n,\quad\forall \tau\in[0,T],
			\end{equation}
			where $M(\tau):=R_3(\tau)R_3^{-1}(T)$ and $N(t):=R_4(\tau)-R_3(\tau)R_3^{-1}(T)R_4(T)$ are smooth curves with values in $\mathbb{R}^{n\times n}$. In particular, there exists $\Lambda>0$, such that
			\begin{equation}\label{dominante_integrabile}
				|\exp_{x,\tau}(\exp_{x,T}^{-1}(y))|^2\leq \Lambda(|x|^2+|y|^2),\qquad\forall x,y\in\mathbb{R}^n,\quad\forall \tau\in[0,T]. 
			\end{equation}
			Since $\alpha\in P_2(\mathbb{R}^n\times\mathbb{R}^n)$, which is equivalent to $\Lambda(|x|^2+|y|^2)\in L^1(\alpha)$, we have that 
			\begin{align}
            \int_{\mathbb{R}^n}|z|^2d\mu_\tau(z)&=\int_{\Gamma}|\gamma(t)|^2d\Pi(\gamma)\\
				&=\int_{\mathbb{R}^{2n}}|\exp_{x,\tau}(\exp_{x,T}^{-1}(y))|^2d\alpha(x,y)\\
				&\leq \Lambda\int_{\mathbb{R}^{2n}}\big(|x|^2+|y|^2\big)d\alpha(x,y)<+\infty.
			\end{align} 
			This means that $\mu_\tau\in P_2(\mathbb{R}^n)$, for every $\tau\in[0,T]$.
            
			Now, let $h\in\mathcal{C}_b(\mathbb{R}^n)$, (a continuous bounded function) and we have
			\begin{align}
				\int_{\mathbb{R}^n}h(z)d\mu_\tau(z)&=\int_{\Gamma}h(\gamma(t))d\Pi(\gamma)\\
					&=\int_{\mathbb{R}^{2n}}h(\exp_{x,\tau}(\exp_{x,T}^{-1}(y)))d\alpha(x,y)\\
					&=\int_{\mathbb{R}^{2n}}h_t(x,y)d\alpha(x,y),
			\end{align} 
			with $h_t(x,y):=h(\exp_{x,\tau}(\exp_{x,T}^{-1}(y)))$.
			We notice that
			\begin{itemize}
				\item $\|h_\tau\|_{\infty}=\|h\|_{\infty}=const, \qquad\forall \tau\in[0,T]$,
				\item for every sequence $\{\tau_m\}$ such that $\tau_m\to \tau_{\infty}$ and every $x,y\in\mathbb{R}^{n}$, $h_{\tau_m}(x,y)$ converges to $h_{\tau_{\infty}}(x,y)$, by continuity of the LQ exponential map.
			\end{itemize} 
			As a consequence, by dominated convergence, it follows that
			\begin{align}
				\lim\limits_{m\to+\infty}\int_{\mathbb{R}^n}h(z)d\mu_{\tau_m}(z)&=\lim\limits_{m\to+\infty}\int_{\mathbb{R}^{2n}}h_{\tau_m}(x,y)d\alpha(x,y)\\
				&=\int_{\mathbb{R}^{2n}}h_{\tau_{\infty}}(x,y)d\alpha(x,y)\\
				&=\int_{\mathbb{R}^n}h(z)d\mu_{\tau_{\infty}}(z).\label{eq:weak}
			\end{align}
			Moreover, using \eqref{dominante_integrabile}, we can apply again dominated convergence to get
			\begin{align}
		  \lim\limits_{m\to+\infty}\int_{\mathbb{R}^n}|z|^2d\mu_{\tau_m}(z)&=\lim\limits_{m\to+\infty}\int_{\mathbb{R}^{2n}}|\exp_{x,\tau_m}(\exp_{x,T}^{-1}(y))|^2d\alpha(x,y)\\
				&=\int_{\mathbb{R}^{2n}}|\exp_{x,\tau_{\infty}}(\exp_{x,T}^{-1}(y))|^2d\alpha(x,y)\\
				&=\int_{\mathbb{R}^n}|z|^2d\mu_{\tau_{\infty}}(z). \label{2momentum}
			\end{align}
            It is well-known that weak convergence \eqref{eq:weak} and convergence of $2$-momentum \eqref{2momentum} is equivalent to convergence $\mu_{\tau_m} \to \mu_{\tau_\infty}$ in the $W_2$-topology \cite[Theorem 6.9]{Villanioldandnew}. This concludes  the proof that $\mu_\tau\in\mathcal{C}([0,T],P_2(\mathbb{R}^n))$.
			
		Note that all the Kantorovich transport costs are finite by Corollary \ref{finite_total_cost}.
			First we claim that, for every $(\eta_\tau)\in \mathcal{C}([0,T],P_2(\mathbb{R}^n))$ and every $0\leq \tau_1 < \tau_2 < \tau_3\leq T$, we have that
			\begin{equation}\label{monotonicity_refining_costs}
				C^{\tau_1,\tau_3}(\eta_{\tau_1},\eta_{\tau_3})\leq C^{\tau_1,\tau_2}(\eta_{\tau_1},\eta_{\tau_2})+C^{\tau_2,\tau_3}(\eta_{\tau_2},\eta_{\tau_3}).
			\end{equation}
			In order to prove previous equation, we consider an optimal plan $\alpha_{1,2}$ between $\eta_{\tau_1}$ and $\eta_{\tau_2}$ for the cost $c^{\tau_1,\tau_2}$ and an optimal plan $\alpha_{2,3}$ between $\eta_{\tau_2}$ and $\eta_{\tau_3}$ for the cost $c^{\tau_2,\tau_3}$. By the gluing Lemma (see e.g. \cite[Lemma 3.1]{Usersguide}), there exists a probability measure $\alpha_{1,2,3}$ on $\mathbb{R}^n\times\mathbb{R}^n\times\mathbb{R}^n$, whose appropriate projections are $\alpha_{1,2}$ and $\alpha_{2,3}$. If we define $\alpha_{1,3}=\pi_{1,3\sharp}\alpha_{1,2,3}$, this is a transport plan between $\eta_{t_1}$ and $\eta_{t_3}$.
            
			Recalling Remark \ref{rem_subadditivity_costs}, we have
			\begin{align}\label{key_computation_opt_coupling}
				C^{\tau_1,\tau_3}(\eta_{\tau_1},\eta_{\tau_3})&\leq\int_{\mathbb{R}^{2n}}c^{\tau_1,\tau_3}(x_1,x_3)d\alpha_{1,3}(x_1,x_3)\\
				&=\int_{\mathbb{R}^{3n}}c^{\tau_1,\tau_3}(x_1,x_3)d\alpha_{1,2,3}(x_1,x_2,x_3) \nonumber\\
				&\leq\int_{\mathbb{R}^{3n}}\left(c^{\tau_1,\tau_2}(x_1,x_2)+c^{\tau_2,\tau_3}(x_2,x_3)\right)d\alpha_{1,2,3}(x_1,x_2,x_3) \nonumber\\
				&=C^{\tau_1,\tau_2}(\eta_{\tau_1},\eta_{\tau_2})+C^{\tau_2,\tau_3}(\eta_{\tau_2},\eta_{\tau_3}).\nonumber
			\end{align}
		For the converse inequality, given $(\mu_\tau)=(e_{\tau\sharp}\Pi)$ and recalling that $\Pi$ is concentrated on the set of optimal curves $\Gamma$, we have
			\begin{align}
				&	C^{\tau_1,\tau_2}(\mu_{\tau_1},\mu_{\tau_2})+C^{\tau_2,\tau_3}(\mu_{\tau_2},\mu_{\tau_3})\\
				&\leq\int_{\mathbb{R}^{2n}}c^{\tau_1,\tau_2}(x_1,x_2)d\left((e_{\tau_1},e_{\tau_2})_{\sharp}\Pi\right)(x_1,x_2)+\int_{\mathbb{R}^{2n}}c^{\tau_2,\tau_3}(x_2,x_3)d\left((e_{\tau_2},e_{\tau_3})_{\sharp}\Pi\right)(x_2,x_3)\\
				&=\int_{\Gamma}\left(c^{\tau_1,\tau_2}(\gamma(\tau_1),\gamma(\tau_2))+c^{\tau_2,\tau_3}(\gamma(\tau_2),\gamma(\tau_3))\right)d\Pi(\gamma)\\
				&=\int_{\Gamma}c^{\tau_1,\tau_3}(\gamma(\tau_1),\gamma(\tau_3))d\Pi(\gamma)=\int_{\mathbb{R}^{2n}}c^{\tau_1,\tau_3}(x_1,x_3)d\left((e_{\tau_1},e_{\tau_3})_{\sharp}\Pi\right)(x_1,x_3),
			\end{align}
			where we used Remark \ref{rem_subadditivity_costs}.
			If we take $\tau_1=0,\tau_2=\tau$ and $\tau_3=T$, since $(e_0,e_T)_{\sharp}\Pi$ is an optimal plan, we get
			\begin{equation}
				C^{0,\tau}(\mu_0,\mu_\tau)+C^{\tau,T}(\mu_\tau,\mu_T)=C^{0,T}(\mu_0,\mu_T),\qquad\forall 0<\tau<T,
			\end{equation}
			and in the previous chain, all inequalities must be equalities. In particular, this implies that $(e_0,e_\tau)_{\sharp}\Pi$ is an optimal plan between $\mu$ and $\mu_\tau$ for the cost $c^{0,\tau}$, for every $\tau\in(0,T]$, and similarly for $(e_\tau,e_T)_{\sharp}\Pi$.
            
			The same argument applied to $\tau_1=0,\tau_2,\tau_3$ proves that also $(e_{\tau_2},e_{\tau_3})_{\sharp}\Pi$ is an optimal plan between $\mu_{\tau_2}$ and $\mu_{\tau_3}$ for the cost $c^{\tau_2,\tau_3}$ and for all $0\leq \tau_2 <\tau_3\leq T$. Finally, plugging this information back into previous equation, we get 
			\begin{equation}
				C^{\tau_1,\tau_2}(\mu_{\tau_1},\mu_{\tau_2})+C^{\tau_2,\tau_3}(\mu_{\tau_2},\mu_{\tau_3})\leq C^{\tau_1,\tau_3}(\mu_{\tau_1},\mu_{\tau_3}),
			\end{equation}
			proving that, indeed, $(\mu_\tau)$ is a displacement interpolation.
    \end{proof}
	
	The purpose of the next theorem is to clarify the relation between optimal dynamical plans and displacement interpolations and to link them to the action minimizers of a suitable lift of the Lagrangian action to the space $(P_2(\mathbb{R}^n),W_2)$. The proof is essentially the same as the one of \cite[Theorem 7.21]{Villanioldandnew}, with some additional care to be taken.The maps $S_{t\to s}$ and $\opt_{t,s}$ are single-valued. As a consequence, we obtain a bijection between optimal transport plans and optimal dynamical plans.
	
	\begin{thm}[Displacement interpolations]\label{dyn_opt_coup_characterization}
		Let $(A,B,Q,T)$ be a LQ optimal control problem satisfying the Kalman condition, with $T<t^*$. Let $\mu, \nu\in P_2(\mathbb{R}^n)$. There is a bijection between:
        \begin{itemize}
            \item optimal plans $\alpha\in P_2(\mathbb{R}^n\times\mathbb{R}^n)$,
            \item optimal dynamical plans $\Pi\in P(\mathcal{C}([0,T],\mathbb{R}^n))$,
            \item displacement interpolations $(\mu_\tau)\in \mathcal{C}([0,T],P_2(\mathbb{R}^n))$.
        \end{itemize}
 In particular, for every optimal dynamical plan $\Pi\in P(\mathcal{C}([0,T],\mathbb{R}^n))$, the corresponding displacement interpolation is given by
		\begin{equation}
			e_{\tau\sharp}\Pi=\mu_\tau,\qquad \forall \tau\in[0,T].
		\end{equation}
		Moreover, for $0\leq t<s\leq T$, we define 
		\begin{equation}
			\mathbb{A}^{t,s}(\eta_{\tau}):=\sup\limits_{\{N\in\mathbb{N}\}}\sup\limits_{\{t=\tau_0<\dots<\tau_N=s\}}\sum_{i=1}^{N}C^{\tau_{i-1},\tau_i}\left(\eta_{\tau_{i-1}},\eta_{\tau_i}\right),\qquad(\eta_{\tau})\in\mathcal{C}([t,s],P_2(\mathbb{R}^n)),
		\end{equation} 
        where $C^{t,s}(\mu,\nu)$ is the Kantorovich transport cost, associated with the pointwise intermediate cost $c^{t,s}:\R^n\times \R^n \to \R$.
		Then, the family of functionals $\{\mathbb{A}^{t,s}\}_{0\leq t<s\leq T}$ on $P_2(\mathbb{R}^n)$ is a Lagrangian action. The corresponding family of action-minimizing costs coincides with the Kantorovich transport costs $\{C^{t,s}\}_{0\leq t<s\leq T}$. Furthermore, $(\mu_\tau)$ is a displacement interpolation between $\mu$ and $\nu$ if and only if it is an action minimizer.
	\end{thm}

        \begin{proof}
            First, we prove that there is a bijection between optimal transport plans $\alpha\in P_2(\mathbb{R}^n\times\mathbb{R}^n)$ and optimal dynamical plans $\Pi\in P(\mathcal{C}([0,T],\mathbb{R}^n))$.

			Let $\mu, \nu\in P_2(\mathbb{R}^n)$ and let $\alpha\in P_2(\mathbb{R}^n\times\mathbb{R}^n)$ be an optimal transport plan for the Kantorovich problem with LQ cost. 
			Recalling Definition \ref{def_selection_maps}, one has that
			\begin{align}
				\opt_{0,T}\circ S_{0\to T}&=S_{0\to T},\\
				(e_0,e_T)\circ  S_{0\to T}&=\mathrm{id}_{\mathbb{R}^n\times\mathbb{R}^n},\\
				S_{0\to T}\circ(e_0,e_T)&=\opt_{0,T}.
			\end{align}
			This suggests to define $\Pi:=(S_{0\to T})_{\sharp}\alpha$, so that
			\begin{gather}
				(e_0,e_T)_{\sharp}\Pi=(e_0,e_T)_{\sharp}\circ(S_{0\to T})_{\sharp}\alpha=\alpha,\\
				(\opt_{0,T})_{\sharp}\Pi=(\opt_{0,T})_{\sharp}\circ(S_{0\to T})_{\sharp}\alpha=(S_{0\to T})_{\sharp}\alpha=\Pi,
			\end{gather}
			hence $\Pi$ is a optimal dynamical plan between $\mu$ and $\nu$.
			If $\alpha_1$ and $\alpha_2$ are both optimal transport plans and $(S_{0\to T})_{\sharp}\alpha_1=(S_{0\to T})_{\sharp}\alpha_2$, then 
			\begin{equation}
				\alpha_1=(e_0,e_T)_{\sharp}\circ(S_{0\to T})_{\sharp}\alpha_1=(e_0,e_T)_{\sharp}\circ(S_{0\to T})_{\sharp}\alpha_2=\alpha_2,
			\end{equation}
			therefore, the relation between optimal dynamical plans and optimal transport plans is a bijection. 
            The rest of the proof follows analogously to the one of \cite[Theorem 7.21]{Villanioldandnew}, with the caveat that we need the curve of measures to lie in $P_2$, in order to lift it to an optimal dynamical plan. In this way, all the Kantorovich costs are finite and the usual dyadic approximation scheme produces a weakly compact sequence of dynamical plans.  
        \end{proof}
        
        We record a corollary for compactly supported and absolutely continuous measures. 
        \begin{cor}\label{cor_dyn_opt_coup}
		Let $\mu, \nu\in P_2(\mathbb{R}^n)$. Then
		\begin{itemize}
			\item if $\mu\ll\mathscr{L}^n$, there exists a unique optimal transport plan $\alpha\in P_2(\mathbb{R}^n\times\mathbb{R}^n)$, a unique optimal dynamical plan $\Pi\in P(\mathcal{C}([0,T],\mathbb{R}^n))$ and a unique displacement interpolation $(\mu_t)\in \mathcal{C}([0,T],P_2(\mathbb{R}^n))$, between $\mu$ and $\nu$.
			\item if $\mu,\nu\in P_c(\mathbb{R}^n)$, then there exists a compact set $K\subset\mathbb{R}^n$, such that, for every displacement interpolation $(\mu_\tau)\in \mathcal{C}([0,T],P_2(\mathbb{R}^n))$ joining $\mu$ and $\nu$, it holds
			\begin{equation}
				\supp(\mu_\tau)\subset K,\qquad\forall \tau\in[0,T].
			\end{equation}
		\end{itemize}
	\end{cor}
		\begin{proof}
			If $\mu, \nu\in P_2(\mathbb{R}^n)$, with $\mu\ll\mathscr{L}^n$, by Theorem \ref{Brenier_McCann} there exists a unique optimal transport plan between $\mu$ and $\nu$. We conclude by Theorem \ref{dyn_opt_coup_characterization}. 
			
			If $\mu,\nu\in P_c(\mathbb{R}^n)$, we define $K_0=\supp(\mu)$ and $K_T=\supp(\nu)$. Then, any optimal dynamical plan $\Pi$ between $\mu$ and $\nu$ is supported on the compact non-empty set $\Gamma^{0,T}_{K_0\to K_T}$ (see Proposition \ref{properties_LQ_action}). Additionally, the map $F:[0,T]\times\Gamma^{0,T}_{K_0\to K_T}\to\mathbb{R}^n$, defined by
			\begin{equation}
				F(\tau,\gamma):=\gamma(\tau),\qquad (\tau,\gamma)\in [0,T]\times \Gamma^{0,T}_{K_0\to K_T},
			\end{equation}
			is continuous from a compact space. Therefore, $K:=\mathrm{Im}(F)$ is compact and, given $\mu_\tau=e_{\tau\sharp}\Pi$, we have $\supp(\mu_\tau)=\supp(e_{\tau\sharp}\Pi)\subset K$, for every $\tau\in[0,T]$.
		\end{proof}
\subsection{Absolute continuity} 

Absolute continuity for displacement interpolations are proven e.g.\ in \cite[Theorem 5.1]{FathiFigalli} for Lagrangian actions with Tonelli Lagrangians (that does not include the LQ case). The case of more general Lagrangian actions including the LQ cases with $-Q\geq 0$ is discussed in \cite[Theorem 4.1]{agrachev2007opttransnonholconst}. These arguments can be applied nonetheless to the general case. We include a self-contained proof.

	\begin{thm}[Absolute continuity]\label{absolute_continuity-dyn_opt_coupl}
      Let $(A,B,Q,T)$ be a LQ optimal control problem satisfying the Kalman condition, with $T<t^*$. Let $\mu, \nu\in P_2(\mathbb{R}^n)$, with $\mu\ll\mathscr{L}^n$, and let $(\mu_\tau)=(e_{\tau\sharp}\Pi)\in\mathcal{C}([0,T],P_2(\mathbb{R}^n))$ be the unique displacement interpolation joining them. Then, $\mu_\tau\ll\mathscr{L}^n$, for every $\tau\in[0,T)$. 

        Additionally, let $\psi$ be a Kantorovich potential from $\mu$ to $\nu$ for the cost $c^{0,T}$. Then, the optimal transport map $T_\tau:\mathbb{R}^n\to\mathbb{R}^n$, from $\mu$ to $\mu_\tau$ for the cost $c^{0,\tau}$ is given by
        \begin{equation}
			T_\tau(x)=\exp_{x,\tau}(\nabla\psi(x)),\qquad\mu\text{-almost every }x\in\mathbb{R}^n, \quad\forall \tau\in[0,T].
		\end{equation}
        In particular, $T_\tau$ is $\mu$-almost everywhere differentiable.
	\end{thm}
	\begin{proof}
		First, we apply Theorems \ref{Brenier_McCann} and \ref{dyn_opt_coup_characterization}.
		Since $\mu\ll\mathscr{L}^n$, there is a unique displacement interpolation $(\mu_\tau)\in\mathcal{C}([0,T],P_2(\mathbb{R}^n))$ joining it to $\nu$ in time $T$, and such curve is obtained as $(e_{\tau\sharp}\Pi)$, for a unique optimal dynamical plan $\Pi\in P(\mathcal{C}([0,T],\mathbb{R}^n))$. Moreover there exists a unique optimal map $T:\mathbb{R}^n\to\mathbb{R}^n$, from $\mu$ to $\nu$ of the form:
            \begin{equation}
			T(x)=\exp_{x,T}(\nabla\psi(x)),\qquad\mu\text{-almost every }x\in\mathbb{R}^n,
		\end{equation}
        where $\psi:\mathbb{R}^n\to\mathbb{R}^n$ is a $c^{0,T}$-convex Kantorovich potential. The corresponding optimal plan $\alpha_{0,T}$ satisfies $(S_{0\to T})_{\sharp}\alpha_{0,T}=\Pi$ (see Corollary \ref{cor_dyn_opt_coup}). Fix $\tau\in [0,T)$. We get that
		\begin{align}
			\mu_\tau &=\pi_{2\sharp}\circ(e_0,e_\tau)_{\sharp}\Pi\\
			&=\pi_{2\sharp}\circ(e_0,e_\tau)_{\sharp}\circ(S_{0\to T})_{\sharp}\alpha_{0,T}\\
			&=\pi_{2\sharp}\circ(e_0,e_\tau)_{\sharp}\circ(S_{0\to T})_{\sharp}\circ(\mathrm{id},T)_{\sharp}\mu = (T_\tau)_{\sharp}\mu,
		\end{align}
        where $T_\tau:\R^n \to \R^n$ is a transport map, between $\mu$ and $\mu_\tau$, defined by
		\begin{equation}\label{formula_T_t}
			T_\tau(x)=\exp_{x,\tau}(\nabla\psi(x)),\qquad\mu\text{-almost every }x\in\mathbb{R}^n.
		\end{equation}
        Note that since $\psi$ is semi-convex (see Proposition \ref{Kant_potentials_are_semi-convex}), thus $\psi$ is twice-differentiable and $T_\tau$ is differentiable $\mu$-almost everywhere (the differentiability set is independent on the choice of $\tau$). Furthermore, $\tau\mapsto T_\tau(x)$ is the unique optimal curve joining $x$ and $T(x)$ on $[0,T]$.
		
		Let $E\subset\supp(\mu)$ be a Borel set of full $\mu$-measure such that $\psi$ is differentiable on $E$ and \eqref{formula_T_t} holds. We also define $E_\tau=T_\tau(E)$, so that $E_\tau$ is of full $\mu_\tau$-measure.

		By definition and by optimality of the map $T$, we have
		\begin{equation}\label{first_support_estimate}
			\psi^c(y)-\psi(x)\leq c^{0,T}(x,y),\qquad\forall x,y\in\R^n ,
		\end{equation}
		with equality for $x\in E$ and $y=T(x)$.
		Recall that (see Remark \ref{rem_subadditivity_costs})
		\begin{equation}\label{second_support_estimate}
			c^{0,T}(x,y)\leq c^{0,\tau}(x,z)+c^{\tau,T}(z,y),\qquad \forall x,y,z\in\mathbb{R}^n,
		\end{equation}
		with equality if and only if $z$ is the point at time $\tau$ of the unique optimal trajectory joining $x$ and $y$, namely $z=\exp_{x,\tau}\circ\exp_{x,T}^{-1}(y)$. Pairing \eqref{first_support_estimate} with \eqref{second_support_estimate}, we get
		\begin{equation}\label{third_support_estimate}
			\psi^c(y)-c^{\tau,T}(z,y)\leq\psi(x)+c^{0,\tau}(x,z),\qquad \forall x,y,z\in\R^n,
		\end{equation}
		with equality for $x\in E$, $z=T_\tau(x)$ and $y=T(x)$
		Taking suprema and infima in \eqref{third_support_estimate}, we define
		\begin{align}
			f_\tau(z):&=\sup\{\psi^c(y)-c^{\tau,T}(z,y)\mid y\in \R^n\},\qquad\forall z\in\mathbb{R}^n,\\
			g_\tau(z):&=\inf\{\psi(x)+c^{0,\tau}(x,z)\mid x\in \R^n\},\qquad\forall z\in\mathbb{R}^n,
		\end{align}
		and we have
		\begin{equation}
			f_\tau(z)\leq g_\tau(z),\qquad \forall z\in\mathbb{R}^n,
		\end{equation}
		with equality at least for $z=T_\tau(x)$, for some $x\in E$.
		
		Moreover, observing that the intermediate costs $c^{t,s}$ are quadratic polynomials and arguing similarly to the proof of Proposition \ref{Kant_potentials_are_semi-convex}, we see that $f_\tau$ is $c^{\tau,T}$-convex, semi-convex  and lower semi-continuous on $\mathbb{R}^n$, whilst $g_\tau$ is $c^{0,\tau}$-concave, semi-concave and upper semi-continuous on $\mathbb{R}^n$.
        
		By Theorem \ref{contact_semicon_pair}, if we define the contact set $\mathcal{E}:=\{z\;|\;f_\tau(z)=g_\tau(z)\}$, then both functions are differentiable everywhere in $\mathcal{E}$ and the function $z\mapsto \nabla g_\tau(z)=\nabla f_\tau(z)$ is locally Lipschitz on $\mathcal{E}$.
		By construction, we have 
		\begin{equation}
			f_\tau(T_t(x))=g_\tau(T_\tau(x)),\qquad\forall x\in E,
		\end{equation}
		that is 
		\begin{equation}
			f_\tau(z)=g_\tau(z),\qquad\mu_\tau\text{-almost every }z\in\mathbb{R}^n.
		\end{equation} 
		Hence, the closed set $\mathcal{E}$ is of full $\mu_\tau$-measure, with $E_\tau\subset\mathcal{E}$.
		If we consider $z\in E_\tau$, and we take any $x\in E$ such that $T_\tau(x)=z$, then it follows that
		\begin{equation}
			g_\tau(z)-c^{0,\tau}(x,z)=\psi(x),
		\end{equation}
		and, by definition of $g_\tau$, it also holds that 
		\begin{equation}
			g_\tau(w)-c^{0,\tau}(x,w)\leq\psi(x).
		\end{equation}
		Therefore, the function $w\mapsto g_\tau(w)-c^{0,\tau}(x,w)$ is differentiable at $w=z$ and attains a maximum there. Differentiating, we get
		\begin{equation}\label{fourth_contact_estimate}
			\nabla g_\tau(z)-\nabla_zc^{0,\tau}(x,z)=0.
		\end{equation}
         In other words, $\nabla g_\tau(z) = \nabla_z\tilde{c}^{0,\tau}(z,x)$, where $\tilde{c}^{0,\tau}$ is the cost of the backwards LQ problem. By Lemma \ref{lem2}, and denoting with a tilde the exponential map of the backwards problem, we have
		\begin{equation}
			x=\widetilde{\exp}_{z,\tau}(-\nabla g_\tau(z)),
		\end{equation}
		that is $x$ is uniquely determined by $z=T_\tau(x)$. Hence, $T_\tau$ is injective on $E$ and the function $S_\tau:E_\tau\to E$, defined by 
		\begin{equation}
			S_\tau(z):=\widetilde{\exp}_{z,\tau}(-\nabla g_\tau(z)),
		\end{equation}
		is a locally Lipschitz inverse for $T_\tau$.
		In particular, $S_\tau$ is measurable and 
		\begin{align}\label{inverse_of_tr_map}
			\mu|_E=S_{\tau\sharp}\circ T_{\tau\sharp}(\mu|_{E})=S_{\tau\sharp}(\mu_\tau|_{E_\tau}).
		\end{align}
		Finally, consider $F_\tau=\supp(\mu_\tau)\cap E_\tau\subset\mathcal{E}$, which is $\mu_\tau$-full measure. If, by contradiction, $\mu_\tau$ is not absolutely continuous with respect to Lebesgue, there exists a Borel set $A\subset F_\tau$, such that
		\begin{equation}
			\mu_\tau(A)>0,\qquad\mathscr{L}^n(A)=0.
		\end{equation}
            By interior regularity of $\mu_\tau$, we may assume that $A$ is compact, so that $S_\tau|_A$ is Lipschitz by construction.
		Therefore, by \eqref{inverse_of_tr_map}, we have that
		\begin{equation}
			\mu(S_\tau(A))=\mu_\tau(A)>0,\qquad\mathscr{L}^n(S_\tau(A))=0,
		\end{equation}
	    contradicting $\mu\ll\mathscr{L}^n$. Hence, $\mu_\tau \ll \mathscr{L}^n$ for all $\tau\in (0,T)$.
        
        We now prove that $T_\tau$ is the optimal transport map between $\mu$ and $\mu_\tau$ with respect to the cost $c^{0,\tau}$ or, equivalently, that $\alpha_{0,\tau}:=(id,T_\tau)_{\sharp}\mu$ is an optimal transport plan. Observe that $(g_\tau^{c-},g_\tau)$ is a pair of Kantorovich potentials. 
        Let $x\in E$, $y=T(x)$, and $z=T_\tau(x)$. It holds
        \begin{align}
            g_\tau(z) - c^{0,\tau}(x,z) & = f_\tau(z)-c^{0,\tau}(x,z) & \quad \text{since $f_\tau(z)=g_\tau(z)$} \\
            & \geq \psi^c(y)-c^{\tau,T}(z,y)-c^{0,\tau}(x,z) & \quad \text{by definition of $f_\tau$} \\
            & = \psi(x) & \quad \text{by equality in \eqref{first_support_estimate} and \eqref{second_support_estimate}} \\
            & \geq g_\tau(z)-c^{0,\tau}(x,z). & \quad \text{by definition of $g_\tau$}
        \end{align}
        By the characterization of $c$-subdifferentials (see Remark \ref{charach_contact_set}) and since $\mu(E)=1$, it follows that $\alpha_{0,\tau}$ is concentrated on the set $\partial_c g_\tau^{c-}$. We conclude by Theorem \ref{supp_opt_plans}.
        \end{proof}

    \section{Interpolation inequalities}
        \subsection{Main Jacobian estimate}
        Let $(A,B,Q,T)$ be a LQ optimal control problem satisfying the Kalman condition, with $T<t^*$. Let $\lambda:[0,T]\to T^*\mathbb{R}^n$ be an integral curve for the LQ Hamiltonian vector field $\overrightarrow{\mathcal{H}}$.
        
		A Jacobi vector field $X\in\mathcal{X}(T^*\mathbb{R}^n)$ along $\lambda$ is a smooth section $X:[0,T]\to T(T^*\mathbb{R}^n)$, such that
			\begin{itemize}
				\item $\pi(X(\tau))=\lambda(\tau),\quad \forall \tau\in[0,T]$, where $\pi:T(T^*\mathbb{R}^n)\to T^*\mathbb{R}^n$ is the canonical projection,
				\item $\dot{X}(\tau)=0,\quad \forall \tau\in(0,T)$, where the dot denotes the Lie derivative in the direction of $\overrightarrow{\mathcal{H}}$, namely
				\begin{equation}
					\dot{X}(\tau):=\frac{d}{d\xi}\bigg|_{\xi=0}e_{*}^{-\xi\overrightarrow{\mathcal{H}}}X(\tau+\xi).
				\end{equation}
			\end{itemize}
		The vector field $X$ can be written as 
		\begin{equation}
            X(\tau)=e^{\tau\overrightarrow{\mathcal{H}}}_{*}X(0),\qquad\forall \tau\in [0,T].
		\end{equation}
		Therefore, the space of Jacobi vector fields along a curve $\lambda$ is $2n$-dimensional.
        
        Similarly, given a family of $n$ independent Jacobi vector fields $\{X_1,\dots, X_n\}$, we can consider the family of $n$-dimensional subspaces
		\begin{equation}
			\mathcal{L}_\tau:=\mathrm{span}\{X_1(\tau),\dots,X_n(\tau)\}\subset T_{\lambda(\tau)}(T^*\mathbb{R}^n),\qquad \forall \tau\in[0,T], 
		\end{equation}
		and we have that 
		\begin{equation}
			\mathcal{L}_\tau=e^{\tau\overrightarrow{\mathcal{H}}}_{*}\mathcal{L}_0,\qquad\forall \tau\in[0,T].
		\end{equation}
		The family $\mathcal{L}_\tau$ can be thought as a curve in a suitable Grassmannian bundle over $T^*\mathbb{R}^n$.
        The canonical identification $T_{\lambda}(T^*\mathbb{R}^n)\cong T_{\lambda}(\mathbb{R}^{2n})\cong\mathbb{R}^{2n}$ allows us to perform very explicit computations. Indeed, under such identification, we have that $\overrightarrow{\mathcal{H}}=-\Omega H$ (recall Theorem \ref{ham_LQ}) and a Jacobi vector field $X\in\mathcal{C}^{\infty}([0,T],\mathbb{R}^{2n})$ takes the form
		\begin{equation}	
			X(\tau)=e^{-\tau\Omega H}X(0),\qquad X(0)\in\mathbb{R}^{2n},
		\end{equation}
		and a curve of subspaces $\mathcal{L}_\tau$ can be identified with a curve of smooth matrices
		\begin{equation}
			\mathbf{J}(\tau)=\begin{pmatrix}
				M(\tau)\\
				N(\tau)
			\end{pmatrix}=e^{-\tau\Omega H}\mathbf{J}(0),\quad \forall \tau\in[0,T],\qquad\mathbf{J}(0)\in\mathbb{R}^{2n\times n},\quad \mathrm{rank}(\mathbf{J}(0))=n.
		\end{equation}
		We call $\mathbf{J}$ a \emph{Jacobi matrix} and $M$, $N$ are its \emph{vertical} and \emph{horizontal} components, respectively. Any Jacobi matrix is determined by its value $\mathbf{J}(s)$ at any intermediate time $s\in[0,T]$. In particular, for every $s\in[0,T]$, we can define the following class of special Jacobi matrices as 
		\begin{align}
			\mathbf{J}^V_s(\tau)&:=\begin{pmatrix}
				R_1(\tau-s)\\
				R_3(\tau-s)
			\end{pmatrix},\quad \text{ such that }\mathbf{J}^V_s(s)=\begin{pmatrix}
			\mathbb{1}_n\\
			\mathbb{0}_n
			\end{pmatrix},\quad \text{(vertical at time }s),\\
			\mathbf{J}^H_s(\tau)&:=\begin{pmatrix}
				R_2(\tau-s)\\
				R_4(\tau-s)
			\end{pmatrix},\quad \text{ such that }\mathbf{J}^H_s(s)=\begin{pmatrix}
				\mathbb{0}_n\\
				\mathbb{1}_n
			\end{pmatrix},\quad \text{(horizontal at time }s),
		\end{align}
		where $R_i,$ with $i=1,\dots,4$ are the $n\times n$ blocks of the matrix representing the Hamiltonian flow, as in Definition \ref{blocks_R_i}.
        
        Since $T<t^*$, the matrix $R_3(\tau)$ is non-singular for every $\tau\in(0,T]$. As a consequence, any Jacobi vector field is determined by its horizontal components at two different times. 
		\begin{lem}\label{uniq_det}
			For any $0\leq s_1 < s_2 \leq T$ and any $x_1, x_2\in\mathbb{R}^n$, there exists a unique Jacobi field $X:[0,T]\to\mathbb{R}^{2n}$, such that
			\begin{equation}
				X(s_1)=\begin{pmatrix}
					*\\
					x_1
				\end{pmatrix},\qquad X(s_2)=\begin{pmatrix}
					*\\
					x_2
				\end{pmatrix},
			\end{equation}
			that is $X(s_i)$ has horizontal component equal to $x_i$, for $i=1,2$.
		\end{lem}
		\begin{proof}
			Any Jacobi field of the form 
			\begin{equation}
				X(\tau)=e^{-(\tau-s_1)\Omega H}\begin{pmatrix}
					p_1\\
					x_1
				\end{pmatrix},\qquad \tau\in[0,T],\qquad p_1\in\mathbb{R}^n,
			\end{equation}
			satisfies the first condition. Then, we have that 
			\begin{align}
				X(s_2)&=e^{-(s_2-s_1)\Omega H}\begin{pmatrix}
					p_1\\
					x_1
				\end{pmatrix}\\
				&=\begin{pmatrix}
				R_1(s_2-s_1) & R_2(s_2-s_1)\\
				R_3(s_2-s_1) & R_4(s_2-s_1)
				\end{pmatrix}\begin{pmatrix}
				p_1\\
				x_1
				\end{pmatrix}\\
				&=\begin{pmatrix}
				R_1(s_2-s_1)p_1+R_2(s_2-s_1)x_1\\
				R_3(s_2-s_1)p_1+R_4(s_2-s_1)x_1
				\end{pmatrix},
			\end{align}
			so that $p_1$ is uniquely determined by the equation
			\begin{equation}
				x_2=R_3(s_2-s_1)p_1+R_4(s_2-s_1)x_1,
			\end{equation}
			where we used the fact that $\det(R_3(s_2-s_1))\neq 0$, for $0\leq s_1< s_2\leq T$.
		\end{proof}
		Now, we use these tools to estimate the Jacobian of the transport maps $\nabla T_\tau$, for $\tau\in(0,T]$. 
		\begin{thm}[Main Jacobian estimate]\label{main_jacobian_estimate} Let $(A,B,Q,T)$ be a LQ optimal control problem satisfying the Kalman condition, with $T<t^*$.
			Consider $\mu,\nu\in P_2(\mathbb{R}^n)$, with $\mu\ll\mathscr{L}^n$, and consider a Kantorovich potential $\psi:\mathbb{R}^n\to\mathbb{R}^n$ for the LQ transport problem from $\mu$ to $\nu$. If $x\in\mathbb{R}^n$ is a point where $\psi$ is twice differentiable, then the transport maps $T_\tau$ are differentiable at $x$ for every $\tau\in[0,T]$, and, for every $s\in(0,T]$, it holds
			\begin{equation}\label{m_j_estimate}
				\det(\nabla T_\tau(x))^{1/n}\geq\frac{\det(R_3(\tau-s))^{1/n}}{\det(R_3(-s))^{1/n}}+\frac{\det(R_3(\tau))^{1/n}}{\det(R_3(s))^{1/n}}\det(\nabla T_s(x))^{1/n},\quad\forall \tau\in[0,s].
			\end{equation}
			Moreover, both terms in the right hand side of \eqref{m_j_estimate} are non-negative for $\tau\in[0,s]$ and, for $\tau\in[0,s)$, the first one is positive. In particular, $\det(\nabla T_\tau(x))>0$ for all $\tau\in[0,T)$.
		\end{thm}
		\begin{proof}

		The proof of Theorem \ref{main_jacobian_estimate} is involved, therefore we split it into several steps. First, the following two lemmas establish some useful properties of the matrices $R_i$, for $i=1,\dots,4$.
		\begin{lem}[W matrix]\label{W_matrix}
			Consider the smooth family of $n\times n$ matrices $R_1(\tau)$, with $\tau\in[0,T]$. Then, there exists $\varepsilon>0$ such that $\det(R_1(\tau))\neq 0$ for $\tau\in[0,\varepsilon]$. Moreover, if we define
			\begin{equation}\label{W_formula}
				W(\tau):=R_3(\tau)R_1^{-1}(\tau),\qquad\forall \tau\in[0,\varepsilon],
			\end{equation}
			then $W$ is symmetric and $W>0$ for $\tau\in(0,\varepsilon]$.
		\end{lem}
		\begin{proof}
			Since $R_1(0)=\mathbb{1}_n$ and by continuity, there exists $\varepsilon>0$ such that $\det(R_1(\tau))\neq 0$, for $\tau\in[0,\varepsilon]$. Thus, on $[0,\varepsilon]$, we define  $W$ according to \eqref{W_formula} and the map $[0,\varepsilon]\ni t\to W(t)$ is smooth. Note that 
			\begin{equation}
				\Omega e^{-\tau\Omega H}=e^{-\tau H\Omega}\Omega,
			\end{equation}	
			and recall Definition \ref{blocks_R_i}. We have
			\begin{equation}
				\begin{pmatrix}
					R_1(\tau)\\ 
					R_3(\tau)
				\end{pmatrix}=e^{-\tau\Omega H}\begin{pmatrix}
					\mathbb{1}_n\\ 
					\mathbb{0}_n
				\end{pmatrix}.
			\end{equation}
			
			Therefore, for every $\tau\in\mathbb{R}^n$, it holds
			\begin{align}
				R_1^*(\tau)R_3(\tau)-R_3^*(\tau)R_1(\tau)&=\begin{pmatrix}
					R_1^*(\tau) & R_3^*(\tau)
				\end{pmatrix}\Omega\begin{pmatrix}
				R_1(\tau)\\ 
				R_3(\tau)
				\end{pmatrix}\\
				&=\begin{pmatrix}
					\mathbb{1}_n & \mathbb{0}_n
				\end{pmatrix}e^{\tau H\Omega}\Omega e^{-\tau\Omega H}\begin{pmatrix}
					\mathbb{1}_n\\ 
					\mathbb{0}_n
				\end{pmatrix}\\
				&=\begin{pmatrix}
					\mathbb{1}_n & \mathbb{0}_n
				\end{pmatrix}\Omega\begin{pmatrix}
					\mathbb{1}_n\\ 
					\mathbb{0}_n
				\end{pmatrix}=0,
			\end{align}
			which is equivalent to the symmetry of $W$, as long as $R_1$ is invertible.
			In addition, differentiating the matrix $W$ and recalling Definition \ref{blocks_R_i}, we get
			\begin{align}
				\dot{W}(\tau)&=\dot{R_3}(\tau)R_1^{-1}(\tau)-R_3(\tau)R_1^{-1}(\tau)\dot{R_1}(\tau)R_1^{-1}(\tau)\\
				&=\left(BB^*R_1(\tau)+AR_3(\tau)\right)R_1^{-1}(\tau)\\
				&\qquad-R_3(\tau)R_1^{-1}(\tau)\left(-A^*R_1(\tau)-QR_3(\tau)\right)R_1^{-1}(\tau)\\
				&=BB^*+AW(\tau)+W(\tau)A^*+W(\tau)QW(\tau),\qquad\forall \tau\in(0,\varepsilon).
			\end{align}
			so that $W$ satisfies the matrix Riccati equation
			\begin{equation}
				\dot{W}=BB^*+AW+WA^*+WQW
			\end{equation}
			with $BB^*\geq 0$.
			Hence, applying the matrix Riccati comparison Theorem (see \cite[Appendix A]{Barilari_2016}) it follows that $W(\tau)\geq 0$, for $\tau\in[0,\varepsilon]$. Since $R_3(\tau)$ is non-singular for $\tau\in(0,T]$ and $R_1(\tau)$ is non-singular for $\tau\in[0,\varepsilon]$, we have that $W(\tau)>0$, for $\tau\in(0,\varepsilon]$.
		\end{proof}
		\begin{lem}[S matrix]\label{S_matrix}
			Consider the smooth family of $n\times n$ matrices 
			\begin{equation}
				S(\tau):=R_3^{-1}(\tau)R_4(\tau),\qquad \tau\in(0,T].
			\end{equation}
			Then, $S$ is symmetric and $\dot{S}\leq 0$.
		\end{lem}
		\begin{proof}
		We already know, from the computation of the LQ cost function (see Theorem \ref{comp_cost} and Lemma \ref{lem1}), that $S(\tau)$ is symmetric.

		In addition, we define
		\begin{equation}
			Z(\tau):=e^{-\tau\Omega H}\begin{pmatrix}
			S(\tau)\\
			-\mathbb{1}_n
			\end{pmatrix},\qquad\forall \tau\in(0,T],
		\end{equation}
        and note that 
        \begin{equation}
            Z(\tau)=\begin{pmatrix}
			Z_V(\tau)\\ 
			\mathbb{0}_n
			\end{pmatrix},\qquad \text{with} \qquad Z_V(\tau):=R_1(\tau)S(\tau)-R_2(\tau).
        \end{equation}
		By definition of matrix exponential, for all $\tau\in \R$ we have
		\begin{equation}
			H e^{-\tau\Omega H}=e^{-\tau H\Omega}H,\qquad \left(e^{-\tau\Omega H}\right)^*=e^{\tau H\Omega},\qquad \Omega e^{\tau\Omega H}=e^{\tau H\Omega}\Omega.
		\end{equation}
		As a consequence, for every $\tau \in(0,T)$, it follows that
		\begin{align}
			\dot{S}(\tau)&=\begin{pmatrix}
				S(\tau) & -\mathbb{1}_n
			\end{pmatrix}\Omega\begin{pmatrix}
			\dot{S}(\tau)\\
			\mathbb{0}_n
			\end{pmatrix}\\
			&=\begin{pmatrix}
				S(\tau) & -\mathbb{1}_n
			\end{pmatrix}\Omega\frac{d}{d\tau}\bigg(e^{\tau\Omega H}Z(\tau)\bigg)\\
			&=\begin{pmatrix}
				S(\tau) & -\mathbb{1}_n
			\end{pmatrix}\Omega\bigg(\Omega He^{\tau\Omega H}Z(\tau)+e^{\tau\Omega H}\dot{Z}(\tau)\bigg)\\
			&=-\begin{pmatrix}
				S(\tau) & -\mathbb{1}_n
			\end{pmatrix}e^{\tau H\Omega}HZ(\tau)+\begin{pmatrix}
				S(\tau) & -\mathbb{1}_n
			\end{pmatrix}e^{\tau H\Omega}\Omega\dot{Z}(\tau)\\
			&=-Z^*(\tau)HZ(\tau)+Z^*(\tau)\Omega\dot{Z}(\tau)\\
			&=-Z_V^*(\tau)BB^*Z_V(\tau)+\begin{pmatrix}
				Z_V(\tau) & \mathbb{0}_n
			\end{pmatrix}\Omega\begin{pmatrix}
				\dot{Z}_V(\tau)\\
				\mathbb{0}_n
			\end{pmatrix}\\
			&=-Z_V^*(\tau)BB^*Z_V(\tau)\leq 0,
		\end{align}
		proving that $\dot{S}\leq 0$, for $\tau\in(0,T)$.
		\end{proof}
		Consider $x\in\mathbb{R}^n$, where $\psi$ is twice differentiable and, by Theorem \ref{absolute_continuity-dyn_opt_coupl}, we have that 
		\begin{equation}
			T_\tau(x)=\exp_{x,\tau}(\nabla\psi(x))=R_3(\tau)\nabla\psi(x)+R_4(\tau)x,\qquad\forall \tau\in(0,T].
		\end{equation}
		Differentiating the previous equation with respect to $x$, we get that
		\begin{equation}
			\nabla T_\tau(x)=R_3(\tau)\Hess(\psi)(x)+R_4(\tau),\qquad\forall \tau\in(0,T],
		\end{equation}
		so that, by construction, $\nabla T_\tau(x)$ is the horizontal component of the Jacobi matrix
		\begin{equation}\label{def_Jacobi_Hessian}
			\mathbf{J}(\tau)=\begin{pmatrix}
				M(\tau)\\
				N(\tau)
			\end{pmatrix}:=e^{-\tau\Omega H}\begin{pmatrix}
				\Hess(\psi)(x)\\
				\mathbb{1}_n
			\end{pmatrix},\qquad\forall \tau\in(0,T].
		\end{equation}
		In order to compare the Jacobian of $T_\tau$ at $x$, for different times $\tau\in(0,T]$, we exploit Lemma \ref{uniq_det} to rewrite $\mathbf{J}(\tau)$ and the special Jacobi matrices in a different basis. 
		\begin{lem}[Change of basis]\label{change_of_basis}
			For every $s\in(0,T]$ and $\tau\in[0,T]$, we have
			\begin{equation}\label{special_change_of_basis}
				\mathbf{J}^V_s(\tau)=-\mathbf{J}^V_0(\tau)R_3^{-1}(s)R_4(s)R_3(-s)+\mathbf{J}^H_0(\tau)R_3(-s).
			\end{equation}
			In addition, the Jacobi matrix $\mathbf{J}(\tau)$ can be written as
			\begin{equation}\label{change_of_basis_Hessian}
				\mathbf{J}(\tau)=\mathbf{J}^V_0(\tau)\Hess(\psi)(x)+\mathbf{J}^H_0(\tau),
			\end{equation}
			or, equivalently, for every $s\in(0,T]$ and every $\tau\in[0,s]$, 
			\begin{equation}\label{change_of_basis_Hessian_2}
				\mathbf{J}(\tau)=\mathbf{J}^V_s(\tau)R_3^{-1}(-s)+\mathbf{J}^V_0(\tau)R_3^{-1}(s)N(s).
			\end{equation}
		\end{lem} 
		\begin{proof}
			To prove \eqref{special_change_of_basis}, we notice that, since a Jacobi matrix is determined by its value at any intermediate time, there exists a unique pair of $n\times n$ matrices $(C^V,C^H)$, such that
			\begin{equation}
				\mathbf{J}^V_s(\tau)=\mathbf{J}^V_0(\tau)C^V+\mathbf{J}^H_0(\tau)C^H.
			\end{equation}
			Evaluating at $\tau=0$ and at $\tau=s$ the previous equation, we get
            \begin{equation}
              \begin{cases}
					R_3(-s)=C^H,\\
					\mathbb{0}_n=R_3(s)C^V+R_4(s)C^H,
				\end{cases} \qquad \Rightarrow \qquad \begin{cases}
					C^H=R_3(-s),\\
					C^V=-R_3^{-1}(s)R_4(s)R_3(-s).
				\end{cases}
            \end{equation}
			Expanding equation \eqref{def_Jacobi_Hessian}, for every $\tau\in(0,T]$, we have 
			\begin{align}
				\mathbf{J}(\tau)&=e^{-\tau\Omega H}\begin{pmatrix}
					\Hess(\psi)(x)\\
					\mathbb{1}_n
				\end{pmatrix}\\
				&=\begin{pmatrix}
				R_1(\tau) & R_2(\tau)\\
				R_3(\tau) & R_4(\tau)
				\end{pmatrix}\begin{pmatrix}
				\Hess(\psi)(x)\\
				\mathbb{1}_n
				\end{pmatrix}\\
				&=\begin{pmatrix}
					R_1(\tau)\\
					R_3(\tau)
				\end{pmatrix}\Hess(\psi)(x)+\begin{pmatrix}
				R_2(\tau)\\
				R_4(\tau)
				\end{pmatrix},
			\end{align}
			proving equation \eqref{change_of_basis_Hessian}.
			Moreover, by Lemma \ref{uniq_det}, a Jacobi matrix is uniquely determined by its horizontal components at two different times $s_1\neq s_2$. Therefore, taking $s_1=0$ and $s_2=s$, there exists a pair of $n\times n$ matrices $(D_0,D_s)$ such that
			\begin{equation}
				\mathbf{J}(\tau)=\mathbf{J}^V_s(\tau)D_s+\mathbf{J}^V_0(\tau)D_0
			\end{equation}
			Evaluating the previous equation at time $\tau=0$ and $\tau=s$, we get
			\begin{equation}
				\begin{cases}
					\mathbb{1}_n=N(0)=R_3(-s)D_s,\\
					N(s)=R_3(s)D_0,
				\end{cases}\qquad \Rightarrow\qquad
                				\begin{cases}
					D_s=R_3^{-1}(-s),\\
					D_0=R_3^{-1}(s)N(s),
				\end{cases}
			\end{equation}
			proving identity \eqref{change_of_basis_Hessian_2}.
		\end{proof}
		Implementing all the information from previous lemmas and the properties of the Kantorovich potential, we get the following key result. The proof follows the blueprint of \cite[Lemma 29]{SRintineq}, with more explicit computation due to the explicitness of the Hamiltonian flow, and with the difference that the cost is not (the square of) a distance. 
		\begin{lem}[Positivity]\label{positivity_Jacobi}
			Under the hypotheses of Theorem \ref{main_jacobian_estimate}, we have
			\begin{equation}
				\det(R_3(\tau))>0,\qquad\forall \tau\in(0,T].
			\end{equation}
			Moreover, for $0<\tau\leq s\leq T$, it also holds
			\begin{enumerate}
				\item\label{i} $R_3^{-1}(\tau)R_3(\tau-s)R_3^{-1}(-s)\geq 0$,
				\item\label{ii} $R_3^{-1}(s)N(s)\geq 0$,
			\end{enumerate}	
			where $N(s)$ is the horizontal component of the Jacobi matrix defined in \eqref{def_Jacobi_Hessian}.	
		\end{lem}
		\begin{proof}
			Since $T<t^*$ and by Lemma \ref{W_matrix}, we have that
			\begin{itemize}
				\item $\det(R_3(\tau))\neq 0$, for every $\tau\in(0,T]$,
				\item the matrix $W(\tau)=R_3(\tau)R_1^{-1}(\tau)$ is well-defined and positive-definite for small positive times.
			\end{itemize}
			Therefore, since $R_1(0)=\mathbb{1}_n$, by continuity, $\det(R_1(t))>0$ for small times. As a consequence, we have that $R_3(\tau)$ has positive determinant for small, and thus all, times $\tau\in(0,T]$.
			For item \eqref{i}, we use equation \eqref{special_change_of_basis} from Lemma \ref{change_of_basis}, from which the horizontal component reads as:
			\begin{equation}
				R_3(\tau-s)=-R_3(\tau)R_3^{-1}(s)R_4(s)R_3(-s)+R_4(\tau)R_3(-s).
			\end{equation}
			Multiplying from the left by $R_3^{-1}(\tau)$ and from the right by $R_3^{-1}(-s)$, we get
			\begin{equation}
				R_3^{-1}(t)R_3(\tau-s)R_3^{-1}(-s)=-R_3^{-1}(s)R_4(s)+R_3^{-1}(\tau)R_4(\tau)=-S(s)+S(\tau).
			\end{equation} 
			By Lemma \ref{S_matrix}, the matrix $S$ is monotone non-increasing, from which item \eqref{i} follows.
			
			For item \eqref{ii}, again by Lemma \ref{change_of_basis} and reading the horizontal component of $\mathbf{J}(s)$, we get 
			\begin{equation}
				N(s)=R_3(s)\Hess(\psi)(x)+R_4(s),
			\end{equation}
			that is 
            \begin{equation}                
				R_3^{-1}(s)N(s) =\Hess(\psi)(x)+S(s),
            \end{equation}
			from which $R_3^{-1}(s)N(s)$ is symmetric, for every $s\in(0,T]$.
			
			If we consider $\gamma:[0,T]\to\mathbb{R}^n$ defined as 
			\begin{equation}
				\gamma(s)=\exp_{x,s}(\nabla\psi(x)),\qquad\forall s\in[0,T],
			\end{equation}
			and set $y:=\gamma(T)$, then we have
			\begin{itemize}
				\item $(x,y)\in\partial_c\psi$, that is 
				\begin{equation}
					\psi(x)+c^{0,T}(x,y)\leq\psi(z)+c^{0,T}(z,y),\qquad\forall z\in\mathbb{R}^n,
				\end{equation}
				\item recalling Remark \ref{rem_subadditivity_costs}, we have
				\begin{align}
					c^{0,T}(x,y)&=c^{0,s}(x,\gamma(s))+c^{s,T}(\gamma(s),y),\quad\forall s\in(0,T),\\
					c^{0,T}(z,y)&\leq c^{0,s}(z,\gamma(s))+c^{s,T}(\gamma(s),y),\quad\forall z\in\mathbb{R}^n,\quad\forall s\in(0,T).
				\end{align}
			\end{itemize}
			Therefore, for every $s\in(0,T]$, the function $z\mapsto\psi(z)+c^{0,s}(z,\gamma(s))$ is twice differentiable at $z=x$ and has a minimum there (note that for $s=T$, this follows from definition of Kantorovich potentials and not from previous argument). 
			This implies that its Hessian is non-negative definite, which is precisely 
			\begin{equation}
				\Hess(\psi)(x)+R_3^{-1}(s)R_4(s)\geq 0,\qquad\forall s\in(0,T],
			\end{equation}
			where we used that $\nabla^2_zc^{0,s}(z,\gamma(s))=R_3^{-1}(s)R_4(s)$, coming from the expression of the cost given in Theorem \ref{comp_cost}.
		\end{proof}
		
        We are ready to complete the proof of Theorem \ref{main_jacobian_estimate}. 
        For time $\tau=0$, there is nothing to prove, since $T_0(x)=x$ and the right hand side of equation \eqref{m_j_estimate} is equal to $1$. 
        Recalling that $N(\tau)=\nabla T_\tau(x)$ by \eqref{def_Jacobi_Hessian}, reading the horizontal components of \eqref{change_of_basis_Hessian} and \eqref{change_of_basis_Hessian_2}, we have that
		\begin{equation}\label{equation_without_number}
			\nabla T_\tau(x)=R_3(\tau-s)R_3^{-1}(-s)+R_3(\tau)R_3^{-1}(s)\nabla T_s(x),\qquad\forall 0<\tau\leq s\leq T.
		\end{equation}
		Each of the matrices of equation \eqref{equation_without_number}, upon left multiplication by $R_3^{-1}(T)$, are symmetric and non-negative. By Minkowski's determinant Theorem \cite[Theorem 4.1.8]{marcus1992survey}, the function $A\mapsto\det(A)^{1/n}$ is concave on the set of symmetric non-negative matrices. Therefore, coupling it with Lemma \ref{positivity_Jacobi} from which $\det(R_3(\tau))>0,$ for every $\tau\in(0,T]$, it follows that 
		\begin{equation}
			\det(\nabla T_\tau(x))^{1/n}\geq\frac{\det(R_3(\tau-s))^{1/n}}{\det(R_3(-s))^{1/n}}+\frac{\det(R_3(\tau))^{1/n}}{\det(R_3(s))^{1/n}}\det(\nabla T_s(x))^{1/n},
		\end{equation}
		for $0<\tau\leq s\leq T$. By construction, both terms in the right hand side are non-negative. Moreover, if $\tau<s$, then the first term is non-zero and thus positive (recall that $-t^*<\tau<t^*$, with $\tau\neq 0$, implies $\det(R_3(\tau))\neq 0$, by Remark \ref{rmk_conj_times}).
		\end{proof}
        
        \subsection{Interpolation of densities}
        Under the hypotheses of Corollary \ref{cor_dyn_opt_coup}, we consider the unique displacement interpolation $(\mu_\tau)\in\mathcal{C}([0,T],P_2(\mathbb{R}^n))$ between $\mu$ and $\nu$. By Theorem \ref{absolute_continuity-dyn_opt_coupl}, we have $\mu_\tau\ll\mathscr{L}^n$ and we define
		\begin{equation}
			\rho_\tau:=\frac{d\mu_\tau}{d\mathscr{L}^n}\in L^1(\R^n),\qquad\forall \tau\in[0,T).
		\end{equation}
		If $\nu\ll\mathscr{L}^n$, we set $\rho_T:=d\nu/d\mathscr{L}^n$ as well.
        
		\begin{Def}[Distortion coefficients]\label{def:distcoeff}
			Let $(A,B,Q,T)$ be a LQ optimal control problem satisfying the Kalman condition, with $T<t^*$. Given two non-empty subsets $E,F\subset \mathbb{R}^n$, define the set of $t$-intermediate points as
			\begin{equation}
				Z_\tau(E,F):=\{\gamma(\tau)\mid \gamma\in\mathcal{C}([0,T],\R^n), \text{ $\gamma$ optimal,}\; \gamma(0)\in E,\;\gamma(T)\in F\},
			\end{equation}
			 Then, for every $\tau\in[0,T]$ and $x,y\in\mathbb{R}^{n}$, we define the $t$-intermediate distortion coefficient as
			\begin{equation}
				\beta_\tau(x,y):=\limsup\limits_{r \to 0^{+}}\frac{\mathscr{L}^n(Z_\tau(x,B_r(y)))}{\mathscr{L}^n(B_r(y))},
			\end{equation}
                     where $B_r(y)$ is the Euclidean ball with center $y$ and radius $r$.
		\end{Def}
        \begin{rem}
        We will prove in Lemma \ref{distortion_coefficients} that the distortion coefficients do not depend on the choice of $x,y$. We denote by $\beta^{(A,B,Q,T)}_{\tau}$ (or just $\beta_{\tau}$ if there is no risk of confusion) the distortion coefficients for any pair of points and $\tau\in[0,T]$. We observe also that $\beta_T=1$ and $\beta_0=0$.
        \end{rem}

		With all the tools developed in previous sections, we can compute the LQ distortion coefficients in terms of 
       the blocks $R_i:\mathbb{R}\to\mathbb{R}^{n\times n}, i=1,\dots,4$, of the Hamiltonian flow.
        
		\begin{lem}\label{distortion_coefficients}
			Given a LQ optimal control problem $(A,B,Q,T)$, with $T<t^*$, we have that
			\begin{equation}\label{dist_coeff_formula}
				\beta_\tau(x,y)=\frac{\det(R_3(\tau))}{\det(R_3(T))},\qquad\forall \tau\in[0,T],\quad\forall(x,y)\in\mathbb{R}^n\times\mathbb{R}^n.
			\end{equation}
			In particular, $\beta_\tau(x,y)$ does not depend on the points $(x,y)$ and it is positive for $\tau>0$. Moreover, the backwards problem $(-A,B,Q,T)$ has the same distortion coefficients.
		\end{lem}
		\begin{proof}
			By Theorem \ref{LQ_exponential} and Remark \ref{formula_opt_traj}, between any two points $(x,y)\in\mathbb{R}^n$ there exists a unique optimal trajectory, given by $\tau\mapsto\exp_{x,\tau}\circ\exp_{x,T}^{-1}(y)$. Thus, it holds 
			\begin{equation}
				Z_\tau(x,B_r(y))=\exp_{x,\tau}\circ\exp_{x,T}^{-1}(B_r(y)),\qquad\forall(x,y)\in\mathbb{R}^n\times\mathbb{R}^n.
			\end{equation}
			Furthermore, the map $z\mapsto\exp_{x,t}\circ\exp_{x,T}^{-1}(z)$ is a smooth, affine,  on $\mathbb{R}^n$, whose Jacobian at any point is given by $R_3(t)R_3^{-1}(T)$.
			Therefore, we  have
			\begin{align}
                \mathscr{L}^n(Z_\tau(x,B_r(y)))&=\int_{Z_\tau(x,B_r(y))}d\mathscr{L}^n\\
				&=\int_{B_r(y)}|\det(R_3(\tau)R_3^{-1}(T))|d\mathscr{L}^n\\
				&=|\det(R_3(\tau)R_3^{-1}(T))|\mathscr{L}^n(B_r(y)),
			\end{align}
			which is independent on the points $(x,y)\in\mathbb{R}^n\times\mathbb{R}^n$. Recalling that $\det(R_3(\tau))>0$ for every $\tau\in(0,T]$ by Lemma \ref{positivity_Jacobi}, we obtain formula \eqref{dist_coeff_formula} and that $\beta_\tau(x,y)>0$, for $t\in(0,T]$.

			In addition, by Lemma \ref{backwards_flow} and Lemma \ref{time_reflection_symmetry}, the flow of the backwards Hamiltonian $-\Omega \tilde{H}$ can be expressed in terms of the matrices $R_i$, $i=1,\ldots,4$, as 
			\begin{equation}
				e^{-\tau\Omega\tilde{H}}=\begin{pmatrix}
					R_1(-\tau) & -R_2(-\tau)\\
					-R_3(-\tau) & R_4(-\tau)
				\end{pmatrix}, \quad \tau\in\mathbb{R}^n,
			\end{equation}
			and $-R_3(-\tau)=R_3^*(\tau)$, for every $\tau\in\mathbb{R}^n$. Therefore,
			\begin{equation}
				\frac{\det(-R_3(-\tau))}{\det(-R_3(-T))}=\frac{\det(R_3^*(\tau))}{\det(R_3^*(T))}=\frac{\det(R_3(\tau))}{\det(R_3(T))},
			\end{equation}
			from which the forward and backwards distortion coefficients coincide.
		\end{proof}

        We show an example of LQ distortion coefficients that mimic the distortion coefficients of the model Riemannian manifolds of dimension $n$ and constant curvature $k$ (see \cite{SturmGeomMmsI,SturmGeomMmsII,lott2006riccicurvmmspaces} or \cite[Chapter 14]{Villanioldandnew}).
		\begin{exmp}\label{reproducing_CD_kn_Riem}
				We define the block-diagonal matrix
			\begin{equation}
				\mathcal{I}_n:=\begin{pmatrix}
					\mathbb{1}_{n-1} & \mathbb{0}_{n-1,1}\\
					\mathbb{0}_{1,n-1} & \mathbb{0}_1
				\end{pmatrix},
			\end{equation}
			where $\mathbb{1}_i$ is a square identity matrix of dimension $i$, $\mathbb{0}_j$ is a square null matrix of dimension $j$ and $\mathbb{0}_{i,j}$ is rectangular null matrix of dimension $i\times j$.
			
			For $K\in\mathbb{R}$, we consider an $n$-dimensional LQ problem where 
			\begin{equation}
				A:=\mathbb{0}_n,\qquad B:=\mathbb{1}_n,\qquad Q:=K\mathcal{I}_n,
			\end{equation}
			so that the corresponding Hamiltonian vector field on $T^*\mathbb{R}^n$ is 
			\begin{equation}
				\overrightarrow{\mathcal{H}}_{(p,x)}=-\Omega H\begin{pmatrix}
					p\\
					x
				\end{pmatrix}=\begin{pmatrix}
					0 & -K\mathcal{I}_n\\
					\mathbb{1}_n & 0
				\end{pmatrix}\begin{pmatrix}
					p\\
					x
				\end{pmatrix}.
			\end{equation}
			The first powers of $-\Omega H$ are 
			\begin{equation}
				(-\Omega H)^2=\begin{pmatrix}
					-K\mathcal{I}_n & 0\\
					0 & -K\mathcal{I}_n
				\end{pmatrix},
			\end{equation}
			\begin{equation}
				(-\Omega H)^3=\begin{pmatrix}
					0 & K^2\mathcal{I}_n\\
					-K\mathcal{I}_n& 0
				\end{pmatrix},
			\end{equation}
			\begin{equation}
				(-\Omega H)^4=\begin{pmatrix}
					K^2\mathcal{I}_n & 0\\
					0 & K^2\mathcal{I}_n
				\end{pmatrix}.
			\end{equation}
			Noticing that $\mathcal{I}_n$ is idempotent and $(-\Omega H)^4$ is block-diagonal, we can compute explicitly the matrix exponential representing the flow of the vector field, just relying on the first four powers, and we get
			\begin{equation}
				e^{-\tau\Omega H}=\begin{pmatrix}
					\cos(\tau\sqrt{K})\mathbb{1}_{n-1}& \mathbb{0}_{n-1,1} & -\sqrt{K}\sin(\tau\sqrt{K})\mathbb{1}_{n-1} &\mathbb{0}_{n-1,1}\\
					\mathbb{0}_{1,n-1} & \mathbb{1}_1 & \mathbb{0}_{1,n-1} & \mathbb{0}_1\\
					\frac{\sin(\tau\sqrt{K})}{\sqrt{K}}\mathbb{1}_{n-1} & \mathbb{0}_{n-1,1} & \cos(\tau\sqrt{K})\mathbb{1}_{n-1}&\mathbb{0}_{n-1,1}\\
					\mathbb{0}_{1,n-1} & \tau\mathbb{1}_1 & \mathbb{0}_{1,n-1} & \mathbb{1}_1
				\end{pmatrix},\qquad\forall \tau\in\mathbb{R},
			\end{equation}
			with the convention that $\sin(\tau\sqrt{K})/\sqrt{K}=\tau$, for $K=0$ and with obvious convention, for $K<0$.

			From previous formula, we have that 
			\begin{equation}
				R_3(\tau)=\begin{pmatrix}
					\frac{\sin(\tau\sqrt{K})}{\sqrt{K}}\mathbb{1}_{n-1}  & \mathbb{0}_{n-1,1}\\
					\mathbb{0}_{1,n-1} & t\mathbb{1}_1 
				\end{pmatrix},\qquad\forall t\in\mathbb{R},
			\end{equation}
			so that the first conjugate time is $t^*=+\infty$, for $K\leq 0$, and $t^*=\pi/\sqrt{K}$, for $K>0$. Assuming that $K<\pi^2$, we can choose $T=1$. Therefore, applying Lemma \ref{distortion_coefficients}, for $0\leq \tau\leq 1$, we get 
			\begin{equation}
				\beta_\tau=\frac{\det(R_3(\tau))}{\det(R_3(1))}=\begin{cases}
					\tau\left(\frac{\sin(\tau\sqrt{K})}{\sin(\sqrt{K})}\right)^{n-1}&\text{ for }K>0,\\
					\tau^n&\text{ for }K=0,\\
					\tau\left(\frac{\sinh(\tau\sqrt{-K})}{\sinh(\sqrt{-K})}\right)^{n-1}&\text{ for }K<0.
				\end{cases}
			\end{equation}
			If we set $\theta=d(x,y)\sqrt{|k|/(n-1)}$ and take $K=\mathrm{sgn}(k)\theta^2$, we get 
			\begin{equation}
				\beta_{\tau}=\begin{cases}
					\tau\left(\frac{\sin(\tau\theta)}{\sin(\theta)}\right)^{n-1}&\text{ for }k>0,\\
					\tau^n&\text{ for }k=0,\\
					\tau\left(\frac{\sinh(\tau\theta)}{\sinh(\theta)}\right)^{n-1}&\text{ for }k<0,
				\end{cases}
			\end{equation}
			where we shall have $\theta<\pi$ if $k>0$.
			
			In other words, with the above choices, we have that $\beta_\tau^{(A,B,Q,1)}$ are the distortion coefficients of a Riemannian space form with constant Ricci curvature equal to $k$ and dimension $n$.
		\end{exmp}

        We turn equation \eqref{m_j_estimate} into a relation between the distortion coefficients $\beta_\tau$ of the LQ problem and the displacement interpolation $(\mu_\tau)$. In particular, we rely upon the following Lemma (see \cite[Lemma 40]{SRintineq}, to which we refer for a detailed proof).
        
		\begin{lem}\label{Jacobian_and_densities}
			Let $\rho\in L^1(\mathbb{R}^n)$ be a non-negative function. Let $f:\mathbb{R}^n\to\mathbb{R}^n$ be a measurable function. Suppose that there exists a Borel set $\Sigma\subset\mathbb{R}^n$ such that
			\begin{itemize}
				\item $f$ is differentiable at any point $x\in\Sigma$,
				\item $f|_{\Sigma}$ is injective,
				\item the set $\{\rho>0\}\setminus\Sigma$ is $\mathscr{L}^n$-negligible.
			\end{itemize}
			Then, the measure $f_{\sharp}(\rho\mathscr{L}^n) \ll \mathscr{L}^n$ if and only if $\det(\nabla f(x))>0$ for $\rho\mathscr{L}^n$-almost every $x\in\mathbb{R}^n$. In such case, letting $f_{\sharp}(\rho\mathscr{L}^n)=\rho_f\mathscr{L}^n$, we have
			\begin{equation}
				\rho_f(f(x))=\frac{\rho(x)}{\det(\nabla f(x))},\qquad\forall x\in\Sigma.
			\end{equation}
		\end{lem}
		\begin{thm}[Interpolation inequalities]\label{density_interpolation_inequality}
Let $(A,B,Q,T)$ be a LQ problem satisfying the Kalman condition, with $T<t^*$. Consider $\mu,\nu\in P_2(\mathbb{R}^n)$, with $\mu, \nu \ll\mathscr{L}^n$, and let $(\mu_\tau)=(e_{\tau\sharp}\Pi)$ be the unique displacement interpolation between them. Let $T_\tau$ be the optimal transport map from $\mu$ to $\mu_\tau$, for the cost $c^{0,\tau}$. Then, for $\mu$-almost every $x\in\mathbb{R}^n$ and every $\tau\in[0,T]$, we have
			\begin{equation}\label{density_interpolation_inequality_formula}
				\frac{1}{\rho_\tau(T_\tau(x))^{1/n}}\geq\frac{\beta_{T-\tau}^{1/n}}{\rho_0(x)^{1/n}} +\frac{\beta_{\tau}^{1/n}}{\rho_T(T(x))^{1/n}},
			\end{equation}
			where $\beta_s=\beta_s^{(A,B,Q,T)}$ is the distortion coefficient of the LQ problem and $\mu_\tau=\rho_\tau\mathscr{L}^n$. 
            If $\nu$ is not absolutely continuous, \eqref{density_interpolation_inequality_formula_intro} holds for $\tau\in[0,T)$ and omitting the second term in the right hand side.
            \end{thm}
		\begin{proof}
			First, let $\mu,\nu\ll \mathscr{L}^n$. By Theorem \ref{absolute_continuity-dyn_opt_coupl}, letting $\psi:\mathbb{R}^n\to\mathbb{R}^n$ be a $c^{0,T}$-convex Kantorovich potential from $\mu$ to $\nu$, it holds $\mu_\tau = (T_\tau)_\sharp\mu$, where
\begin{equation}
			T_\tau(x)=\exp_{x,\tau}(\nabla\psi(x)),\qquad\mu\text{-almost every }x\in\mathbb{R}^n, \quad\forall \tau\in[0,T],
\end{equation}
and $\mu$-almost everywhere $\psi$ is twice differentiable. Pairing this with Theorem \ref{main_jacobian_estimate}, we find a borel set $E$, with $\mu(E)=1$, such that

			\begin{itemize}
				\item $\rho_0(x)>0$, for every $x\in E$,
				\item $T_\tau(x)=\exp_{x,\tau}(\nabla\psi(x))$, for every $x\in E$ and every $\tau\in[0,T]$,
				\item $T_\tau$ is differentiable at $x$, for every $x\in E$ and every $\tau\in[0,T]$,
				\item $\det(\nabla T_\tau(x))>0$, for every $x\in E$ and every $\tau\in[0,T)$,
			\end{itemize} 
			where we note that in last item $\tau=T$ is excluded apriori.
			Nonetheless, we can apply Lemma \ref{Jacobian_and_densities} to the map $T$, so that, from the absolute continuity of $\nu$, we also have that $\det(\nabla T(x))>0$, for $\mu$-almost every $x\in\mathbb{R}^n$.
            
			Therefore, passing to a smaller set $E$, of full $\mu$-measure, we apply again Lemma \ref{Jacobian_and_densities}, to get 
			\begin{equation}\label{Monge_Ampere}
				\det(\nabla T_\tau(x))=\frac{\rho_0(x)}{\rho_\tau(T_\tau(x))},\qquad\forall \tau\in[0,T],\quad\forall x\in E.
			\end{equation}
			As a consequence, taking $s=T$ in equation \eqref{m_j_estimate} and using equation \eqref{Monge_Ampere}, we get 
			\begin{equation}
				\frac{\rho_0(x)^{1/n}}{\rho_t(T_\tau(x))^{1/n}}\geq\frac{\det(R_3(\tau-T))^{1/n}}{\det(R_3(-T))^{1/n}}+\frac{\det(R_3(\tau))^{1/n}}{\det(R_3(T))^{1/n}}\frac{\rho_0(x)^{1/n}}{\rho_T(T(x))^{1/n}},\quad \forall x\in E.
			\end{equation}
			By Lemma \ref{dist_coeff_formula}, we can substitute $\beta_{T-\tau}$ and $\beta_{\tau}$ in previous equation to get \eqref{density_interpolation_inequality_formula}, since $\rho_0(x)>0$ for every $x\in E$. If $\nu\in P_2(\mathbb{R}^n)$ is singular, we no longer have $\det(\nabla T(x))>0$, for $\mu$-almost every $x\in\mathbb{R}^n$. Nonetheless, we proceed in the same way, noticing that all the terms in equation \eqref{m_j_estimate} are non-negative, and we drop the second one, when evaluating at $s=T$.
		\end{proof}

                \subsection{Entropic inequalities}\label{entr_ineq}

In this section, we prove that dynamical interpolations for LQ optimal control problems satisfy suitable interpolation inequalities. We introduce the class $\mathcal{DC}_N$, introduced \cite{McCann1997ACP}. We refer to \cite[Chapter 17]{Villanioldandnew} for reference and more properties.

		\begin{Def}[Displacement convexity classes]\label{def:DCN}
			For $N\in[1,\infty]$, the displacement convexity class of dimension $N$, denoted by $\mathcal{DC}_N$, is defined as the set of continuous convex functions $U:\mathbb{R}^{+}\to\mathbb{R}$, such that 
			\begin{itemize}
				\item $U(0)=0$,
				\item $U|_{(0,+\infty)}\in\mathcal{C}^2\big((0,+\infty)\big),$
				\item the function $\delta\mapsto u(\delta)$, defined as
			 	\begin{equation}
			 		u(\delta):=\begin{cases}
			 			\delta^NU(\delta^{-N}), &\text{if }N<\infty,\delta>0,\\
			 			e^{\delta}U(e^{-\delta}), &\text{if }N=\infty,\delta\in\mathbb{R},
			 			\end{cases}
			 	\end{equation}
                is convex.
			\end{itemize}
		\end{Def}
		\begin{rem}\label{nonincreasing_u}
			Since $U$ is convex and $U(0)=0$, the function $u$ in Definition \ref{def:DCN} is  non-increasing.
		\end{rem}
		\begin{exmp}
			Note that 
			\begin{itemize}
				\item for any $\alpha\geq 1$, the function $U(r)=r^{\alpha}$ belongs to all classes $\mathcal{DC}_N$.
				\item \label{ex_DC_2} if $\alpha<1$, then the function $U(r)=-r^{\alpha}$ belongs to $\mathcal{DC}_N$ if and only if $N\leq (1-\alpha)^{-1}$, i.e.\ $\alpha\geq 1-1/N$. In particular, the function $U(r)=-r^{1-1/N}$ is, in some sense, the minimal representative of $\mathcal{DC}_N$.
				\item the function $U_{\infty}(r)=r\log(r)$ belongs to $\mathcal{DC}_{\infty}$. It can be seen as the limit of the functions $U_N(r)=-N(r^{1-1/N}-r)$, which are the same (up to multiplication and addition of a linear function) as the functions appearing in item \eqref{ex_DC_2}. 
			\end{itemize}
		\end{exmp}

To each function in the $\mathcal{DC}_N$ class, we associate a corresponding functional. Such functionals are defined on a subset $\mathrm{Dom}(\cdot)\subset P_2^{ac}(\mathbb{R}^n)$, that depends on $n$ and $N$. We refer to \cite[Theorem 17.8]{Villanioldandnew} for a characterization of $\mathrm{Dom}(\cdot)$. Note that, if $n=N$ it holds $\mathrm{Dom}(\cdot)=P_2^{ac}(\mathbb{R}^n)$, for every $n\geq 3$. Moreover, compactly supported measures are always in the domain of the entropy and, if $\mu, \nu$ are compactly supported, then any displacement interpolation joining them is uniformly compactly supported (as a consequence of Proposition \ref{properties_LQ_action}).
		\begin{Def}[Entropy functional] Let $N\in[1,+\infty]$.
			For $U\in\mathcal{DC}_N$, we define the corresponding entropy functional $\mathcal{U}:\mathrm{Dom}(\mathcal{U})\subset P_2^{ac}(\mathbb{R}^n)\to\mathbb{R}$ as 
			\begin{equation}
				\mathcal{U}(\mu):=\int_{\mathbb{R}^n}U(\rho(x))d\mathscr{L}^n(x),\qquad \mu=\rho\mathscr{L}^n.
			\end{equation}
		\end{Def}
            As in the classical framework, we restrict to the case $N=n$, which is the relevant one thanks to Theorem \ref{density_interpolation_inequality}, and we prove the following distorted displacement inequality.

		\begin{thm}[Entropic inequalities]\label{dist_displ_interp}
	 	Let $(A,B,Q,T)$ be a LQ optimal control problem satisfying the Kalman condition, with $T<t^*$.
	 	Consider $U\in\mathcal{DC}_n$ and $\mu,\nu\in P_2(\mathbb{R}^n)$, with $\mu,\nu\ll \mathscr{L}^n$. Let $(\mu_\tau)$ be the unique displacement interpolation between $\mu$ and $\nu$. Assume that $(\mu_\tau)\subset\mathrm{Dom}(\mathcal{U})$ for every $\tau\in[0,T]$. Then, the entropy functional $\mathcal{U}$ satisfies
			\begin{multline}
				\mathcal{U}(\mu_\tau)\leq\left(\frac{T}{T-\tau}\right)^{n-1}\beta_{T-\tau}\int_{\mathbb{R}^n}U\bigg(\frac{\rho_0(x)}{\beta_{T-\tau}}\left(\frac{T-\tau}{T}\right)^n\bigg)d\mathscr{L}^n(x)\\
				+\left(\frac{T}{\tau}\right)^{n-1}\beta_{\tau}\int_{\mathbb{R}^n}U\bigg(\frac{\rho_T(y)}{\beta_{\tau}}\left(\frac{\tau}{T}\right)^n\bigg)d\mathscr{L}^n(y),\qquad\forall \tau\in(0,T).
			\end{multline}
            where $\beta_s=\beta_s^{(A,B,Q,T)}$ is the distortion coefficient of the LQ problem and $\mu_\tau=\rho_\tau\mathscr{L}^n$.
		\end{thm}
		\begin{proof}
			By Theorem \ref{absolute_continuity-dyn_opt_coupl}, $\mu_\tau\ll\mathscr{L}^n$, for every $\tau\in (0,T)$, with $\rho_\tau=d\mu_\tau/d\mathscr{L}^n$. Moreover, let $T_\tau$ be the optimal transport map between $\mu=\mu_0$ and $\mu_\tau$ for the cost $c^{0,\tau}$.
            
			Then $T_\tau$ is differentiable $\mu$-almost everywhere and by Lemma \ref{Jacobian_and_densities}, we have 
			\begin{equation}
				\det(\nabla T_\tau(x))=\frac{\rho_0(x)}{\rho_\tau(T_\tau(x))},\qquad\forall \tau\in[0,T],\quad\mu\text{-almost every }x\in\mathbb{R}^n.
			\end{equation}
			As a consequence, by change of variable formula applied to $z=T_\tau(x)$, we have 
			\begin{align}
				\int_{\mathbb{R}^n}U\big(\rho_\tau(z)\big)d\mathscr{L}^n(z)&=\int_{\mathbb{R}^n}U\big(\rho_\tau(T_\tau(x))\big)\det(\nabla T_\tau(x))d\mathscr{L}^n(x)\\
				&=\int_{\mathbb{R}^n}U\bigg(\frac{\rho_0(x)}{\det(\nabla T_\tau(x))}\bigg)\det(\nabla T_\tau(x))d\mathscr{L}^n(x).
			\end{align}
			Since $U(0)=0$, the contribution of $\{\rho_0=0\}$ is negligible, yielding 
			\begin{align}
				\mathcal{U}(\mu_\tau)&=\int_{\mathbb{R}^n}U\bigg(\frac{\rho_0(x)}{\det(\nabla T_\tau(x))}\bigg)\det(\nabla T_\tau(x))d\mathscr{L}^n(x)\\
				&=\int_{\{\rho_0>0\}}U\bigg(\frac{\rho_0(x)}{\det(\nabla T_\tau(x))}\bigg)\frac{\det(\nabla T_\tau(x))}{\rho_0(x)}d\mu(x)\\
				&=\int_{\{\rho_0>0\}}U\big(\delta(\tau,x)^{-n}\big)\delta(\tau,x)^{n}d\mu(x)
                =\int_{\{\rho_0>0\}}u\big(\delta(\tau,x)\big)d\mu(x),
			\end{align}
			where $u(\delta)=\delta^{n}U(\delta^{-n})$ and 
			\begin{equation}
				\delta(\tau,x)=\frac{1}{\rho_\tau(T_\tau(x))^{1/n}}=\frac{\det(\nabla T_\tau(x))^{1/n}}{\rho_0(x)^{1/n}}.
			\end{equation} 
			By Theorem \ref{density_interpolation_inequality}, we have that
			\begin{equation}
				\delta(\tau,x)\geq\frac{\beta_{T-\tau}^{1/n}}{\rho_0(x)^{1/n}} +\frac{\beta_{\tau}^{1/n}}{\rho_T(T(x))^{1/n}},
			\end{equation}
			whilst, since $U\in\mathcal{DC}_n$, the function $u$ is convex and non-increasing (see Definition \ref{def:DCN} and Remark \ref{nonincreasing_u}). Therefore, continuing the computation, we have
			\begin{align}
				\mathcal{U}(\mu_\tau)&=\int_{\{\rho_0>0\}}u\left(\delta(\tau,x)\right)d\mu(x)\\
				&\leq\int_{\{\rho_0>0\}}u\left(\frac{\beta_{T-\tau}^{1/n}}{\rho_0(x)^{1/n}} +\frac{\beta_{\tau}^{1/n}}{\rho_T(T(x))^{1/n}}\right)d\mu(x)\\
				&\leq \frac{T-\tau}{T}\int_{\{\rho_0>0\}}u\bigg(\frac{T}{T-\tau}\frac{\beta_{T-\tau}^{1/n}}{\rho_0(x)^{1/n}}\bigg)d\mu(x)
				+\frac{\tau}{T}\int_{\{\rho_0>0\}}u\left(\frac{T}{\tau}\frac{\beta_{\tau}^{1/n}}{\rho_T(T(x))^{1/n}}\right)d\mu(x)
			\end{align}
			Now, we substitute back $u(\delta)=\delta^{n}U(\delta^{-n})$ and use the relation 
			\begin{equation}
				\frac{1}{\rho_T(T(x))}d\mu(x)=\frac{\rho_0(x)}{\rho_T(T(x))}d\mathscr{L}^n(x)=\det(\nabla T(x))d\mathscr{L}^n(x),
			\end{equation}
			to get
			\begin{align}
				&\frac{T-\tau}{T}\int_{\{\rho_0>0\}}u\left(\frac{T}{T-\tau}\frac{\beta_{T-\tau}^{1/n}}{\rho_0(x)^{1/n}}\right)d\mu(x)\\
				&\qquad+\frac{\tau}{T}\int_{\{\rho_0>0\}}u\left(\frac{T}{\tau}\frac{\beta_{\tau}^{1/n}}{\rho_T(T(x))^{1/n}}\right)d\mu(x)\\
				&=\left(\frac{T}{T-\tau}\right)^{n-1}\beta_{T-\tau}\int_{\{\rho_0>0\}}U\bigg(\frac{\rho_0(x)}{\beta_{T-\tau}}
				\left(\frac{T-\tau}{T}\right)^n\bigg)\frac{1}{\rho_0(x)}d\mu(x)\\
				&\qquad+\left(\frac{T}{\tau}\right)^{n-1}\beta_{\tau}\int_{\{\rho_0>0\}}U\bigg(\frac{\rho_T(T(x))}{\beta_{\tau}}
				\left(\frac{\tau}{T}\right)^n\bigg)\frac{1}{\rho_T(T(x))}d\mu(x)\\
				&=\left(\frac{T}{T-\tau}\right)^{n-1}\beta_{T-\tau}\int_{\mathbb{R}^n}U\bigg(\frac{\rho_0(x)}{\beta_{T-\tau}}
				\left(\frac{T-\tau}{T}\right)^n\bigg)d\mathscr{L}^n(x)\\
				&\qquad+\left(\frac{T}{\tau}\right)^{n-1}\beta_{\tau}\int_{\mathbb{R}^n}U\bigg(\frac{\rho_T(y)}{\beta_{\tau}}
				\left(\frac{\tau}{T}\right)^n\bigg)d\mathscr{L}^n(y).
			\end{align}
			This completes the proof of the entropic inequality.
		\end{proof}
        
\appendix
    
    \section{Semi-concave functions}\label{app:semiconcave}
    In this section, we list some definitions and results of interest about semi-concave/semi-convex functions (with linear modulus). We refer to \cite[Chapters 2 and 3]{CannarsaSinestrari} for detailed proofs.
    
	Throughout this section, $\Omega\subset\mathbb{R}^n$ is an open set and $[x,y]$ is the closed segment between $x,y\in\mathbb{R}^n$.
	\begin{Def}[Semi-concave function]\label{def_semi_concave}
		Consider a function $u:\Omega\to\mathbb{R}$. Then $u$ is semi-concave in $\Omega$ if there exists  
            a positive constant $\kappa>0$ such that 
            \begin{equation}\label{eq_semi_concavity}
			\lambda u(x)+(1-\lambda)u(y)-u(\lambda x+(1-\lambda)y)\leq\lambda(1-\lambda)\kappa|x-y|^2,
		\end{equation}
		for every $x,y$ such that $[x,y]\subset \Omega$ and for every $\lambda\in[0,1]$. We say that $u$ is semi-convex if $-u$ is semi-concave. The optimal $\kappa$ for which equation \eqref{eq_semi_concavity} holds is called constant of semi-concavity of $u$ in $\Omega$. 
		We indicate the set of semi-concave functions over $\Omega$ with $\mathcal{SC}(\Omega)$.
	\end{Def}
    We have the following characterization (see \cite[Proposition 1.1.3]{CannarsaSinestrari}):
    \begin{prop}\label{prop_charac_semi_concave_smooth}
        Consider a function $u:\Omega\subset\mathbb{R}^n\to\mathbb{R}$. Then, the following are equivalent:
        \begin{itemize}
            \item $u\in \mathcal{SC}(\Omega)$, with semi-concavity constant $\kappa>0$,
            \item the function $x\mapsto u(x)-\frac\kappa2|x|^2$ is concave in $\Omega$,
            \item $u=\inf\limits_{i\in I}u_i$, where $u_i\in\mathcal{C}^2(\Omega)$ and $\Hess(u_i)\leq\kappa\cdot\mathbb{1}_n$ in $\Omega$, for every $i\in I$.
        \end{itemize}
    \end{prop}
    Since functions in $\mathcal{SC}(\Omega)$ are the sum of a concave function with a smooth one, their regularity (up to second order) can be deduced directly from that of concave functions. To this extent, let us introduce the following definition.
    \begin{Def}[$c-c$-hypersurface]\label{def_c_c_hyper}
    A set $E\subset\mathbb{R}^n$ is a $c-c$-hypersurface if, up to permutation of the indexes, there exist two convex functions $f,g:\mathbb{R}^{n-1}\to\mathbb{R}$ such that $E$ is the graph of $f-g$, i.e.
    \begin{equation}
        E=\{(z,t)\in\mathbb{R}^{n-1}\times\mathbb{R}\mid t=f(z)-g(z)\}.
    \end{equation}
    \end{Def}
    The following theorem characterizes the set where a convex function is not differentiable.
    \begin{thm}[Zaj\'i\v cek \cite{zajicek1979}]\label{Zajicek_thm}
        Let $\psi:\mathbb{R}^n\to\mathbb{R}$ be a convex function. Then, the set of points where $\psi$ is not differentiable is contained in a countable union of $c-c$-hypersurfaces. Conversely, if a set $E\subset\mathbb{R}^n$ can be covered by countably many $c-c$-hypersurfaces, there exists a convex function $\psi:\mathbb{R}^n\to\mathbb{R}$ that is not differentiable at any point of $E$.
    \end{thm}
    In light of Proposition \ref{prop_charac_semi_concave_smooth}, Theorem \ref{Zajicek_thm} implies that if $u\in \mathcal{SC}(\Omega)$, then $u$ is differentiable away from a set that can be covered by countably many $c-c$-hypersurfaces. Paying the price of losing such a precise description of the ``bad'' set, one can also retrieve second-order regularity properties for functions in $\mathcal{SC}(\Omega)$ (see \cite[Theorem 2.3.1]{CannarsaSinestrari}):
	\begin{thm}[Alexandroff's Theorem for semi-concave functions]\label{Alexandroff's_theorem}
		Let $u\in \mathcal{SC}(\Omega)$, then $u$ is twice differentiable $\mathscr{L}^n$-almost everywhere in $\Omega$. Explicitly, for $\mathscr{L}^n$-almost every $x\in \Omega$, there exist a vector $p_x\in\mathbb{R}^n$ and a symmetric matrix $B_x$, such that
		\begin{equation}
			\lim\limits_{h\to 0}\frac{u(x+h)-u(x)-\langle p_x,h\rangle-\frac{1}{2}\langle B_xh,h\rangle}{|h|^2}=0,
		\end{equation}
		where $\langle\cdot,\cdot\rangle$ denotes the Euclidean scalar product. 
            In addition, the gradient of $u$, defined almost everywhere in $\Omega$, belongs to the class $\mathcal{BV}_{\mathrm{loc}}(\Omega,\mathbb{R}^n)$.
	\end{thm}
	Finally, we are interested in the properties of the contact set of two functions, where the one touching from below is semi-convex and the one touching from above is semi-concave. The following result is a slight generalization of \cite[Corollary 3.3.8]{CannarsaSinestrari}, and shows that we gain even more regularity in such situation. For a proof of it, we refer to \cite[Theorem A.19]{FathiFigalli}.
	\begin{thm}\label{contact_semicon_pair}
		Let $u_1,u_2:\Omega\to\mathbb{R}$ be two functions such that $u_1\leq u_2$ on $\Omega$. Assume that $-u_1,u_2\in \mathcal{SCL}(\Omega)$, that is $u_1$ is semi-convex and $u_2$ is semi-concave, with linear moduli.
        Then, if we define the contact set of the pair as $\mathcal{E}:=\{x\in \Omega\;|\;u_1(x)=u_2(x)\}$, we have that
		\begin{itemize}
			\item both $u_1$ and $u_2$ are differentiable at each point $x\in\mathcal{E}$, with $\nabla u_1=\nabla u_2$,
			\item the function $x\mapsto \nabla u_1$ is locally Lipschitz on $\mathcal{E}$.
		\end{itemize}
	\end{thm}

\bibliographystyle{alphaabbr} 
\bibliography{Bibliography}

@book {coron2007control,
	AUTHOR = {Coron, J.-M.},
	TITLE = {Control and nonlinearity},
	SERIES = {Mathematical Surveys and Monographs},
	VOLUME = {136},
	PUBLISHER = {American Mathematical Society, Providence, RI},
	YEAR = {2007},
	PAGES = {xiv+426},
	ISBN = {978-0-8218-3668-2; 0-8218-3668-4},
	MRCLASS = {93-02 (35Q30 35Q53 35Q55 93C20)},
	MRNUMBER = {2302744},
	MRREVIEWER = {Vilmos Komornik},
	DOI = {10.1090/surv/136},
	URL = {https://doi.org/10.1090/surv/136},
}

@book {Jurdjevic1996GeometricCT,
	AUTHOR = {Jurdjevic, V.},
	TITLE = {Geometric control theory},
	SERIES = {Cambridge Studies in Advanced Mathematics},
	VOLUME = {52},
	PUBLISHER = {Cambridge University Press, Cambridge},
	YEAR = {1997},
	PAGES = {xviii+492},
	ISBN = {0-521-49502-4},
	MRCLASS = {93-02 (58E25 93B27)},
	MRNUMBER = {1425878},
	MRREVIEWER = {Heinz Sch\"{a}ttler},
}

@book {agrachev2013control,
	AUTHOR = {Agrachev, A. and Sachkov, Y.},
	TITLE = {Control theory from the geometric viewpoint},
	SERIES = {Encyclopaedia of Mathematical Sciences},
	VOLUME = {87},
	NOTE = {Control Theory and Optimization, II},
	PUBLISHER = {Springer-Verlag, Berlin},
	YEAR = {2004},
	PAGES = {xiv+412},
	ISBN = {3-540-21019-9},
	MRCLASS = {93-02 (49-02 49K15 93B27 93C15)},
	MRNUMBER = {2062547},
	MRREVIEWER = {Kevin A. Grasse},
	DOI = {10.1007/978-3-662-06404-7},
	URL = {https://doi.org/10.1007/978-3-662-06404-7},
}

@article {hinpomriff2011,
	AUTHOR = {Hindawi, A. and Pomet, J.-B. and Rifford, L.},
	TITLE = {Mass transportation with {LQ} cost functions},
	JOURNAL = {Acta Appl. Math.},
	FJOURNAL = {Acta Applicandae Mathematicae},
	VOLUME = {113},
	YEAR = {2011},
	NUMBER = {2},
	PAGES = {215--229},
	ISSN = {0167-8019},
	MRCLASS = {49Q20 (93C05)},
	MRNUMBER = {2763305},
	MRREVIEWER = {Luca Granieri},
	DOI = {10.1007/s10440-010-9595-1},
	URL = {https://doi.org/10.1007/s10440-010-9595-1},
}

@article {Agrachev_2018,
	AUTHOR = {Agrachev, A. and Barilari, D. and Rizzi, L.},
	TITLE = {Curvature: a variational approach},
	JOURNAL = {Mem. Amer. Math. Soc.},
	FJOURNAL = {Memoirs of the American Mathematical Society},
	VOLUME = {256},
	YEAR = {2018},
	NUMBER = {1225},
	PAGES = {v+142},
	ISSN = {0065-9266},
	ISBN = {978-1-4704-2946-0; 978-1-4704-4913-1},
	MRCLASS = {49J21 (53C17 58B20)},
	MRNUMBER = {3852258},
	MRREVIEWER = {Paolo Piccione},
	DOI = {10.1090/memo/1225},
	URL = {https://doi.org/10.1090/memo/1225},
}

@inbook{Villanioldandnew,
	author = {Villani, C.},
	year = {2008},
	month = jan,
	pages = {xxii+973},
	title = {{Optimal transport -- Old and new}},
	volume = {338},
	doi = {10.1007/978-3-540-71050-9},
	publisher ={Springer}
}

@book {RiffSrOptTrans,
	AUTHOR = {Rifford, L.},
	TITLE = {Sub-{R}iemannian geometry and optimal transport},
	SERIES = {SpringerBriefs in Mathematics},
	PUBLISHER = {Springer, Cham},
	YEAR = {2014},
	PAGES = {viii+140},
	ISBN = {978-3-319-04803-1; 978-3-319-04804-8},
	MRCLASS = {49Q20 (53C17)},
	MRNUMBER = {3308395},
	MRREVIEWER = {Yanir A. Rubinstein},
	DOI = {10.1007/978-3-319-04804-8},
	URL = {https://doi.org/10.1007/978-3-319-04804-8},
}

@incollection {Usersguide,
	AUTHOR = {Ambrosio, L. and Gigli, N.},
	TITLE = {A user's guide to optimal transport},
	BOOKTITLE = {Modelling and optimisation of flows on networks},
	SERIES = {Lecture Notes in Math.},
	VOLUME = {2062},
	PAGES = {1--155},
	PUBLISHER = {Springer, Heidelberg},
	YEAR = {2013},
	MRCLASS = {49Q20 (35R70 49-01)},
	MRNUMBER = {3050280},
	MRREVIEWER = {Luca Granieri},
	DOI = {10.1007/978-3-642-32160-3\_1},
	URL = {https://doi.org/10.1007/978-3-642-32160-3_1},
}

@book {CannarsaSinestrari,
	AUTHOR = {Cannarsa, P. and Sinestrari, C.},
	TITLE = {Semiconcave functions, {H}amilton-{J}acobi equations, and
	optimal control},
	SERIES = {Progress in Nonlinear Differential Equations and their
	Applications},
	VOLUME = {58},
	PUBLISHER = {Birkh\"{a}user Boston, Inc., Boston, MA},
	YEAR = {2004},
	PAGES = {xiv+304},
	ISBN = {0-8176-4084-3},
	MRCLASS = {49-02 (35F20 49K20 49L20)},
	MRNUMBER = {2041617},
	MRREVIEWER = {Pierre Cardaliaguet},
}

@article {FathiFigalli,
	AUTHOR = {Fathi, A. and Figalli, A.},
	TITLE = {Optimal transportation on non-compact manifolds},
	JOURNAL = {Israel J. Math.},
	FJOURNAL = {Israel Journal of Mathematics},
	VOLUME = {175},
	YEAR = {2010},
	PAGES = {1--59},
	ISSN = {0021-2172},
	MRCLASS = {49Q20 (37J50)},
	MRNUMBER = {2607536},
	MRREVIEWER = {Luca Granieri},
	DOI = {10.1007/s11856-010-0001-5},
	URL = {https://doi.org/10.1007/s11856-010-0001-5},
}

@article {SRintineq,
	AUTHOR = {Barilari, D. and Rizzi, L.},
	TITLE = {Sub-{R}iemannian interpolation inequalities},
	JOURNAL = {Invent. Math.},
	FJOURNAL = {Inventiones Mathematicae},
	VOLUME = {215},
	YEAR = {2019},
	NUMBER = {3},
	PAGES = {977--1038},
	ISSN = {0020-9910},
	MRCLASS = {53C17 (49J15 49Q20)},
	MRNUMBER = {3935035},
	MRREVIEWER = {Emmanuel Tr\'{e}lat},
	DOI = {10.1007/s00222-018-0840-y},
	URL = {https://doi.org/10.1007/s00222-018-0840-y},
}

@article {BakriEmeryandmodelspaces,
	AUTHOR = {Barilari, D. and Rizzi, L.},
	TITLE = {Bakry-\'{E}mery curvature and model spaces in sub-{R}iemannian
	geometry},
	JOURNAL = {Math. Ann.},
	FJOURNAL = {Mathematische Annalen},
	VOLUME = {377},
	YEAR = {2020},
	NUMBER = {1-2},
	PAGES = {435--482},
	ISSN = {0025-5831},
	MRCLASS = {53C17 (49J15 49N10)},
	MRNUMBER = {4099638},
	MRREVIEWER = {Jing Wang},
	DOI = {10.1007/s00208-020-01982-x},
	URL = {https://doi.org/10.1007/s00208-020-01982-x},
}

@article {UnifiedsyntheticRicci,
    AUTHOR = {Barilari, Davide and Mondino, Andrea and Rizzi, Luca},
     TITLE = {Unified synthetic {R}icci curvature lower bounds for
              {R}iemannian and sub-{R}iemannian structures},
   JOURNAL = {Mem. Amer. Math. Soc.},
  FJOURNAL = {Memoirs of the American Mathematical Society},
    VOLUME = {317},
      YEAR = {2026},
    NUMBER = {1613},
     PAGES = {viii+145},
      ISSN = {0065-9266,1947-6221},
      ISBN = {978-1-4704-7806-3; 978-1-4704-8585-6},
   MRCLASS = {53C21 (51F30 53C17 53C23)},
  MRNUMBER = {5038044},
       DOI = {10.1090/memo/1613},
       URL = {https://doi.org/10.1090/memo/1613},
}

@article{rotem,
  title={Curvature-Dimension for Autonomous Lagrangians},
  author={Assouline, Rotem},
  journal={Geometric Aspects of Functional Analysis: Israel Seminar (GAFA) - to appear},
  year={2024}
}

@article {Barilari_2016,
	AUTHOR = {Barilari, D. and Rizzi, L.},
	TITLE = {Comparison theorems for conjugate points in sub-{R}iemannian
	geometry},
	JOURNAL = {ESAIM Control Optim. Calc. Var.},
	FJOURNAL = {ESAIM. Control, Optimisation and Calculus of Variations},
	VOLUME = {22},
	YEAR = {2016},
	NUMBER = {2},
	PAGES = {439--472},
	ISSN = {1292-8119},
	MRCLASS = {53C17 (49N10 53C21 53C22)},
	MRNUMBER = {3491778},
	MRREVIEWER = {Davide Vittone},
	DOI = {10.1051/cocv/2015013},
	URL = {https://doi.org/10.1051/cocv/2015013},
}

@book {marcus1992survey,
	AUTHOR = {Marcus, M. and Minc, H.},
	TITLE = {A survey of matrix theory and matrix inequalities},
	NOTE = {Reprint of the 1969 edition},
	PUBLISHER = {Dover Publications, Inc., New York},
	YEAR = {1992},
	PAGES = {xii+180},
	ISBN = {0-486-67102-X},
	MRCLASS = {15-01},
	MRNUMBER = {1215484},
}

@article {figalli2009masstransSRmflds,
	AUTHOR = {Figalli, A. and Rifford, L.},
	TITLE = {Mass transportation on sub-{R}iemannian manifolds},
	JOURNAL = {Geom. Funct. Anal.},
	FJOURNAL = {Geometric and Functional Analysis},
	VOLUME = {20},
	YEAR = {2010},
	NUMBER = {1},
	PAGES = {124--159},
	ISSN = {1016-443X},
	MRCLASS = {49Q20 (53C17)},
	MRNUMBER = {2647137},
	MRREVIEWER = {Luca Granieri},
	DOI = {10.1007/s00039-010-0053-z},
	URL = {https://doi.org/10.1007/s00039-010-0053-z},
}

@article {SturmGeomMmsII,
	AUTHOR = {Sturm, K.-T.},
	TITLE = {On the geometry of metric measure spaces. {II}},
	JOURNAL = {Acta Math.},
	FJOURNAL = {Acta Mathematica},
	VOLUME = {196},
	YEAR = {2006},
	NUMBER = {1},
	PAGES = {133--177},
	ISSN = {0001-5962},
	MRCLASS = {53C23},
	MRNUMBER = {2237207},
	MRREVIEWER = {Juha Heinonen},
	DOI = {10.1007/s11511-006-0003-7},
	URL = {https://doi.org/10.1007/s11511-006-0003-7},
}

@article {SturmGeomMmsI,
	AUTHOR = {Sturm, Karl-Theodor},
	TITLE = {On the geometry of metric measure spaces. {I}},
	JOURNAL = {Acta Math.},
	FJOURNAL = {Acta Mathematica},
	VOLUME = {196},
	YEAR = {2006},
	NUMBER = {1},
	PAGES = {65--131},
	ISSN = {0001-5962},
	MRCLASS = {53C23},
	MRNUMBER = {2237206},
	MRREVIEWER = {Juha Heinonen},
	DOI = {10.1007/s11511-006-0002-8},
	URL = {https://doi.org/10.1007/s11511-006-0002-8},
}

@article {lott2006riccicurvmmspaces,
	AUTHOR = {Lott, J. and Villani, C.},
	TITLE = {Ricci curvature for metric-measure spaces via optimal
	transport},
	JOURNAL = {Ann. of Math. (2)},
	FJOURNAL = {Annals of Mathematics. Second Series},
	VOLUME = {169},
	YEAR = {2009},
	NUMBER = {3},
	PAGES = {903--991},
	ISSN = {0003-486X},
	MRCLASS = {53C23 (49Q15)},
	MRNUMBER = {2480619},
	MRREVIEWER = {Alessio Figalli},
	DOI = {10.4007/annals.2009.169.903},
	URL = {https://doi.org/10.4007/annals.2009.169.903},
}

@article {CaffarelliRegConvMaps,
	AUTHOR = {Caffarelli, L.-A.},
	TITLE = {The regularity of mappings with a convex potential},
	JOURNAL = {J. Amer. Math. Soc.},
	FJOURNAL = {Journal of the American Mathematical Society},
	VOLUME = {5},
	YEAR = {1992},
	NUMBER = {1},
	PAGES = {99--104},
	ISSN = {0894-0347},
	MRCLASS = {35B65 (35A30 35J60)},
	MRNUMBER = {1124980},
	DOI = {10.2307/2152752},
	URL = {https://doi.org/10.2307/2152752},
}

@article {agrachev2007opttransnonholconst,
	AUTHOR = {Agrachev, A. and Lee, P.},
	TITLE = {Optimal transportation under nonholonomic constraints},
	JOURNAL = {Trans. Amer. Math. Soc.},
	FJOURNAL = {Transactions of the American Mathematical Society},
	VOLUME = {361},
	YEAR = {2009},
	NUMBER = {11},
	PAGES = {6019--6047},
	ISSN = {0002-9947},
	MRCLASS = {49J15 (37J60 49J30 53C17)},
	MRNUMBER = {2529923},
	MRREVIEWER = {Filippo Santambrogio},
	DOI = {10.1090/S0002-9947-09-04813-2},
	URL = {https://doi.org/10.1090/S0002-9947-09-04813-2},
}

@article {ARiemintineq,
	AUTHOR = {Cordero-Erausquin, D. and McCann, R.-J. and
	Schmuckenschl\"{a}ger, M.},
	TITLE = {A {R}iemannian interpolation inequality \`a la {B}orell,
	{B}rascamp and {L}ieb},
	JOURNAL = {Invent. Math.},
	FJOURNAL = {Inventiones Mathematicae},
	VOLUME = {146},
	YEAR = {2001},
	NUMBER = {2},
	PAGES = {219--257},
	ISSN = {0020-9910},
	MRCLASS = {58E35 (28C99 60E15)},
	MRNUMBER = {1865396},
	MRREVIEWER = {C\'{e}dric Villani},
	DOI = {10.1007/s002220100160},
	URL = {https://doi.org/10.1007/s002220100160},
}

@article {Brenier1991,
	AUTHOR = {Brenier, Y.},
	TITLE = {Polar factorization and monotone rearrangement of
	vector-valued functions},
	JOURNAL = {Comm. Pure Appl. Math.},
	FJOURNAL = {Communications on Pure and Applied Mathematics},
	VOLUME = {44},
	YEAR = {1991},
	NUMBER = {4},
	PAGES = {375--417},
	ISSN = {0010-3640},
	MRCLASS = {46E40 (35Q99 46E99 49Q99)},
	MRNUMBER = {1100809},
	MRREVIEWER = {Robert McOwen},
	DOI = {10.1002/cpa.3160440402},
	URL = {https://doi.org/10.1002/cpa.3160440402},
}

@article {Agrachev_2014,
	AUTHOR = {Agrachev, A. and Rizzi, L. and Silveira, P.},
	TITLE = {On conjugate times of {LQ} optimal control problems},
	JOURNAL = {J. Dyn. Control Syst.},
	FJOURNAL = {Journal of Dynamical and Control Systems},
	VOLUME = {21},
	YEAR = {2015},
	NUMBER = {4},
	PAGES = {625--641},
	ISSN = {1079-2724},
	MRCLASS = {70G45 (49N10 53D12)},
	MRNUMBER = {3394719},
	MRREVIEWER = {C\'{e}sar Rodrigo},
	DOI = {10.1007/s10883-014-9251-6},
	URL = {https://doi.org/10.1007/s10883-014-9251-6},
}

@book{monge1781,
	title={{M{\'e}moire sur la th{\'e}orie des d{\'e}blais et des remblais}},
	author={Monge, G.},
	url={https://books.google.it/books?id=IG7CGwAACAAJ},
	year={1781},
	publisher={Imprimerie royale}
}

@article{Kantorovich2006,
	title={{On the Traslocation of Masses}},
	journal={Journal of Mathematical Sciences},
	pages={1381-1382},
	author={Kantorovich, L.-V.},
	volume={133},
	year={2006},
	note={Translation of the original work on Dokl. Akad. Nauk SSSR, 37, No. 7-8, 229-229 of 1942}

}

@article {GangboMcCann1996,
	AUTHOR = {Gangbo, W. and McCann, R.-J.},
	TITLE = {The geometry of optimal transportation},
	JOURNAL = {Acta Math.},
	FJOURNAL = {Acta Mathematica},
	VOLUME = {177},
	YEAR = {1996},
	NUMBER = {2},
	PAGES = {113--161},
	ISSN = {0001-5962},
	MRCLASS = {49Q20 (28A99 35Q99 60B99 65K10)},
	MRNUMBER = {1440931},
	MRREVIEWER = {Ludger R\"{u}schendorf},
	DOI = {10.1007/BF02392620},
	URL = {https://doi.org/10.1007/BF02392620},
}

@article {GigliInverseBrenier2009,
	AUTHOR = {Gigli, N.},
	TITLE = {On the inverse implication of {B}renier-{M}c{C}ann theorems
	and the structure of {$(P_2(M),W_2)$}},
	JOURNAL = {Methods Appl. Anal.},
	FJOURNAL = {Methods and Applications of Analysis},
	VOLUME = {18},
	YEAR = {2011},
	NUMBER = {2},
	PAGES = {127--158},
	ISSN = {1073-2772},
	MRCLASS = {49Q20 (28A33)},
	MRNUMBER = {2847481},
	MRREVIEWER = {Luca Granieri},
	DOI = {10.4310/MAA.2011.v18.n2.a1},
	URL = {https://doi.org/10.4310/MAA.2011.v18.n2.a1},
}

@article {McCann2001,
	AUTHOR = {McCann, R.-J.},
	TITLE = {Polar factorization of maps on {R}iemannian manifolds},
	JOURNAL = {Geom. Funct. Anal.},
	FJOURNAL = {Geometric and Functional Analysis},
	VOLUME = {11},
	YEAR = {2001},
	NUMBER = {3},
	PAGES = {589--608},
	ISSN = {1016-443X},
	MRCLASS = {58E15 (46N10 49Q20 53C20)},
	MRNUMBER = {1844080},
	MRREVIEWER = {Lucio Renato Berrone},
	DOI = {10.1007/PL00001679},
	URL = {https://doi.org/10.1007/PL00001679},
}

@book {ambrosioGradientFlowsMetric2008,
	AUTHOR = {Ambrosio, L. and Gigli, N. and Savar\'{e}, G.},
	TITLE = {Gradient flows in metric spaces and in the space of
	probability measures},
	SERIES = {Lectures in Mathematics ETH Z\"{u}rich},
	EDITION = {Second},
	PUBLISHER = {Birkh\"{a}user Verlag, Basel},
	YEAR = {2008},
	PAGES = {x+334},
	ISBN = {978-3-7643-8721-1},
	MRCLASS = {49-02 (28A33 35K55 35K90 49Q20 60B05)},
	MRNUMBER = {2401600},
	MRREVIEWER = {Pietro Celada},
}

@article {AMBROSIO2004261,
	AUTHOR = {Ambrosio, L. and Rigot, S.},
	TITLE = {Optimal mass transportation in the {H}eisenberg group},
	JOURNAL = {J. Funct. Anal.},
	FJOURNAL = {Journal of Functional Analysis},
	VOLUME = {208},
	YEAR = {2004},
	NUMBER = {2},
	PAGES = {261--301},
	ISSN = {0022-1236},
	MRCLASS = {49Q20 (35H10 43A80 53C17)},
	MRNUMBER = {2035027},
	MRREVIEWER = {Enrico Valdinoci},
	DOI = {10.1016/S0022-1236(03)00019-3},
	URL = {https://doi.org/10.1016/S0022-1236(03)00019-3},
}

@article {bernard2007optimalmasstransportationmather,
	AUTHOR = {Bernard, P. and Buffoni, B.},
	TITLE = {Optimal mass transportation and {M}ather theory},
	JOURNAL = {J. Eur. Math. Soc. (JEMS)},
	FJOURNAL = {Journal of the European Mathematical Society (JEMS)},
	VOLUME = {9},
	YEAR = {2007},
	NUMBER = {1},
	PAGES = {85--121},
	ISSN = {1435-9855},
	MRCLASS = {49Q20},
	MRNUMBER = {2283105},
	MRREVIEWER = {Filippo Santambrogio},
	DOI = {10.4171/JEMS/74},
	URL = {https://doi.org/10.4171/JEMS/74},
}

@article {balogh2018geometricinequalitiesheisenberggroups,
	AUTHOR = {Balogh, Z.-M. and Krist\'{a}ly, A. and Sipos, K.},
	TITLE = {Geometric inequalities on {H}eisenberg groups},
	JOURNAL = {Calc. Var. Partial Differential Equations},
	FJOURNAL = {Calculus of Variations and Partial Differential Equations},
	VOLUME = {57},
	YEAR = {2018},
	NUMBER = {2},
	PAGES = {Paper No. 61, 41},
	ISSN = {0944-2669},
	MRCLASS = {49Q20 (53C17)},
	MRNUMBER = {3774461},
	MRREVIEWER = {Alireza Ranjbar-Motlagh},
	DOI = {10.1007/s00526-018-1320-3},
	URL = {https://doi.org/10.1007/s00526-018-1320-3},
}

@article {balogh2019jacobiandeterminantinequalitycorank,
	AUTHOR = {Balogh, Z.-M. and Krist\'{a}ly, A. and Sipos, K.},
	TITLE = {Jacobian determinant inequality on corank 1 {C}arnot groups
	with applications},
	JOURNAL = {J. Funct. Anal.},
	FJOURNAL = {Journal of Functional Analysis},
	VOLUME = {277},
	YEAR = {2019},
	NUMBER = {12},
	PAGES = {108293, 36},
	ISSN = {0022-1236},
	MRCLASS = {53C17 (35R03 49Q20)},
	MRNUMBER = {4019096},
	MRREVIEWER = {Pei Biao Zhao},
	DOI = {10.1016/j.jfa.2019.108293},
	URL = {https://doi.org/10.1016/j.jfa.2019.108293},
}

@article {Cavalletti_2017,
	AUTHOR = {Cavalletti, F. and Mondino, A.},
	TITLE = {Optimal maps in essentially non-branching spaces},
	JOURNAL = {Commun. Contemp. Math.},
	FJOURNAL = {Communications in Contemporary Mathematics},
	VOLUME = {19},
	YEAR = {2017},
	NUMBER = {6},
	PAGES = {175-203, 27},
	ISSN = {0219-1997},
	MRCLASS = {49Q20 (53C23)},
	MRNUMBER = {3691502},
	MRREVIEWER = {Alexander O. Ivanov},
	DOI = {10.1142/S0219199717500079},
	URL = {https://doi.org/10.1142/S0219199717500079},
}

@article {cavalletti2014existenceuniquenessoptimaltransport,
	AUTHOR = {Cavalletti, F. and Huesmann, M.},
	TITLE = {Existence and uniqueness of optimal transport maps},
	JOURNAL = {Ann. Inst. H. Poincar\'{e} C Anal. Non Lin\'{e}aire},
	FJOURNAL = {Annales de l'Institut Henri Poincar\'{e} C. Analyse Non Lin\'{e}aire},
	VOLUME = {32},
	YEAR = {2015},
	NUMBER = {6},
	PAGES = {1367--1377},
	ISSN = {0294-1449},
	MRCLASS = {49Q20},
	MRNUMBER = {3425266},
	MRREVIEWER = {Luca Granieri},
	DOI = {10.1016/j.anihpc.2014.09.006},
	URL = {https://doi.org/10.1016/j.anihpc.2014.09.006},
}

@article {McCann1997ACP,
	AUTHOR = {McCann, R.-J.},
	TITLE = {A convexity principle for interacting gases},
	JOURNAL = {Adv. Math.},
	FJOURNAL = {Advances in Mathematics},
	VOLUME = {128},
	YEAR = {1997},
	NUMBER = {1},
	PAGES = {153--179},
	ISSN = {0001-8708},
	MRCLASS = {82B05 (26B25 90C08)},
	MRNUMBER = {1451422},
	MRREVIEWER = {Carlos Matr\'{a}n},
	DOI = {10.1006/aima.1997.1634},
	URL = {https://doi.org/10.1006/aima.1997.1634},
}

@article {FiJu2008,
	AUTHOR = {Figalli, A. and Juillet, N.},
	TITLE = {Absolute continuity of {W}asserstein geodesics in the
	{H}eisenberg group},
	JOURNAL = {J. Funct. Anal.},
	FJOURNAL = {Journal of Functional Analysis},
	VOLUME = {255},
	YEAR = {2008},
	NUMBER = {1},
	PAGES = {133--141},
	ISSN = {0022-1236},
	MRCLASS = {53C17 (22E25 49Q15)},
	MRNUMBER = {2417812},
	MRREVIEWER = {C\'{e}dric Villani},
	DOI = {10.1016/j.jfa.2008.03.006},
	URL = {https://doi.org/10.1016/j.jfa.2008.03.006},
}

@article {zajicek1979,
    AUTHOR = {Zaj{\'i}{\v c}ek, L.},
     TITLE = {On the differentiation of convex functions in finite and
              infinite dimensional spaces},
   JOURNAL = {Czechoslovak Math. J.},
  FJOURNAL = {Czechoslovak Mathematical Journal},
    VOLUME = {29(104)},
      YEAR = {1979},
    NUMBER = {3},
     PAGES = {340--348},
      ISSN = {0011-4642},
   MRCLASS = {46G05 (26A27 52A05)},
  MRNUMBER = {536060},
MRREVIEWER = {Jean-Paul\ Penot},
}
\end{document}